\documentclass[11pt]{amsart}
\usepackage{format}
\usepackage{mathtools}

\newcommand\CG{{\mathcal{G}}}
\newcommand\CA{{\mathcal{A}}}

\newcommand\CF{{\mathcal{F}}}

\newcommand\CS{{\mathcal{S}}}

\newcommand\ft{{\mathfrak{t}}}

\newcommand \tr{{\mathrm tr}}

\begin{document}

\title{The Unitarity of Arthur Packets for Real Reductive Groups}

\author{Jeffrey Adams}
\address{Jeffrey Adams\\University of Maryland}
\email{jda@umd.edu}

\author{Andrei Ionov}
\address{Andrei Ionov\\University of Texas at Austin}
\email{andrei.ionov@austin.utexas.edu}

\author{Lucas Mason-Brown}
\address{Lucas Mason-Brown\\Columbia University}
\email{lm4157@columbia.edu}

\author{David Vogan}
\address{David Vogan\\Massachusetts Institute of Technology}
\email{dav@math.mit.edu}


\begin{abstract}
Let $G$ be a connected reductive algebraic group defined over $\RR$. In \cite{Arthur1983,Arthur1989}, Arthur conjectured the existence of certain packets of irreducible admissible representations of $G(\RR)$ satisfying various remarkable properties. These packets were given a precise definition in \cite{AdamsBarbaschVogan} in terms of microlocal geometry on a space of Langlands parameters. A longstanding conjecture, originating in \cite{Arthur1983}, is that all Arthur packets consist of \emph{unitary} representations. In this paper, we prove this conjecture in general. The main new idea is a `Jordan decomposition' for Arthur packets: a canonical two-step process for realizing an arbitrary Arthur packet via real parabolic and cohomological induction from a unipotent Arthur packet for a certain Levi subgroup. This process is analogous to the decomposition of an element of a complex algebraic group as a (unique) commuting product of elliptic, hyperbolic, and unipotent parts. Using this decomposition, we reduce the question of unitarity to the case of unipotent Arthur packets, where the answer is already known (by work of Adams-Arancibia-Mezo, Adams-van Leeuwen-Miller-Vogan, Arthur, Barbasch, Barbasch-Ma-Sun-Zhu, and Davis-Mason-Brown). As an application of the same methods, we also give a proof of Jiang's conjecture for real reductive groups, which gives an upper bound on the associated varieties of the members of an Arthur packet in terms of the Barbasch-Vogan dual of the Arthur $SL_2(\CC)$. 
\end{abstract}

\maketitle

\tableofcontents

\section{Introduction}

Let $G(\RR)$ be the real points of a connected reductive algebraic group defined over $\RR$ and let ${}^{\vee}G^{\Gamma}$ be the associated $L$-group. The Langlands classification gives a parameterization of the set $\Pi(G)$ of irreducible admissible representations of inner twists of $G(\RR)$ in terms of conjugacy classes in ${}^{\vee}G^{\Gamma}$.\footnote{Much of this discussion applies, with some modification, to connected reductive groups over arbitrary local fields. However, we will restrict our attention in this paper to the real case.} In more detail, a \emph{Langlands parameter} for ${}^{\vee}G^{\Gamma}$ is a continuous homomorphism
$$\varphi \colon W_{\RR} \to {}^{\vee}G^{\Gamma}$$
from the Weil group $W_{\RR}$ of $\RR$ into ${}^{\vee}G^{\Gamma}$ such that
\begin{itemize}
    \item the image of $\varphi$ consists of semisimple elements;
    \item $\varphi$ is compatible with the projections $W_{\RR} \to \Gamma$ and ${}^{\vee}G^{\Gamma} \to \Gamma$.
\end{itemize}
We write $\Phi({}^{\vee}G^{\Gamma})$ for the set of ${}^{\vee}G$-conjugacy classes of Langlands parameters for ${}^{\vee}G^{\Gamma}$. Here is a simple formulation of the Langlands classification.

\begin{theorem}[\cite{Langlands1979},\cite{AdamsBarbaschVogan}]
For each Langlands parameter $\varphi \in \Phi({}^{\vee}G^{\Gamma})$, there is an associated set $\Pi(G)_{\varphi} \subset \Pi(G)$, called an `$L$-packet'. As $\varphi$ ranges over $\Phi({}^{\vee}G^{\Gamma})$, the $L$-packets partition $\Pi(G)$.
\end{theorem}

A more precise version of this result, including a definition of $L$-packets and a parameterization of their members will be given in Section \ref{sec:LLCABV}.

In \cite{Arthur1983}, Arthur introduced another set of parameters for irreducible representations. An \emph{Arthur parameter} for ${}^{\vee}G^{\Gamma}$ is a continuous homomorphism
$$\psi\colon W_{\RR} \times SL_2(\CC) \to {}^{\vee}G^{\Gamma}$$
such that 
\begin{itemize}
    \item  $\psi|_{W_{\RR}}$ is a Langlands parameter for ${}^{\vee}G^{\Gamma}$;
    \item $\psi(W_{\RR})$ is bounded;
    \item $\psi|_{SL_2(\CC)}$ is algebraic.
\end{itemize}
We write $\Psi({}^{\vee}G^{\Gamma})$ for the set of ${}^{\vee}G$-conjugacy classes of Arthur parameters for ${}^{\vee}G^{\Gamma}$.

To each Arthur parameter $\psi$, one can associate a Langlands parameter $\varphi_{\psi}$ according to the formula
$$\varphi_{\psi}\colon W_{\RR} \to {}^{\vee}G^{\Gamma}, \qquad \varphi_{\psi}(w) = \psi(w, \begin{pmatrix} ||w||^{1/2} & 0\\ 0 & ||w||^{-1/2} \end{pmatrix}). 
$$
This defines an injection
$$\Psi({}^{\vee}G^{\Gamma}) \hookrightarrow \Phi({}^{\vee}G^{\Gamma}).$$ 
In \cite{Arthur1983,Arthur1989}, Arthur conjectured that $\Psi({}^{\vee}G^{\Gamma})$ should parameterize a large (but incomplete) portion of the irreducible \emph{unitary} representations of inner twists of $G(\RR)$ (and \emph{all} of the unitary representations of interest for global applications). A bit more precisely

\begin{conj}[Arthur's Conjectures for Real Reductive Groups, \cite{Arthur1983,Arthur1989}, see also {\cite[Chapter 1, Problems A-F]{AdamsBarbaschVogan}}]\label{conj:Arthur}
For each Arthur parameter $\psi \in \Psi({}^{\vee}G^{\Gamma})$, there is
\begin{itemize}
    \item a set $\Pi(G)_{\psi} \subset \Pi(G)$, and
    \item a function
    $$\chi_{\psi}: \Pi(G)_{\psi} \to \{\text{nonzero finite-dimensional representations of } A_{\psi} := \mathrm{Com}(\mathrm{Stab}_{{}^{\vee}G}(\psi))\}$$
    (`Com' denotes the component group)
\end{itemize}
such that the following are true:
\begin{itemize}
    \item[(i)] There is an inclusion
    $$\Pi(G)_{\varphi_{\psi}} \subseteq \Pi(G)_{\psi}.$$
    \item[(ii)] The virtual representation
    \begin{equation}\label{eq:Arthurstable}\eta_{\psi}(1) = \sum_{\pi \in \Pi(G)_{\psi}} \epsilon(\pi)\dim(\chi_{\psi}(\pi))\pi\end{equation}
    is stable in the sense of Langlands and Shelstad (\cite{LanglandsShelstad}) (here, $\epsilon(\pi)$ is a sign which is also to be defined). Note that we can also define virtual representations
\begin{equation}\label{eq:Arthurstables}\eta_{\psi}(s) = \sum_{\pi \in \Pi(G)_{\psi}} \epsilon(\pi)\mathrm{Tr}(\chi_{\psi}(\pi)(s))\pi\end{equation}
    for each $s \in A_{\psi}$ (although these are not typically stable for $s\neq 1$).
    \item[(iii)] The members of $\Pi(G)_{\psi}$ are unitary.
\end{itemize}
\end{conj}

In \cite{AdamsBarbaschVogan}, Adams, Barbasch, and Vogan gave a construction of the sets $\Pi(G)_{\psi}$ and functions $\chi_{\psi}$ in terms of microlocal geometry on a space of Langlands parameters (in Sections \ref{sec:ABV} and \ref{sec:functoriality}, we will denote them by $\Pi(G)_{\psi}^{mic}$ and $\chi_{\psi}^{mic}$, respectively, to emphasize their connections to microlocal geometry and to make a clearer distinction with the notation for $L$-packets). They proved in \cite{AdamsBarbaschVogan} that their definitions satisfy most of Arthur's predictions, including (i) and (ii) of Conjecture \ref{conj:Arthur} above. However, they were not able to prove that the members of $\Pi(G)_{\psi}$ are unitary. In this paper, we give a general proof of this fact, thus completing the proof of Arthur's conjectures in the Archimedean case.

\begin{theorem}[See Corollary \ref{cor:unitary} below]\label{thm:unitaryintro}
All Arthur packets consist of unitary representations.
\end{theorem}

We will now briefly summarize the structure of our proof. There is an important class of basic Arthur parameters, called `unipotent' by Arthur.

\begin{definition}[\cite{Arthur1989}, see also {\cite[Definition 27.1]{AdamsBarbaschVogan}}]
An Arthur parameter $\psi$ is said to be \emph{unipotent} if 
$$\psi(\CC^{\times}) \subset Z({}^{\vee}G).$$
\end{definition}

Thanks to the work of many mathematicians over several decades, Theorem \ref{thm:unitaryintro} is now known to be true in the unipotent case. 

\begin{theorem}[Adams-Arancibia-Mezo (\cite{AARM}), Adams-van Leeuwen-Miller-Vogan (\cite{AdamsMillerVogan}), Arthur (\cite{Arthur2013}), Barbasch (\cite{Barbasch1989}), Barbasch-Ma-Sun-Zhu (\cite{BarbaschMaSunZhu}), Davis-Mason-Brown (\cite{DavisMasonBrownOrbitMethod})]\label{thm:unipotentunitary}
For each unipotent Arthur parameter $\psi$, the Arthur packet $\Pi(G)_{\psi}$ consists of unitary representations. 
\end{theorem}

\begin{proof}
It is sufficient to prove the theorem under the following assumptions: $G$ is almost simple  and $\psi|_{\CC^{\times}} = 1$. A representation belongs to some such $\Pi(G)_{\psi}$ if and only if it is \emph{special unipotent} (cf. \cite[Corollary 27.13]{AdamsBarbaschVogan}). Thus, it is sufficient to prove that all special unipotent representations of almost simple groups are unitary. This is proved
\begin{itemize}
    \item for real exceptional groups in \cite{AdamsMillerVogan};
    \item for quasisplit symplectic and special orthogonal groups in \cite{Arthur2013,AARM};
    \item for complex classical groups in \cite{Barbasch1989};
    \item for all classical groups in \cite{BarbaschMaSunZhu};
    \item for all complex groups (and for arbitrary groups, provided the Barbasch-Vogan dual of the Arthur $SL_2(\CC)$ is a \emph{rigid} nilpotent orbit) in \cite{DavisMasonBrownOrbitMethod}.
\end{itemize}
\end{proof}

Our strategy for proving Theorem \ref{thm:unitaryintro} in general is to show that every Arthur packet can be obtained in an essentially canonical fashion from a unipotent Arthur packet for a suitable Levi subgroup by means of several unitarity-preserving constructions, namely \emph{real parabolic} and \emph{cohomological induction}. The recipe for this induction is in some sense encoded in the structure of the Arthur parameter. Given an Arthur parameter $\psi: W_{\RR} \times SL_2(\CC) \to {}^{\vee}G^{\Gamma}$, the image under $\psi$ of $\RR_+ \subset W_{\RR}$ defines a Levi subgroup of ${}^{\vee}G$
\begin{equation}\label{eq:M1intro}{}^{\vee}M_1 := Z_{{}^{\vee}G}(\psi(\RR_+)).\end{equation}
Since $\RR_+$ commutes with $S^1 \subset W_{\RR}$, the image of $S^1$ belongs to ${}^{\vee}M_1$. The co-character $\psi|_{S^1}\colon S^1 \to {}^{\vee}M_1$ defines a $\ZZ$-grading of ${}^{\vee}\fm_1$
$${}^{\vee}\fm_1 := \bigoplus {}^{\vee}\fm_{1,n}, \qquad {}^{\vee}\fm_{1,n} := \{X \in {}^{\vee}\fm_1 \mid \Ad(\psi(z))X = z^nX, \quad \forall z \in S^1\},$$
and hence a parabolic subalgebra
\begin{equation}\label{eq:M2intro}{}^{\vee}\fq_2: = {}^{\vee}\fm_2 \oplus {}^{\vee}\fu_2, \qquad {}^{\vee}\fm_2 := {}^{\vee}\fm_{1,0} = Z_{{}^{\vee}\fm_1}(\psi(S^1)), \qquad {}^{\vee}\fu_2 := \bigoplus_{n >0} {}^{\vee}\fm_{1,n}.\end{equation}
Let ${}^{\vee}M_2$ (resp. ${}^{\vee}Q_2$) denote the connected subgroup of ${}^{\vee}M_1$ corresponding to ${}^{\vee}\fm_2$ (resp. ${}^{\vee}\fq_2$). Since $j$ normalizes $\CC^{\times} \subset W_{\RR}$, $\psi(j)$ normalizes ${}^{\vee}M_2$. Thus, the subgroup ${}^{\vee}M_2^{\Gamma} \subset {}^{\vee}G^{\Gamma}$ generated by ${}^{\vee}M_2$ and $\psi(j)$ is a $\ZZ/2\ZZ$-extension of ${}^{\vee}M_2$ (a `weak $E$-group' in the language of \cite{AdamsBarbaschVogan}). By construction, this subgroup contains the image of $\psi$. The resulting Arthur parameter
$$\psi_2\colon W_{\RR} \times SL_2(\CC) \to {}^{\vee}M_2^{\Gamma}$$
is unipotent. 

The subgroups ${}^{\vee}M_2^{\Gamma} \subset {}^{\vee}M_1^{\Gamma} \subset {}^{\vee}G^{\Gamma}$ determine subgroups $M_2 \subset M_1 \subset G$ as well as inner classes of real forms of $M_2$ and $M_1$. Moreover, the parabolic subgroup ${}^{\vee}Q_2 \subset {}^{\vee}M_1$ determines a parabolic subgroup $Q_2 \subset M_1$ with Levi factor $M_2$. By construction, ~$Q_2$ is $\theta$\emph{-stable} in the following sense: every involution of $M_1$ restricting to an involution of $M_2$ in the specified inner class takes $Q_2$ to its opposite. Choose \emph{any} parabolic subgroup $Q_1 \subset G$ with Levi factor $M_1$ which is real in the corresponding sense. Then we can define \emph{real parabolic induction} from $M_1$ to $G$ and \emph{cohomological induction} from $M_2$ to $M_1$. We can regard these as maps between the relevant Grothendieck groups (of representations of real forms of $M_2$, $M_1$, and $G$):
\begin{equation}\label{eq:inductionsintro}R^{\fg}_{\fq_1}\colon K\Pi(M_1) \to K\Pi(G), \qquad R^{\fm_1}_{\fq_2}\colon K\Pi(M_2) \to K\Pi(M_1).\end{equation}
%
%
%
The precise definitions will be given in Section \ref{sec:parabolicinductionextended}. We prove the following theorem.

\begin{theorem}[Jordan Decomposition for Arthur Packets, see Theorem \ref{thm:decompositiongeneral}(v) below]\label{thm:Jordandecompintro}
Let $\psi\colon W_{\RR} \times SL_2(\CC) \to {}^{\vee}G^{\Gamma}$ be an Arthur parameter and fix all of the notation introduced in the paragraphs above. Then for any $s \in A_{\psi_2}$, there is an identity in $K\Pi(G)$
\begin{equation}\label{eq:characteridentitiesintro}\eta_{\psi}(s) =  \kappa R^{\fg}_{\fq_1}R^{\fm_1}_{\fq_2}\eta_{\psi_2}(s)\end{equation}
where $\eta_{\psi}(s)$ and $\eta_{\psi_2}(s)$ are the virtual representations defined in (\ref{eq:Arthurstables}) and $\kappa$ is a certain explicitly determined sign (see Proposition \ref{prop:estarthetastableinduction}(iii)).
\end{theorem}

We note that a similar result was obtained, using rather different methods, by Moeglin and Renard (\cite{MoeglinRenard2018}) in the following special case: $G(\RR)$ is a quasisplit symplectic or special orthogonal group and $\psi$ is of `good parity'. 

Setting $s=1$ in Theorem \ref{thm:Jordandecompintro}, we obtain the following result.

\begin{cor}[See Theorem \ref{thm:decompositiongeneral}(ii) below]\label{cor:Jordandecompintro}
Let $\psi\colon W_{\RR} \times SL_2(\CC) \to {}^{\vee}G^{\Gamma}$ be an Arthur parameter and fix all of the notation introduced in the paragraphs preceding Theorem \ref{thm:Jordandecompintro}. For every $\pi \in \Pi(G)_{\psi}$, there is some $\pi_2 \in \Pi(M_2)_{\psi_2}$ such that $\pi$ is an irreducible composition factor of $R^{\fg}_{\fq_1}R^{\fm_1}_{\fq_2}\pi_2$. 
\end{cor}

The inductions appearing in Corollary \ref{cor:Jordandecompintro} are known to preserve unitarity, at least when applied to elements of the unipotent Arthur packet $\Pi(M_2)_{\psi_2}$. Thus, Arthur's unitarity conjecture (Theorem \ref{thm:unitaryintro}) follows from Corollary \ref{cor:Jordandecompintro}, combined with Theorem \ref{thm:unipotentunitary}.

We note that Corollary \ref{cor:Jordandecompintro} also implies a form of Jiang's conjecture \cite[Conjecture 4.2]{Jiang}, which posits an upper bound on the associated varieties (equivalently, wavefront sets) of the members of Arthur packets in terms of the Barbasch-Vogan dual of the Arthur $SL_2(\CC)$. See Corollary \ref{cor:Jiang} for a precise statement.



In the statements above, the Levi subgroups ${}^{\vee}M_2^{\Gamma} \subset {}^{\vee}M_1^{\Gamma} \subset {}^{\vee}G^{\Gamma}$ and parabolic subgroups $Q_2 \subset M_1$, $Q_1 \subset G$ are constructed in a very specific fashion from the Arthur parameter $\psi$. However, many of our results hold in a much more general setting. In Section \ref{subsec:decompositionsparameters}, we introduce the notion of a `decomposition' of $\psi$. This is a complicated collection of Lie-theoretic structures, including Levi subgroups ${}^{\vee}M_2^{\Gamma} \subset {}^{\vee}M_1^{\Gamma} \subset {}^{\vee}G^{\Gamma}$ containing the image of $\psi$ and parabolic subgroups $Q_2 \subset M_1, Q_1 \subset G$ satisfying various conditions. Given any decomposition, one can define induction maps as in (\ref{eq:inductionsintro}) as well as virtual characters $\eta_{\psi}(s)$, $\eta_{\psi_2}(s)$. In this more general setting, we are not able to establish all of the identities of Theorem \ref{thm:Jordandecompintro}. However, we do get an identity for $s=1$. Thus, Corollary \ref{cor:Jordandecompintro} holds in the setting of a general decomposition. This is useful for the following reason: using general decompositions, it is often possible to realize even \emph{unipotent} Arthur packets via real parabolic and cohomological induction from unipotent Arthur packets for proper Levi subgroups of $G$. This allows us to reduce the study of general Arthur packets (and in particular, questions of unitarity) to an even smaller class of unipotent Arthur packets for which special arguments may be available (see Remark \ref{rmk:reductiontodistinguished}).

There are several (interrelated) technical complications which we have so far largely ignored. First, if we replace $L$-groups with the more general $E$-groups of \cite{AdamsBarbaschVogan} (as we must, in the case of ${}^{\vee}M_2^{\Gamma}$), $L$-packets (and Arthur packets) consist not of representations of real forms of $G$, but rather of certain algebraic covers of real forms of $G$. A second issue is that the refined version of the Langlands classification (which is required for the formulation of our results) is not quite canonical, as it requires a certain auxiliary choice (namely, the choice of a \emph{Whittaker datum} for the quasisplit inner form of $G$). We have not explained above how to make such choices consistently. Finally, it is technically very important to consider \emph{strong real forms} of $G$ (and its Levi subgroups $M_2$ and $M_1$) rather than ordinary real forms in order to make the statements above precise. These issues (and a few others) are carefully addressed in Sections \ref{sec:ABV} and \ref{sec:functoriality}. 

Now we will describe the structure of the paper. In Section \ref{sec:reptheory}, we will recall some representation theory preliminaries, including extended groups, strong real forms, $E$-groups and $L$-groups, cohomological induction, and coherent continuation. Much of this material is from \cite{AdamsBarbaschVogan}. We will also introduce a generalization of coherent continuation which we call `$L$-coherent continuation', depending on a Levi subgroup $L$. This is an important ingredient in the proof of our main theorem. In Section \ref{sec:ABV}, we will recall the Langlands classification and definition of Arthur packets as formulated in \cite{AdamsBarbaschVogan}. We will also define some convolution functors between categories of sheaves on spaces of Langlands parameters and explain their relation to coherent continuation. In Section \ref{sec:functoriality}, we will prove our main theorems. We first consider the case of `good range' decompositions (i.e., decompositions for which the infinitesimal character of the unipotent Arthur parameter $\psi_2: W_{\RR} \times SL_2(\CC) \to {}^{\vee}M_2^{\Gamma}$ is in the good range for $\fq_2$). Then we reduce to this case using $L$-coherent continuation and convolution functors. The reduction step requires a fairly involved Lie-theoretic calculation, which we have relegated to an appendix.

\subsection{Acknowledgments}

We would like to thank Chi-Heng Lo for suggesting a connection between our main result and Jiang's Conjecture. We would also like to thank Jörg Schürmann for pointing out an error in an earlier version of this article.

\section{Representation theory preliminaries}\label{sec:reptheory}

In this section, we will recall some preliminary facts and definitions related to admissible representations of real reductive groups, e.g. extended groups, $E$-groups, strong real forms, weakly unipotent representations, (cohomological) parabolic induction, and coherent continuation. The only new idea in this section is the notion of `$L$-coherent continuation' (for a Levi subgroup $L$), which generalizes the classical notion of coherent continuation.

\subsection{Extended groups and strong real forms}\label{sec:extendedgroups}

Let $G$ be a complex connected reductive algebraic group. 

\begin{definition} [{\cite[Definition 2.13]{AdamsBarbaschVogan}}]
A \emph{weak extended group} for $G$ is a Lie group $G^{\Gamma}$ subject to the following conditions:
\begin{itemize}
    \item[(i)] $G^{\Gamma}$ contains $G$ as an index-two subgroup, and
    \item[(ii)] Every element of $G^{\Gamma} \setminus G$ acts on $G$ by an anti-holomorphic automorphism. 
    \end{itemize}
\end{definition}

Every weak extended group admits a canonical surjective homomorphism $G^{\Gamma} \to \Gamma := \mathrm{Gal}(\CC,\RR)$, giving rise to a short exact sequence of Lie groups
$$1 \to G \to G^{\Gamma} \to \Gamma \to 1.$$
Note that conjugation by an element $\delta \in G^{\Gamma} \setminus G$ determines a well-defined involution of the based root datum of $G$, see \cite[Corollary 2.16]{AdamsBarbaschVogan}, and hence an inner class of real forms of $G$, see \cite[Proposition 2.12]{AdamsBarbaschVogan}. Conjugation by $\delta$ also determines a well-defined involution $\sigma_Z$ of $Z(G)$ (which is in fact determined by the root datum involution).

\begin{definition} [{\cite[Definition 2.13]{AdamsBarbaschVogan}}]
Let $G^{\Gamma}$ be a weak extended group. A \emph{strong real form} of $G^{\Gamma}$ is an element $\delta \in G^{\Gamma} \setminus G$ such that $\delta^2$ is a finite-order element of $Z(G)$. The \emph{real form associated to $\delta$} is the anti-holomorphic involution $\Ad(\delta)$ of $G$ defined by conjugation by $\delta$. The \emph{group of real points} of $\delta$ is the real reductive group $G(\RR,\delta) := G^{\Ad(\delta)}$. Two strong real forms are \emph{equivalent} if they are conjugate under $G$. 
\end{definition}

\begin{definition}[{\cite[Definition 1.12]{AdamsBarbaschVogan}}]
An \emph{extended group} for $G$ is a pair $(G^{\Gamma},\mathcal{W})$, subject to the following conditions.
\begin{itemize}
    \item[(i)] $G^{\Gamma}$ is a weak extended group for $G$.
    \item[(ii)] $\mathcal{W}$ is a conjugacy class of triples $(\delta,N,\chi)$, where
    \begin{itemize}
        \item[(a)] $\delta \in G^{\Gamma} \setminus G$ is a strong real form of $G^{\Gamma}$.
        \item[(b)] $N \subset G$ is a maximal unipotent subgroup, normalized by $\delta$ (write $N(\RR) = N^{\Ad(\delta)}$).
        \item[(c)] $\chi$ is a one-dimensional non-degenerate unitary character of $N(\RR)$ (`non-degenerate' means non-trivial on each restricted simple root subgroup of $N(\RR)$).
    \end{itemize}
\end{itemize}
\end{definition}

Here is what `restricted simple root subgroup' means. The group $N(\RR)$
  is normalized by a real Cartan subgroup $H(\RR)=TA$, with $T$
  compact and $A$ a vector group. The weights of $\mathrm{Lie}(A)$ on $\fg$ are called
  `restricted roots'; they form a (possibly non-reduced) root
  system. The weights in $\mathrm{Lie}(N(\RR))$ are a set of positive roots, so some
  of these are simple.

\subsection{$L$-groups and $E$-groups}\label{sec:Egroups}

Let $G$ be a complex connected reductive algebraic group and let $\Phi_0(G)$ be the associated based root datum. Let ${}^{\vee}\Phi_0(G)$ denote the dual based root datum. A \emph{dual group} for $G$ is a complex connected reductive algebraic group ${}^{\vee}G$ equipped with a based root datum isomorphism $\Phi_0({}^{\vee}G) \xrightarrow{\sim} {}^{\vee}\Phi_0(G)$. Note that any two dual groups for $G$ are isomorphic; the isomorphism is canonical up to inner automorphisms.

\begin{definition}[{\cite[Definition 4.3]{AdamsBarbaschVogan}}]
A \emph{weak $E$-group} for $G$ is a complex algebraic group ${}^{\vee}G^{\Gamma}$ containing a dual group ${}^{\vee}G$ as an index-two subgroup. 
\end{definition}

Every weak $E$-group admits a canonical surjective homomorphism ${}^{\vee}G^{\Gamma} \to \Gamma := \mathrm{Gal}(\CC,\RR)$, giving rise to a short exact sequence of algebraic groups
$$1 \to {}^{\vee}G \to {}^{\vee}G^{\Gamma} \to \Gamma \to 1.$$

Note that conjugation by an element ${}^{\vee}\delta \in {}^{\vee}G^{\Gamma} \setminus {}^{\vee}G$ determines a well-defined involution of the based root datum of $G$, see \cite[Proposition 4.4]{AdamsBarbaschVogan}, and hence an inner class of real forms of $G$. We say that a weak $E$-group ${}^{\vee}G^{\Gamma}$ \emph{corresponds} to a weak extended group $G^{\Gamma}$ if ${}^{\vee}G^{\Gamma}$ and $G^{\Gamma}$ give rise to the same involution of the based root datum of $G$ (or equivalently, the same inner class of real forms). Conjugation by ${}^{\vee}\delta$ also determines a well-defined involution $\theta_Z$ of $Z({}^{\vee}G)$ (which is in fact determined by the root datum involution).

Suppose $\theta$ is an algebraic involution of $G$ preserving a Borel subgroup $B$. Choose a $\theta$-stable maximal torus $T \subset B$. We say that the pair $(\theta,B)$ is \emph{large} (or that $B$ is \emph{large} for $\theta$) if $\theta$ acts by $-1$ on every $\theta$-stable simple root space (this definition is independent of the choice of $T$ in $B$).

\begin{definition} [{\cite[Definition 4.14]{AdamsBarbaschVogan}}]\label{def:Egroup2}
An $E$-group for $G$ is a pair $({}^{\vee}G^{\Gamma},\mathcal{S})$ subject to the following conditions:
\begin{itemize}
    \item[(i)] ${}^{\vee}G^{\Gamma}$ is a weak $E$-group.
    \item[(ii)] $\mathcal{S}$ is a conjugacy class of pairs $({}^{\vee}\delta,{}^{\vee}B)$, where ${}^{\vee}\delta$ is a finite-order element in ${}^{\vee}G^{\Gamma} \setminus {}^{\vee}G$ and ${}^{\vee}B$ is a Borel subgroup of ${}^{\vee}G$.
    \item[(iii)] Suppose $({}^{\vee}\delta,{}^{\vee}B) \in \mathcal{S}$. Then conjugation by ${}^{\vee}\delta$ defines an involution $\theta$ of ${}^{\vee}G$ preserving ${}^{\vee}B$, and the pair $(\theta,{}^{\vee}B)$ is large.
\end{itemize}
\end{definition}

Choose a Borel subgroup ${}^{\vee}B \subset {}^{\vee}G$ and a maximal torus ${}^{\vee}H \subset {}^{\vee}B$. Write 
$$2\rho\colon \CC^{\times} \to {}^{\vee}H$$
for the sum of the positive co-roots of ${}^{\vee}G$. Consider the order-2 element
$$z(\rho) := 2\rho(-1)  \in Z({}^{\vee}G).$$
This element is independent of the choice of ${}^{\vee}B$ and ${}^{\vee}H$, and is fixed by $\theta_Z$. The \emph{second invariant} of $({}^{\vee}G^{\Gamma},\mathcal{S})$ is the element $z(\rho)({}^{\vee}\delta)^2 \in Z({}^{\vee}G)^{\theta_Z}$, for any choice of ${}^{\vee}\delta$ such that $({}^{\vee}\delta,{}^{\vee}B) \in \mathcal{S}$.

\begin{definition}
An $L$-group for $G$ is an $E$-group $({}^{\vee}G^{\Gamma},\mathcal{S})$ with second invariant $1$. 
\end{definition}

This is equivalent to the original definition of an $L$-group (see for example \cite[page 29]{borelL}).

\subsection{Langlands parameters and Arthur parameters}

\begin{definition}[see \cite{borelL},\cite{Langlands}]
The \emph{Weil group} of $\RR$ is the Lie group $W_{\RR}$ generated by $\CC^{\times}$ and a single element $j$, subject to the relations
$$j^2=-1 \in \CC^{\times}, \qquad jzj^{-1} = \bar{z}, \quad z \in \CC^{\times}.$$
\end{definition}

There is a surjective Lie group homomorphism $W_{\RR} \to \Gamma := \mathrm{Gal}(\CC,\RR)$ taking $\CC^{\times}$ to the identity and $j$ to complex conjugation. This gives rise to a short exact sequence of Lie groups
$$1 \to \CC^{\times} \to W_{\RR} \to \Gamma \to 1.$$
Now let ${}^{\vee}G^{\Gamma}$ be a weak $E$-group for $G$. 

\begin{definition}[see \cite{borelL},\cite{Langlands}]\label{def:Langlandsparam}
A \emph{Langlands parameter} for ${}^{\vee}G^{\Gamma}$ is a continuous homomorphism
$$\varphi\colon W_{\RR} \to {}^{\vee}G^{\Gamma}$$
such that
\begin{itemize}
    \item[(i)] The image of $\varphi$ consists of semisimple elements;
    \item[(ii)] $\varphi$ is compatible with the projections $W_{\RR} \to \Gamma$ and ${}^{\vee}G^{\Gamma} \to \Gamma$.
\end{itemize}
\end{definition}

Let $P({}^{\vee}G^{\Gamma})$ denote the set of Langlands parameters for ${}^{\vee}G^{\Gamma}$. Note that ${}^{\vee}G$ acts on $P({}^{\vee}G^{\Gamma})$ by conjugation. Two Langlands parameters are said to be \emph{equivalent} if they are conjugate under this action. Let $\Phi({}^{\vee}G^{\Gamma})$ denote the set of equivalence classes of Langlands parameters for ${}^{\vee}G^{\Gamma}$.

\begin{definition}[{\cite[Section 6]{Arthur1989}}]\label{def:Arthurparam}
An \emph{Arthur parameter} for ${}^{\vee}G^{\Gamma}$ is a continuous homomorphism
$$\psi\colon W_{\RR} \times SL_2(\CC) \to {}^{\vee}G^{\Gamma}$$
such that
\begin{itemize}
    \item[(i)] $\psi|_{W_{\RR}}$ is a Langlands parameter for ${}^{\vee}G^{\Gamma}$;
    \item[(ii)] $\psi(W_{\RR})$ is bounded;
    \item[(iii)] $\psi|_{SL_2(\CC)}$ is algebraic.
\end{itemize}
We say that $\psi$ is
\begin{itemize}
    \item[(iv)] \emph{unipotent} if $\psi(\CC^{\times}) \subset Z({}^{\vee}G)$;
    \item[(v)]\emph{distinguished} if $\psi(SL_2(\CC))$ does not belong to a proper Levi subgroup of ${}^{\vee}G$.
\end{itemize}
\end{definition}

If $\psi(\CC^{\times})$ is not contained in $Z({}^{\vee}G)$, then $Z_{{}^{\vee}G}(\psi(\CC^{\times})) \subset {}^{\vee}G$ is a proper Levi subgroup containing $\psi(SL_2(\CC))$. Thus, all distinguished Arthur parameters are unipotent. 

Note that ${}^{\vee}G$ acts by conjugation on the set of Arthur parameters. Two Arthur parameters are said to be \emph{equivalent} if they are conjugate under this action. We write $\Psi({}^{\vee}G^{\Gamma})$ for the set of equivalence classes of Arthur parameters for ${}^{\vee}G^{\Gamma}$.

To each Arthur parameter $\psi\colon W_{\RR} \times SL_2(\CC) \to {}^{\vee}G^{\Gamma}$, we can associate a Langlands parameter $\varphi_{\psi}$ according to the formula
\begin{equation}\label{eq:ArthurtoLanglands}\varphi_{\psi}\colon W_{\RR} \to {}^{\vee}G^{\Gamma}, \qquad \varphi_{\psi}(w) = \psi(w, \begin{pmatrix} ||w||^{1/2} & 0\\ 0 & ||w||^{-1/2} \end{pmatrix}. 
\end{equation}
This defines a one-to-one map
$$\Psi({}^{\vee}G^{\Gamma}) \hookrightarrow \Phi({}^{\vee}G^{\Gamma}).$$

\subsection{Projective representations}\label{sec:projectiverepresentations}

Let $G^{\Gamma}$ be a weak extended group. Associated to $G^{\Gamma}$ is a certain pro-algebraic group $G^{can}$ called the \emph{canonical covering of $G$}, see \cite[Definition 10.1]{AdamsBarbaschVogan}. This group comes equipped with a surjective homomorphism $G^{can} \to G$. The kernel of this homomorphism is a pro-finite group, denoted $\pi_1(G)^{can}$. The definitions are arranged so that there is a natural isomorphism 
$$\{\text{finite-order elements in } Z({}^{\vee}G)^{\theta_Z}\} \xrightarrow{\sim} \{\text{continuous characters of } \pi_1(G)^{can}\}.$$

If $H$ is a subgroup of $G$, we write $H^{G,can}$ (or simply $H^{can}$, if $G$ is understood) for the preimage of $H$ under $G^{can} \to G$.

\begin{definition}[{\cite[Definition 10.3]{AdamsBarbaschVogan}}]\label{def:projectiverepresentations}
Let $G^{\Gamma}$ be a weak extended group and let $G(\RR)$ be a real form of $G$ in the inner class defined by $G^{\Gamma}$. There is a short exact sequence
$$1 \to \pi_1(G)^{can} \to G(\RR)^{G,can} \to G(\RR) \to 1.$$
An \emph{irreducible projective representation of $G(\RR)$} is an irreducible admissible representation $\pi$ of $G(\RR)^{G,can}$. Suppose $z$ is a finite-order element in $Z({}^{\vee}G)^{\theta_Z}$. We say that $\pi$ is \emph{of type $z$} if the restriction of $\pi$ to $\pi_1(G)^{can}$ is a multiple of the character parameterized by $z$. Write $\Pi^z(G(\RR))$ for the set of infinitesimal equivalence classes of irreducible type $z$ projective representations of $G(\RR)$ and $K\Pi^z(G(\RR))$ for the free $\ZZ$-module with basis $\Pi^z(G(\RR))$.
\end{definition}

\begin{definition}\label{def:repsofextendedgroups}
Let $G^{\Gamma}$ be a weak extended group. An \emph{irreducible projective representation of a strong real form of $G^{\Gamma}$} is a pair $(\delta,\pi)$ subject to the following conditions:
\begin{itemize}
    \item[(i)] $\delta$ is a strong real form of $G^{\Gamma}$.
    \item[(ii)] $\pi$ is an irreducible projective representation of $G(\RR,\delta)$.
\end{itemize}
Two such representations $(\delta,\pi)$ and $(\delta',\pi')$ are \emph{infinitesimally equivalent} if there is an element $g \in G$ such that $g\delta g^{-1}=\delta'$ and $\pi \circ \Ad(g^{-1})$ is infinitesimally equivalent to $\pi'$. Suppose $z$ is a finite-order element in $Z({}^{\vee}G)^{\theta_Z}$. We say that $(\delta,\pi)$ is \emph{of type $z$} if $\pi$ is of type $z$. Write $\Pi^z(G^{\Gamma})$ for the set of infinitesimal equivalence classes of irreducible type $z$ projective representations of strong real forms of $G^{\Gamma}$ and $K\Pi^z(G^{\Gamma})$ for the free $\ZZ$-module with basis $\Pi^z(G^{\Gamma})$. 
\end{definition}

\begin{definition}\label{def:projectivegK}
Let $G^{\Gamma}$ be a weak extended group, let $G(\RR)$ be a real form of $G$ in the inner class defined by $G^{\Gamma}$, and let $K(\RR) \subset G(\RR)$ be a maximal compact subgroup with complexification $K$. There is a short exact sequence
$$1 \to \pi_1(G)^{can} \to K^{G,can} \to K \to 1.$$
A \emph{projective $(\fg,K)$-module} is a $(\fg,K^{G,can})$-module  (i.e. a $U(\fg)$-module with a compatible rational $K^{G,can}$-action). Suppose $z$ is a finite-order element in $Z({}^{\vee}G)^{\theta_Z}$. We say that a projective $(\fg,K)$-module is \emph{of type $z$} if its restriction to $\pi_1(G)^{can}$ is a multiple of the character parameterized by $z$. Write $M^z(\fg,K)$ for the category of type-$z$ projective $(\fg,K)$-modules, $M^z_{fl}(\fg,K)$ for the subcategory of finite-length modules, and $KM^z_{fl}(\fg,K)$ for the Grothendieck group of $M^z_{fl}(\fg,K)$.
\end{definition}

If we fix a real form $G(\RR)$ of $G$ and a maximal compact subgroup $K(\RR) \subset G(\RR)$, there is natural identification
$$K\Pi^z(G(\RR)) \xrightarrow{\sim} KM^z_{fl}(\fg,K).$$
Also, if we choose a set $\Delta$ of equivalence classes of strong real forms of $G^{\Gamma}$, then there is a natural identification
$$\bigsqcup_{\delta \in \Delta} \Pi^z(G(\RR,\delta)) \xrightarrow{\sim} \Pi^z(G^{\Gamma}).$$
and hence a natural identification
$$\bigoplus_{\delta \in \Delta} K\Pi^z(G(\RR,\delta)) \xrightarrow{\sim} K\Pi^z(G^{\Gamma})$$

\subsection{Parabolic induction}\label{sec:parabolicinduction}

Let $G(\RR)$ be a real form of $G$. Choose a maximal compact subgroup $K(\RR) \subset G(\RR)$ and let $\theta$ be the corresponding Cartan involution of $\fg$. Choose a parabolic subalgebra $\fq \subset \fg$ with Levi decomposition $\fq=\fl \oplus \fu$. Let $L$ (resp.~$Q$) denote the connected subgroup of $G$ corresponding to $\fl$ (resp.~$\fq$). Assume throughout this section that $\fl$ is $\theta$-stable. 

Let $2\rho_{\fu}$ denote the determinant character of $\fu$:
$$2\rho_{\fu}(l) = \det(\Ad(l)|_{\fu}), \qquad l \in L.$$
Consider the fibered product
\begin{center}
    \begin{tikzcd}
    \widetilde{L} \ar[r] \ar[d] & \CC^{\times} \ar[d,"(\bullet)^2"]\\
    L \ar[r,"2\rho_{\fu}"] & \CC^{\times}
    \end{tikzcd}
\end{center}
The left vertical map $\widetilde{L} \to L$ is a 2-fold cover of $L$ and the top horizontal map $\widetilde{L} \to \CC^{\times}$ is a genuine character of $\widetilde{L}$, which we will denote by $\rho_{\fu}$. By the construction of $L^{can}$, there is a natural covering map $L^{can} \to \widetilde{L}$. By pullback along this covering map, we can regard $\rho_{\fu}$ as a one-dimensional projective $(\fl, L \cap K)$-module. 

Dually, we can consider the algebraic co-character corresponding to the sum of roots in $\fu$
\begin{equation}\label{eq:2rhocheck}(2\rho_{\fu})^{\vee}\colon \CC^{\times} \to {}^{\vee}L\end{equation}
and the order-2 element
\begin{equation}\label{eq:zrhou}z(\rho_{\fu}) := (2\rho_{\fu})^{\vee}(-1) \in Z({}^{\vee}L)^{\theta_Z}\end{equation}
Note that $z(\rho_{\fu})$ depends only on $\fl$ (i.e., it is independent of the choice of $\fq$) and that the projective $(l,L \cap K)$-module $\rho_{\fu}$ is of type $z(\rho_{\fu})$.

Now fix a finite-order element $z \in Z({}^{\vee}G)^{\theta_Z}$. We can regard $z$ also as an element of $Z({}^{\vee}L)^{\theta_Z}$ via the natural inclusion $Z({}^{\vee}G)^{\theta_Z} \subset Z({}^{\vee}L)^{\theta_Z}$. We will define a functor called \emph{parabolic induction}
$$\mathcal{R}^{\fg}_{\fq}\colon M^{zz(\rho_{\fu})}(\fl,L \cap K) \to M^z(\fg,K).$$
This functor is built out of several ingredients:

\begin{itemize}
    \item Let $(\bullet)^{\#}\colon M^{zz(\rho_{\fu})}(\fl, L \cap K) \to M^z(\fl,L \cap K)$ denote the functor given by tensoring with $\rho_{\fu}$.
    \item Let $\mathcal{F}\colon M^z(\fl, L \cap K) \to M^z(\fq, L \cap K)$ denote the forgetful functor corresponding to the quotient map $(\fq, L \cap K) \to (\fl, L \cap K)$.
    \item Let $P\colon M^z(\fq, L \cap K) \to M^z(\fg, L \cap K)$ denote the algebraic induction functor, i.e.
    $$P(\pi) = U(\fg) \otimes_{U(\fq^{op})} \pi.$$
    \item Let $\Pi\colon M^z(\fg, L \cap K) \to M^z(\fg, K)$ denote the Bernstein functor, see \cite[Chapter 4.4]{KnappVogan1995}. 
\end{itemize}
Now we define
$$\mathcal{R}^{\fg}_{\fq}:= \Pi \circ P \circ \mathcal{F} \circ (\bullet)^{\#}\colon M^{zz(\rho_{\fu})}(\fl,L \cap K) \to M^z(\fg,K)$$
Note that  $(\bullet)^{\#}$ and $\mathcal{F}$ are exact and $\Pi$ and $P$ are right exact. Thus, $\mathcal{R}^{\fg}_{\fq}$ is right exact. Since the categories $M^{zz(\rho_{\fu})}(\fl, L \cap K)$ and $ M^z(\fg, K)$ have enough projective objects, there are left-derived functors
$$L_i\mathcal{R}^{\fg}_{\fq}\colon M^{zz(\rho_{\fu})}(\fl,L \cap K) \to M^z(\fg, K).$$

\begin{theorem}[{\cite[Chapter 6.3]{Vogan1981}}]\label{thm:propertiesofI}
The functors
$$L_i\mathcal{R}^{\fg}_{\fq}\colon M^{zz(\rho_{\fu})}(\fl,L \cap K) \to M^z(\fg, K)$$
preserve infinitesimal character and take finite-length modules to finite-length modules. Moreover, $L_i\mathcal{R}^{\fg}_{\fq}=0$ for $i>>0$.
\end{theorem}

Thanks to Theorem \ref{thm:propertiesofI}, we can define a group homomorphism
$$R^{\fg}_{\fq}\colon KM^{zz(\rho_{\fu})}_{fl}(\fl,L \cap K) \to KM^z_{fl}(\fg, K), \qquad  [\pi] \mapsto \sum_i (-1)^i [L_i\mathcal{R}^{\fg}_{\fq}(\pi)],$$
and hence a group homomorphism
$$R^{\fg}_{\fq}\colon K\Pi^{zz(\rho_{\fu})}(L(\RR)) \to K\Pi^z(G(\RR)).$$

\subsubsection{Real parabolic induction}\label{sec:realparabolicinduction}

Suppose now that $\fq$ is the complexified Lie algebra of a real parabolic subgroup $Q(\RR) = L(\RR)U(\RR)$ (it is equivalent to assume that $\fq$ is taken to its opposite by $\theta$). The character $2\rho_{\fu}$, when restricted to the subgroup $L(\RR) \subset L$, takes values in $\RR$. Let $|\rho_{\fu}|$ denote the positive square-root of the absolute value of this character. Then $|\rho_{\fu}|/\rho_{\fu}$ is a one-dimensional type-$z(\rho_{\fu})$ projective $(\fl,K)$-module of infinitesimal character $0$. Define the functor
$$\Ind^{G(\RR)}_{Q(\RR)}\colon M^z(\fl,L \cap K) \to M^z(\fg,K), \qquad \Ind^{G(\RR)}_{Q(\RR)}(\pi) = R^{\fg}_{\fq}(\pi \otimes |\rho_{\fu}|/\rho_{\fu}).$$

\begin{prop}[{\cite[Proposition 6.3.5]{Vogan1981}}]\label{prop:propsofrealinduction}
Suppose that $\fq$ is the complexified Lie algebra of a real parabolic subgroup $Q(\RR) = L(\RR)U(\RR)$. Then $\Ind^{G(\RR)}_{Q(\RR)}$ coincides (after passing to Harish-Chandra modules) with \emph{normalized real parabolic induction}. It has the following properties:
\begin{itemize}
    \item[(i)] $\Ind^{G(\RR)}_{Q(\RR)}$ preserves infinitesimal character.
    \item[(ii)] $\Ind^{G(\RR)}_{Q(\RR)}$ is exact.
    \item[(iii)] $\Ind^{G(\RR)}_{Q(\RR)}$ takes unitary modules to unitary modules.
    \item[(iv)] $\Ind^{G(\RR)}_{Q(\RR)}$ is conservative: that is, if $f$ is a morphism in $M^z(\fl,L \cap K)$ and $\Ind(f)$ is an isomorphism, then $f$ is an isomorphism.
\end{itemize}
\end{prop}

\subsubsection{Cohomological induction}

There is a canonical decomposition
\begin{equation}\label{eq:levidecompCartan}\fh^* \simeq  (\fh/\fz(\fl))^* \oplus \fz(\fl)^*, \qquad \lambda \mapsto \lambda_s +\lambda_z.\end{equation}

\begin{definition}\label{def:weaklyfair}
Let $\lambda \in \fh^*$. We say that $\lambda$ is in
\begin{itemize}
    \item[(i)] \emph{the good range} for $\fq$ if 
    $$\mathrm{Re}\langle \lambda, \alpha^{\vee}\rangle > 0, \qquad \forall \alpha \in \Delta(\fu,\fh);$$
    %
    \item[(ii)] \emph{the weakly fair range} for $\fq$ if 
    $$\mathrm{Re}\langle \lambda_z, \alpha^{\vee} \rangle \geq 0, \qquad \forall \alpha \in \Delta(\fu, \fh).$$
\end{itemize}
We say that $\lambda$ is in \emph{the good range} (resp.~\emph{weakly  fair range}) for $\fl$ if there exists a parabolic subalgebra $\fq' \subset \fg$ with Levi factor $\fl$ such that $\lambda$ is in the good range (resp.~weakly fair range) for $\fq'$. 
\end{definition}

Note that if $\lambda \in \fh^*$ is in the good range (resp.~weakly fair range) for $\fq$ (or $\fl$), then so is $w\lambda$, for any $w \in W_L$. So these are properties which can be ascribed to $W_L$-orbits in $\fh^*$, and hence to infinitesimal characters for $U(\fl)$. If $\pi$ is an irreducible projective type-$zz(\rho_{\fu})$ $(\mathfrak{l},K \cap L)$-module, we say that $\pi$ is in the good range (resp.~weakly fair range) for $\fq$ (or $\fl$) if its infinitesimal character has this property. We will occasionally use the following elementary lemma.

\begin{lemma}\label{lem:rhotwistgoodrange}
Let $\lambda' \in \fh^*$. Then, for $N>>0$, the element $\lambda' + N\rho_{\fu}$ is in the good range for $\fq$.
\end{lemma}

\begin{proof}
Choose a positive system $\Delta^+(\fl,\fh)$ for $\fl$. Then $\Delta^+(\fg,\fh) = \Delta^+(\fl,\fh) \cup \Delta(\fu,\fh)$ is a positive system for $\fg$. Let $\alpha$ be a simple root for $\fg$ in $\Delta(\fu,\fh)$. Then $\langle \frac{1}{2}\sum \Delta^+(\fg,\fh),\alpha^{\vee}\rangle =1$ and $\langle \beta,\alpha^{\vee}\rangle \geq 0$ for every $\beta \in \Delta^+(\fl,\fh)$. In particular
$$\langle \rho_{\fu},\alpha^{\vee}\rangle = \langle \frac{1}{2}\sum \Delta^+(\fg,\fh),\alpha^{\vee}\rangle - \langle \frac{1}{2}\sum \Delta^+(\fl,\fh),\alpha^{\vee}\rangle \geq 1.$$
The lemma follows.
\end{proof}

For the next definition, choose a positive-definite $W$-invariant symmetric bilinear form $\langle -,-\rangle$ on $X^*(H) \otimes_{\ZZ} \RR$. 

\begin{definition}[{\cite[Definition 12.3]{KnappVogan1995}}]\label{def:weaklyunipotent}
Let $\pi$ be an irreducible projective type-$zz(\rho_{\fu})$ $(\fl,L \cap K)$-module of infinitesimal character $\lambda \in X^*(H) \otimes_{\ZZ} \RR$. Then $\pi$ is \emph{weakly unipotent} if for every finite-dimensional algebraic $L^{ad}$-representation $F$ and infinitesimal character $\lambda'$ appearing in $\pi \otimes F$, we have the following inequality
$$||\lambda'_s|| \geq ||\lambda_s||,$$
where $\lambda_s$ and $\lambda'_s$ are defined as in (\ref{eq:levidecompCartan}).
\end{definition}

\begin{prop} [{\cite[Theorems 8.2 and 12.4]{KnappVogan1995}}]\label{prop:weaklyfaircohomological}
Suppose that ${\fq}$ is $\theta$-stable, and let $\pi$ be an irreducible type-$zz(\rho_{\fu})$ projective $(\mathfrak{l},L \cap K)$-module. Let
$$N:=\dim(\fu \cap \fk)$$
Then the following are true:

\begin{itemize}
    \item If $\pi$ is weakly unipotent in the weakly fair range, then $L_i\mathcal{R}^{\fg}_{\fq}(\pi) =0$ for $i \neq N$. If $\pi$ is unitary, then $L_N\mathcal{R}^{\fg}_{\fq}(\pi)$ is unitary.
    \item If $\pi$ is in the good range, then $L_i\mathcal{R}^{\fg}_{\fq}(\pi) =0$ for $i \neq N$, and $L_N\mathcal{R}^{\fg}_{\fq}(\pi)$ is irreducible (and nonzero). If $\pi$ is unitary, then $L_N\mathcal{R}^{\fg}_{\fq}(\pi)$ is unitary. 
\end{itemize}
\end{prop}

In the setting of Proposition \ref{prop:weaklyfaircohomological}, we say that the representation $L_N\mathcal{R}^{\fg}_{\fq}(\pi)$ is \emph{cohomologically induced}.

\subsection{Parabolic induction in the extended group setting}\label{sec:parabolicinductionextended}

It will be convenient to define a version of parabolic induction in a setting which permits various real forms of $L$ and $G$ at once. This is easy to do using the language of extended groups. 

Let $G^{\Gamma}$ be a weak extended group and let $L^{\Gamma}$ be a Levi subgroup of $G^{\Gamma}$(this means that $L^{\Gamma}$ is a weak extended group equipped with an inclusion $L^{\Gamma} \subset G^{\Gamma}$, compatible with the projections to $\Gamma$, such that $L \subset G$ is a Levi subgroup). Choose a parabolic subgroup $Q \subset G$ with Levi decomposition $Q=LU$. Even though we have not fixed a real form of $G$, it is reasonable to say that $Q$ is `real' or `$\theta$-stable' with respect to the inclusion $L^{\Gamma} \subset G^{\Gamma}$. This is justified by the following lemma.

\begin{lemma}\label{lem:realisreal}
The following are equivalent:
\begin{itemize}
    \item[(ia)] The involution $\sigma_Z$ of $Z(L)$ determined by $L^{\Gamma}$ fixes the element $\rho_{\fu} \in Z(\fl)$. 
    \item[(iia)] There is a strong real form $\delta \in L^{\Gamma} \setminus L$ with $\delta^2 \in Z(G)$ such that $\Ad(\delta)Q=Q$.
    \item[(iiia)] Every strong real form $\delta \in L^{\Gamma} \setminus L$ with $\delta^2 \in Z(G)$ has $\Ad(\delta)Q=Q$.
\end{itemize}
Also, the following are equivalent:
\begin{itemize}
    \item[(ib)] $\sigma_Z(\rho_{\fu})=-\rho_{\fu}$.
    \item[(iib)] There is a strong real form $\delta \in L^{\Gamma} \setminus L$ with $\delta^2 \in Z(G)$ such that $\Ad(\delta)Q=Q^{op}$.
    \item[(iiib)] Every strong real form $\delta \in L^{\Gamma} \setminus L$ with $\delta^2 \in Z(G)$ has $\Ad(\delta)Q=Q^{op}$.
\end{itemize}
In the case of (ia)-(iiia), we say that $Q$ is `real'. In the case of (ib)-(iiib), we say that $Q$ is `$\theta$-stable'.
\end{lemma}

\begin{proof}
We will prove the equivalence of conditions (ia)-(iiia). The equivalence of conditions (ib)-(iiib) is proved by a similar argument. 

Let $\delta \in L^{\Gamma} \setminus L$ be a strong real form of $L^{\Gamma}$ with $\delta^2 \in Z(G)$. Then $\Ad(\delta)Q$ is a parabolic subgroup of $G$ with Levi decomposition $L\Ad(\delta)U$. Clearly, $\rho_{\Ad(\delta)\fu} = \Ad(\delta)\rho_{\fu} = \sigma_Z\rho_{\fu}$. On the other hand, since $\Ad(\delta)Q$ is a parabolic subgroup, it is determined by the element $\rho_{\Ad(\delta)\fu}$ (it is the connected subgroup of $G$ corresponding to the nonnegative weights for $\rho_{\Ad(\delta)\fu}$). So $\Ad(\delta)Q=Q$ if and only if $\sigma_Z(\rho_{\fu})=\rho_{\fu}$. This proves the equivalence of (ia)-(iiia). 
\end{proof}

Now fix a finite-order element $z \in Z({}^{\vee}G)^{\theta_Z}$ and recall the order-$2$ element $z(\rho_{\fu}) \in Z({}^{\vee}L)^{\theta_Z}$ (\ref{eq:zrhou}). Form the Grothendieck groups $K\Pi^z(G^{\Gamma})$ and $K\Pi^{zz(\rho_{\fu})}(L^{\Gamma})$ as in Definition \ref{def:repsofextendedgroups}. In this setting, we define 
\begin{equation}\label{eq:indextendedgroups}R^{\fg}_{\fq}\colon K\Pi^{zz(\rho_{\fu})}(L^{\Gamma}) \to K\Pi^z(G^{\Gamma}), \qquad R^{\fg}_{\fq}(\delta, \pi) =  \begin{cases*}
                    (\delta,R^{\fg}_{\fq}\pi)
                     & if  $\delta^2 \in Z(G)$  \\
                     0 & else
                 \end{cases*}\end{equation}
The expression $(\delta,R^{\fg}_{\fq}\pi)$ in the formula above requires a bit of explanation. If $\delta \in L^{\Gamma}\setminus L$ is a strong real form of $L^{\Gamma}$ with $\delta^2 \in Z(G)$, then $\delta$ can be regarded as a strong real form of $G^{\Gamma}$. In particular, conjugation by $\delta$ is a real form of $G$ preserving $L$. Write $L(\RR) \subset G(\RR)$ for the corresponding groups of real points. Now we are in the setting described in the beginning of Section \ref{sec:parabolicinduction}. In particular, for $\pi \in K\Pi^{zz(\rho_{\fu})}(L(\RR))$ we can define the induced representation $R^{\fg}_{\fq}\pi \in K\Pi^z(G(\RR))$, and $(\delta,R^{\fg}_{\fq}\pi) \in K\Pi^z(G^{\Gamma})$. 

We now reexamine the issue of unitarity in the setting of extended groups. There are three cases that will concern us.

\begin{cor}\label{cor:unitaryinduction}
Let $\pi \in K\Pi_{\lambda}^{zz(\rho_{\fu})}(L^{\Gamma})$ be a linear combination of unitary irreducible projective type-$zz(\rho_{\fu})$ representations of strong real forms of $L^{\Gamma}$ of infinitesimal character $\lambda \in \fh^*$. Then the following are true:
\begin{itemize}
    \item[(i)] If $Q$ is real, then $R^{\fg}_{\fq}\pi \in K\Pi^z(G^{\Gamma})$ is a linear combination of unitary irreducible projective type-$z$ representations of strong real forms of $G^{\Gamma}$.
    \item[(ii)] If $Q$ is $\theta$-stable and $\lambda$ is in the good range for $\fq$ (Definition \ref{def:weaklyfair}), then $R^{\fg}_{\fq}\pi \in K\Pi^z(G^{\Gamma})$ is a linear combination of unitary irreducible projective type-$z$ representations of strong real forms of $G^{\Gamma}$.
    \item[(iii)] If $Q$ is $\theta$-stable, $\lambda$ is in the weakly fair range for $\fq$ (Definition \ref{def:weaklyfair}), and $\pi$ is weakly unipotent (Definition \ref{def:weaklyunipotent}), then $R^{\fg}_{\fq}\pi \in K\Pi^z(G^{\Gamma})$ is a linear combination of unitary irreducible projective type-$z$ representations of strong real forms of $G^{\Gamma}$ (or $0$).
\end{itemize}
\end{cor}

\begin{proof}
This follows immediately from Lemma \ref{lem:realisreal} and Propositions \ref{prop:propsofrealinduction}(iii) and \ref{prop:weaklyfaircohomological}.
\end{proof}

\subsection{The Langlands classification via standard limit characters}\label{sec:LLC}

In this section, we will recall the classification of irreducible admissible representations in terms of \emph{standard limit characters} of Cartan subgroups. This formulation of the Langlands classification does not involve the dual group.

\begin{definition}[{\cite[Definition 11.13]{AdamsBarbaschVogan}}]\label{def:limitcharacters}
Let $G^{\Gamma}$ be a weak extended group, and let $G(\RR)$ be a real form of $G$ in the inner class defined by $G^{\Gamma}$. A \emph{limit character for $G(\RR)$} is a quadruple
$$\Lambda = (H(\RR), \Lambda^{can}, \Delta_{i\RR}^+, \Delta_{\RR}^+)$$
subject to the following conditions
\begin{itemize}
    \item[(i)] $H(\RR)$ is the group of real points of a maximal torus $H$ of $G$.
    \item[(ii)] $\Lambda^{can}$ is an irreducible projective representation $H(\RR)$.
    \item[(iii)] $\Delta_{i\RR}^+ \subset \Delta(\fg,\fh)$ is a system of positive imaginary roots.
    \item[(iv)] $\Delta_{\RR}^+ \subset \Delta(\fg,\fh)$ is a system of positive real roots. 
\end{itemize}
Suppose $z$ is a finite-order element in $Z({}^{\vee}G)^{\theta_Z}$. We can regard $z$ as an element of $({}^{\vee}H)^{\theta_Z}$ via the natural inclusion $Z({}^{\vee}G)^{\theta_Z} \subseteq ({}^{\vee}H)^{\theta_Z}$. We say that $\Lambda$ is of \emph{type $z$} if $\Lambda^{can}$ is of type $zz(\rho)$. There are various additional conditions which we may choose to impose on $\Lambda$:
\begin{itemize}
    \item[(v)] We say that $\Lambda$ is \emph{standard} if 
    $$\langle d\Lambda^{can}, \alpha^{\vee} \rangle \geq 0, \qquad \forall \alpha \in \Delta_{i\RR}^+.$$
    \item[(vi)] We say that $\Lambda$ is \emph{nonzero} if the following condition is satisfied: let $\alpha$ be a simple root in $\Delta_{i\RR}^+$ such that
    $$\langle d\Lambda^{can}, \alpha^{\vee}\rangle = 0.$$
    Then $\alpha$ is noncompact.
    \item[(vii)] We say that $\Lambda$ is \emph{final} if the following condition is satisfied: let $\alpha$ be a real root such that
    $$\langle d\Lambda^{can}, \alpha^{\vee}\rangle = 0.$$
    Then $\alpha$ does not satisfy the parity condition for $\Lambda$ (see \textnormal{\cite[Definition 11.10]{AdamsBarbaschVogan}}).
\end{itemize}
There is a notion of \emph{equivalence} of limit characters for $G(\RR)$, see \textnormal{\cite[Definition 11.6]{AdamsBarbaschVogan}}. Write $L^z(G(\RR))$ for the set of equivalence classes of final nonzero standard type-$z$ limit characters for $G(\RR)$. 
\end{definition}


\begin{definition}\label{deflimitcharactersallG}
Let $G^{\Gamma}$ be a weak extended group. A \emph{limit character for $G^{\Gamma}$} is a pair $(\delta,\Lambda)$ subject to the following conditions
\begin{itemize}
    \item[(i)] $\delta$ is a strong real form of $G^{\Gamma}$.
    \item[(ii)] $\Lambda$ is a limit character for $G(\RR,\delta)$.
\end{itemize}
Two such pairs $(\delta,\Lambda)$ and $(\delta',\Lambda')$ are \emph{equivalent} if there is an element $g \in G$ such that $g\delta g^{-1} = \delta'$ and $\Ad(g)\Lambda$ is equivalent to $\Lambda'$. Suppose $z$ is a finite-order element in $Z({}^{\vee}G)^{\theta_Z}$. We say that $(\delta,\Lambda)$ is of type $z$ (resp.~final, resp.~nonzero, resp.~standard) if $\Lambda$ is of type $z$ (resp.~final, resp.~nonzero, resp.~standard). Write $L^z(G^{\Gamma})$ for the set of equivalence classes of final nonzero standard type-$z$ limit characters for $G^{\Gamma}$.
\end{definition}

The idea (made precise in Remark \ref{rmkLanglands} below) is that a standard limit character is exactly what is needed to specify a cuspidal parabolic subgroup $Q(\RR) = MAN(\RR)$ of $G(\RR)$, a limit of discrete series representation $\pi_M$ of $M$, and a character $\nu$ of $A$. An induced representation like $\Ind_{MAN(\RR)}^{G(\RR)}(\pi_M \otimes \nu \otimes 1)$ is the starting point for the original Langlands classification. The {\em nonzero} condition is the same as the requirement that $\pi_M \ne 0$; the {\em final} condition eliminates representations like the reducible unitary principal series of $SL_2(\RR)$.

We will now explain how to attach representations to standard limit characters. Suppose $\Lambda=(H(\RR),\Lambda^{can},\Delta_{i\RR}^+, \Delta_{\RR}^+)$ is a standard type-$z$ limit character for $G(\RR)$. Choose a Cartan involution $\theta$ of $G(\RR)$ which preserves $H(\RR)$, and a system of positive roots $\Delta^+ \subset \Delta(\fg,\fh)$ satisfying the following conditions

\begin{itemize}
    \item[(1)] $\Delta_{\RR}^+ \subseteq \Delta^+$.
    \item[(2)] $\Delta_{i\RR}^+ \subseteq \Delta^+$.
    \item[(3)] If $\alpha \in \Delta^+$ and $\langle d\Lambda^{can}, \alpha^{\vee}\rangle \in \ZZ$, then
    $$\mathrm{Re}\langle \theta d\Lambda^{can} - d\Lambda^{can}, \alpha^{\vee}\rangle \geq 0.$$
    \item[(4)] $\Delta^+ \setminus \Delta_{i\RR}^+$ is $-\theta$-stable.
\end{itemize}
Let $\fb = \fh \oplus \fn$ be the Borel subalgebra of $\fg$ corresponding to $\Delta^+$. Form the modules
$$L_i\mathcal{R}^{\fg}_{\fb}(\Lambda^{can}) \in M^z(\fg,K), \qquad i \geq 0.$$
as in Section \ref{sec:parabolicinduction}. Then $L_i\mathcal{R}^{\fg}_{\fb}(\Lambda^{can}) =0$ unless $i=\dim(\fn \cap \fk)$ (this is \cite[Lemma 8.19]{AdamsVogan}; It follows easily from the transitivity of induction and Propositions \ref{prop:propsofrealinduction} and \ref{prop:weaklyfaircohomological}). The \emph{standard limit representation} attached to $\Lambda$ is the type-$z$ projective $(\fg,K)$-module
$$M(\Lambda):=L_{\dim(\fn \cap \fk)}\mathcal{R}^{\fg}_{\fb}(\Lambda^{can}) \in M^z(\fg,K).$$
Let $\pi(\Lambda)$ denote the co-socle of $M(\Lambda)$, i.e. the largest completely reducible quotient of $M(\Lambda)$. We note that $M(\Lambda)$ (and hence $\pi(\Lambda)$) is independent of the choice of $\Delta^+$ (this is \cite[Proposition 8.20]{AdamsVogan}). Here is one version of the Langlands classification of irreducible representations.

\begin{theorem}[{\cite [Theorem 11.14]{AdamsBarbaschVogan}}]\label{thm:LLC}
Let $G^{\Gamma}$ be a weak extended group and let $G(\RR)$ be a real form of $G$ in the inner class defined by $G^{\Gamma}$. Suppose $z$ is a finite-order element in $Z({}^{\vee}G)^{\theta_Z}$. Then the following are true:
\begin{itemize}
    \item[(i)] If $\Lambda \in L^z(G(\RR))$ is a final nonzero type-$z$ standard limit character, then the corresponding representation $\pi(\Lambda)$ is irreducible.
    \item[(ii)] The map $\Lambda \mapsto \pi(\Lambda)$ defines a bijection
    $$\pi\colon L^z(G(\RR)) \xrightarrow{\sim} \Pi^z(G(\RR))$$
    from equivalence classes of final nonzero type-$z$ standard limit characters for $G(\RR)$ to equivalence classes of irreducible type-$z$ projective representations of $G(\RR)$.
    \item[(iii)] If $\Lambda$ is a standard type-$z$ limit character (not necessarily final or nonzero), then there is a finite (possibly empty) set 
    $$\fin(\Lambda)\subset L^z(G(\RR))$$
    such that
    $$M(\Lambda) = \bigoplus_{\Lambda' \in \fin(\Lambda)} M(\Lambda') \quad \text{and} \quad \pi(\Lambda) = \bigoplus_{\Lambda' \in \fin(\Lambda)} \pi(\Lambda').$$
    \item[(iv)] The bijections in (ii) induce a bijection 
    $$\pi\colon L^z(G^{\Gamma}) \xrightarrow{\sim} \Pi^z(G^{\Gamma}), \qquad \pi(\delta, \Lambda) = (\delta, \pi(\Lambda))$$
    from equivalence classes of final nonzero type-$z$ standard limit characters for $G$ to equivalence classes of irreducible type-$z$ projective representations of strong real forms of $G$.
\end{itemize}
\end{theorem}

\begin{rmk}\label{rmkLanglands}
It is helpful to recall the relationship between construction described above and Langlands' construction in \cite{Langlands1979}. Suppose $\Lambda$ is a final nonzero standard type-$1$ limit character for $G(\RR)$. Let $L \subset G$ denote the Levi subgroup corresponding to the imaginary roots in $G$. Let $Q=LU$ denote the (unique) parabolic subgroup containing the Borel subgroup $B$ constructed above. Since the roots in $\fu$ are stable under $-\theta$, $Q$ is defined over $\RR$; $Q(\RR)$ is the cuspidal parabolic mentioned after Definition \ref{deflimitcharactersallG}, and $L(\RR) = MA$. By the transitivity of induction
\begin{align*}
    \mathcal{R}^{\fg}_{\fb} (\Lambda^{can}) &\simeq \mathcal{R}^{\fg}_{\fq}(\mathcal{R}^{\fl}_{\fl \cap \fb} \Lambda^{can})\\
    &\simeq \mathcal{R}^{\fg}_{\fq}((\mathcal{R}^{\fl}_{\fl \cap \fb} (\Lambda^{can} \otimes \rho_{\fu}/|\rho_{\fu}|) \otimes |\rho_{\fu}|/\rho_{\fu})\\
    &= \Ind^{G(\RR)}_{Q(\RR)}(\mathcal{R}^{\fl}_{\fl \cap \fb} (\Lambda^{can} \otimes \rho_{\fu}/|\rho_{\fu}|))\\
\end{align*}
The inner module $\mathcal{R}^{\fl}_{\fl \cap \fb} (\Lambda^{can} \otimes \rho_{\fu}/|\rho_{\fu}|)$ is the relative limit of discrete series representation of $L(\RR)$ with Harish-Chandra parameter $(\Delta_{i\RR}^+,\Lambda^{can} \otimes \rho_{\fu}/|\rho_{\fu}|)$.
\end{rmk}

\subsection{Coherent families}\label{sec:coherentcontinuation}

Fix all of the notation of Section \ref{sec:parabolicinduction}. Choose a maximal torus $H \subset G$ and let $X = X^*(H) \subset \fh^*$ denote the character lattice. If $\lambda \in \fh^*$, let $KM^z_{\lambda}(\fg,K)$ denote the Grothendieck group of finite-length projective type-$z$ $(\fg,K)$-modules of infinitesimal character $\lambda$.

\begin{definition}[\cite{DufloPolynomes}, see also {\cite[Definition 7.2.5]{Vogan1981}}]
Let $\lambda \in \fh^*$. A \emph{coherent family} of projective type-$z$ $(\fg,K)$-modules based on $\lambda + X$ is a function
$$\Theta\colon \lambda+X \to KM^z(\fg,K)$$
such that
\begin{itemize}
    \item[(i)] For each $\lambda' \in \lambda+X$, $\Theta(\lambda') \in KM^z_{\lambda'}(\fg,K)$, and
    \item[(ii)] For each $\lambda' \in \lambda+X$ and finite-dimensional algebraic $G$-representation $F$, there is an identity in $KM^z(\fg,K)$
    $$\Theta(\lambda') \otimes F = \sum_{\mu \in X} m(\mu,F) \Theta(\mu+\lambda')$$
    where $m(\mu,F)$ denotes the multiplicity of $\mu$ in $F$.
\end{itemize}
\end{definition}


We will need several basic results about coherent families. The first is a uniqueness theorem.

\begin{theorem}[{\cite[Theorem 7.2.27]{Vogan1981}}]\label{thm:coherentunique}
Let $\lambda \in \fh^*$ and let $M \in KM^z_{\lambda}(\fg,K)$. Then there is a coherent family 
$$\Theta\colon \lambda+X \to KM^z(\fg,K)$$
such that $\Theta(\lambda)=M$. If $\lambda$ is regular, then $\Theta$ is unique.
\end{theorem}

This allows us to define

\begin{definition}\label{def:coherent}
Suppose $\lambda \in \fh^*$ is regular and let $\lambda' \in \lambda + X$. Then
$$T(\lambda \to \lambda')\colon KM_{\lambda}^z(\fg,K) \to KM^z_{\lambda'}(\fg,K)$$
is the $\ZZ$-linear map defined by
$$T(\lambda \to \lambda')M = \Theta^M(\lambda'), \qquad M \in KM^z_{\lambda}(\fg,K),$$
where $\Theta^M\colon \lambda + X \to KM^z(\fg,K)$ is the unique coherent family such that $\Theta^M(\lambda)=M$.
\end{definition}

The second fact we will need is that coherent families are preserved under parabolic induction. Fix the notation of Section \ref{sec:parabolicinduction}, i.e. $\fq$, $\fl$, $R^{\fg}_{\fq}$, and so on.

\begin{prop}[{\cite[Lemma 7.2.9]{Vogan1981}}]\label{prop:coherentinduction}
Let $\Theta_0\colon \lambda+X\to KM^{zz(\rho_{\fu})}(\fl,K \cap L)$ be a coherent family of projective type-$zz(\rho_{\fu})$ $(\fl,K \cap L)$-modules. Then $R^{\fg}_{\fq} \circ \Theta_0$ is a coherent family of projective type-$z$ $(\fg,K)$-modules.
\end{prop}


%

\subsection{Action of the Weyl group}\label{sec:cohcontsingular}

Choose a maximal torus $H \subset G$, and let $X = X^*(H) \subset \fh^*$ denote the character lattice. For each $\lambda \in \fh^*$, we obtain a system of \emph{$\lambda$-integral roots}
$$\Delta(\lambda) = \{\alpha \in \Delta(\fg,\fh) \mid \langle \lambda,\alpha^{\vee}\rangle \in \ZZ\}.$$
We denote by $W(\lambda)$ the Weyl group of $\Delta(\lambda)$. Note that
$$W(\lambda) = \{w \in W \mid w\lambda-\lambda \text{ is a sum of roots in } \Delta(\fg,\fh)\}.$$
We also define
$$\Delta_{\lambda} = \{\alpha \in \Delta(\fg,\fh) \mid \langle \lambda,\alpha^{\vee}\rangle =0\}, \qquad \Delta_{\geq 0}(\lambda) =\{\alpha \in \Delta(\fg,\fh) \mid \langle \lambda,\alpha^{\vee}\rangle \in \ZZ_{\geq 0}\}.$$
Note that $\Delta_{\geq 0}(\lambda)$ is a parabolic subsystem of $\Delta(\lambda)$ with Levi subsystem $\Delta_{\lambda}$. We denote by $W_{\lambda}$ the Weyl group of $\Delta_{\lambda}$.

If $\lambda$ is regular, it belongs to a unique Weyl chamber. If $\lambda' \in \lambda+X$ is a (possibly singular) weight belonging to this Weyl chamber, then we get a natural positive system for $\Delta_{\lambda'}$:
$$\Delta^+_{\lambda'} = \{\alpha \in \Delta_{\lambda'} \mid \langle \lambda, \alpha^{\vee}\rangle >0\},$$
and hence a set of simple reflections in $W_{\lambda'}$. Note that the positive system $\Delta^+_{\lambda'}$ and the associated simple reflections in $W_{\lambda'}$ depend on the choice of $\lambda$ (more precisely, on its Weyl chamber), although this is not reflected in our notation.

\begin{definition}[{\cite[Definition 7.2.8]{Vogan1981}}]\label{def:coherentcontinuation}
Suppose $\lambda \in \fh^*$ is regular. For each $w \in W(\lambda)$, let 
$$T^{\lambda}_w := T(\lambda \to w^{-1}\lambda),$$
a $\ZZ$-linear automorphism of $KM^z_{\lambda}(\fg,K)$. The map $w \mapsto T^{\lambda}_w$ defines a representation of $W(\lambda)$ on $KM^z_{\lambda}(\fg,K)$, called \emph{coherent continuation} (at $\lambda$).
\end{definition}

Sometimes, when the base point is understood, we will suppress the superscript $\lambda$ from the notation $T^{\lambda}_w$.

\begin{prop}[{\cite[Proposition 7.3.7 and Corollary 7.3.19]{Vogan1981}}]\label{prop:translationtowall}
Let $\lambda \in \fh^*$. Choose $\lambda_r \in \lambda+X$ regular with $\lambda$ contained in its (closed) Weyl chamber. Consider the operator
$$T(\lambda_r \to \lambda)\colon KM^z_{\lambda_r}(\fg,K) \to KM^z_{\lambda}(\fg,K).$$
(Definition \ref{def:coherent}). Then the following are true:
\begin{itemize}
    \item[(i)] If $\pi \in \Pi^z_{\lambda_r}(\fg,K)$, then $T(\lambda_r \to \lambda)\pi$ is either irreducible or $0$.
    \item[(ii)] $T(\lambda_r \to \lambda)$ induces a bijection
    $$\{\pi \in \Pi^z_{\lambda_r}(\fg,K)\mid T(\lambda_r \to \lambda)\pi \neq 0\} \xrightarrow{\sim} \Pi^z_{\lambda}(\fg,K).$$
    \item[(iii)] If $\pi \in \Pi^z_{\lambda_r}(\fg,K)$, then $T(\lambda_r \to \lambda)\pi=0$ if and only if there is a simple reflection $s \in W_{\lambda}$ such that $T^{\lambda_r}_s \pi=-\pi$. 
    \item[(iv)] Let $\pi \in \Pi^z_{\lambda_r}(\fg,K)$ and suppose that $T^{\lambda_r}_s\pi \neq -\pi$ for some simple reflection $s \in W_{\lambda}$. Then $T^{\lambda_r}_s\pi = \pi + U_s\pi$, where $U_s\pi$ is a $\ZZ_{\geq 0}$-linear combination of irreducibles $\pi'$ such that $T^{\lambda_r}_s \pi' = -\pi'$.
\end{itemize}
\end{prop}

From (iii) and (iv) of Proposition \ref{prop:translationtowall} we deduce the following lemma.

\begin{lemma}\label{lem:kers}
Let $\lambda \in \fh^*$. Choose $\lambda_r \in \lambda+X$ regular with $\lambda$ contained in its (closed) Weyl chamber. Then for any simple reflection $s \in W_{\lambda}$, we have
$$\ker{(1+T_s^{\lambda_r})} \subset  \ker{T(\lambda_r \to \lambda)}.$$
\end{lemma}

\begin{proof}
Suppose $M \in KM^z_{\lambda_r}(\fg,K)$ lies in the kernel of $1+T^{\lambda_r}_s$. Write $M$ as a linear combination of irreducibles:
$$M=\sum_{\pi \in \Pi^z_{\lambda_r}(\fg,K)} c_{\pi} \pi, \qquad c_{\pi} \in \ZZ.$$ 
Then by Proposition \ref{prop:translationtowall}(iii) and (iv), we have
$$0= (1+T^{\lambda_r}_s)M = \sum_{(1+T^{\lambda_r}_s)\pi \neq 0} c_{\pi} (1+T^{\lambda_r}_s)\pi = \sum_{(1+T^{\lambda_r}_s) \pi \neq 0} c_{\pi}\pi + \sum_{(1+T^{\lambda_r}_s)\pi \neq 0} c_{\pi}U_s(\pi),$$
where each $U_s(\pi)$ is a linear combination of irreducibles $\pi'$ such that $(1+T^{\lambda_r}_s)\pi'=0$. It follows that
$$0 = \sum_{(1+T^{\lambda_r}_s)\pi \neq 0} c_{\pi}\pi + \text{(linear combination of other irreducibles)}.$$
Hence, $c_{\pi} = 0$ whenever $(1+T^{\lambda_r}_s)\pi \neq 0$ and therefore $T(\lambda_r \to \lambda)M=0$, as desired.
\end{proof}
We will also need the following elementary facts about coherent continuation. 

\begin{lemma}\label{lem:cohfacts}
Let $\lambda_r \in \fh^*$ be regular and let $\lambda \in \lambda_r + X$. Then the following are true:
\begin{itemize}
    \item[(i)] for any regular $\lambda_r' \in \lambda_r + X$, we have
    $$T(\lambda_r \to \lambda) = T(\lambda_r' \to \lambda) \circ T(\lambda_r \to \lambda_r');$$
    \item[(ii)] for any $w \in W(\lambda)$, we have
    $$T(\lambda_r \to \lambda) = T(w\lambda_r \to w\lambda).$$
\end{itemize}
\end{lemma}

\begin{proof}
Given $M \in KM_{\lambda_r}^z(\fg,K)$, let $\Theta^{\lambda_r}$ denote the unique coherent family based on $\lambda_r+X$ such that $\Theta^{\lambda_r}(\lambda_r) = M$. Let $N = \Theta^{\lambda_r}(\lambda_r') = T(\lambda_r \to \lambda_r')M$ and let $\Theta^{\lambda'_r}$ denote the unique coherent family based on $\lambda_r'+X=\lambda_r+X$ such that $\Theta^{\lambda'_r}(\lambda_r') = N$. Since 
$\Theta^{\lambda_r}(\lambda'_r) = N = \Theta^{\lambda'_r}(\lambda_r')$, the uniqueness of coherent families implies that $\Theta^{\lambda_r}=\Theta^{\lambda'_r}$. Thus
$$T(\lambda_r \to \lambda)M = \Theta^{\lambda_r}(\lambda) = \Theta^{\lambda'_r}(\lambda) = T(\lambda_r'\to \lambda)N = T(\lambda_r' \to \lambda)T(\lambda_R \to \lambda_r')M.$$
This proves (i).

For (ii), we argue similarly. For any $\lambda' \in \lambda+X$, let $\Theta^{w\lambda_r}(\lambda') = \Theta^{\lambda_r}(w^{-1}\lambda')$. Since the weights of any finite-dimensional $G$-representation are $W(\lambda)$-invariant, the map $\Theta^{w\lambda_r}\colon \lambda+X \to KM^z(\fg,K)$ is a coherent family. On the other hand, $\Theta^{w\lambda_r}(w\lambda_r) = \Theta^{\lambda_r}(\lambda_r) = M$, so by definition $T(w\lambda_r \to w\lambda)M = \Theta^{w\lambda_r}(w\lambda) = \Theta^{\lambda_r}(\lambda)$. It follows that
$$T(\lambda_r \to \lambda)M = \Theta^{\lambda_r}(\lambda) = T(w\lambda_r \to w\lambda)M.$$
This proves (ii).
\end{proof}

\subsection{$L$-coherent families}

Choose a Levi subgroup $L \subset G$ and a maximal torus $H \subset L$. Write $X^*(L)$ for the lattice of one-dimensional representations of $L$, a sublattice of $X=X^*(H)$. 

\begin{definition}
Let $\lambda \in \fh^*$. An \emph{$L$-coherent family} of projective type-$z$ $(\fg,K)$-modules based on $\lambda+X^*(L)$ is a function
$$\Theta_L\colon \lambda + X^*(L) \to KM^z(\fg,K)$$
which is the restriction to $\lambda+X^*(L)$ of a (classical) coherent family based on $\lambda+X$.
\end{definition}

The following is an analog of Theorem \ref{thm:coherentunique}.

\begin{prop}\label{prop:Lcoherentunique}
Let $\lambda \in \fh^*$ and let $M \in KM_{\lambda}^z(\fg,K)$. Then there is an $L$-coherent family
$$\Theta_L\colon \lambda + X^*(L) \to KM^z(\fg,K)$$
such that $\Theta_L(\lambda)=M$. If $\lambda$ is in the good range for $\fl$ (Definition \ref{def:weaklyfair}), then $\Theta_L$ is unique.
\end{prop}

\begin{proof}
By Theorem \ref{thm:coherentunique}, there is a (classical) coherent family $\Theta\colon \lambda + X \to KM(\fg,K)$ such that $\Theta(\lambda)=M$. Its restriction to $\lambda+X^*(L)$ is an $L$-coherent family with the desired property.

Now suppose that $\lambda$ is in the good range for a parabolic subalgebra $\fq \subset \fg$ with Levi factor $\fl$. Choose any $\lambda' \in \lambda+X^*(L)$ and $\lambda_r \in \lambda+X$ regular with $\lambda$ in its chamber. For the uniqueness claim, it suffices to show that
$$\ker{T(\lambda_r \to \lambda)} \subseteq \ker{T(\lambda_r \to \lambda')}.$$
Write $W_{L,\lambda'} = W_L \cap W_{\lambda'}$, where $W_L \subset W$ is the Weyl group of $L$. Note that $W_{L,\lambda'}$ is a Levi subgroup of $W_{\lambda'}$. Since $\lambda$ is in the good range for $\fq$, we have that $W_{L,\lambda}=W_{\lambda}$, and since $\lambda-\lambda' \in X^*(L)$, we have that $W_{L,\lambda}=W_{L,\lambda'}$. Hence, $W_{\lambda}$ is a Levi subgroup of $W_{\lambda'}$. Now choose an element $w \in W(\lambda)$ such that $\lambda_r$ contains $w\lambda'$ in its (closed) Weyl chamber, and let $s_1,...,s_n$ be the simple reflections in $W_{\lambda}$ (corresponding to the positive system determined by $\lambda_r$). Then $ws_1w^{-1},...,ws_nw^{-1}$ are a set of (possibly nonstandard) simple reflections for $wW_{\lambda}w^{-1}$, a Levi subgroup of $wW_{\lambda'}w^{-1} = W_{w\lambda'}$. To make them standard, choose an element $w_0 \in W_{w\lambda'}$ such that $w_0wW_{\lambda}w^{-1}w_0^{-1}$ is a standard Levi subgroup of $W_{w\lambda'}$ with simple reflections $w_0ws_1w^{-1}w_0^{-1},...,w_0ws_nw^{-1}w_0^{-1}$, and let $w_1=w_0w$. Note that $w_1\lambda' = w_0w\lambda' = w\lambda'$. 

By Proposition \ref{prop:translationtowall}, $\ker{T(\lambda_r \to \lambda)}$ is spanned by irreducibles $\pi$ such that $T_s^{\lambda_r}\pi=-\pi$ for a simple reflection $s \in W_{\lambda}$. Fix some such pair $(\pi,s)$ and let $s' = w_1sw_1^{-1}$, a simple reflection in $W_{w_1\lambda'}$ by the construction of $w_1$. By Lemma \ref{lem:cohfacts}, we have that
$$T(\lambda_r \to \lambda')\pi = T(w_1^{-1}\lambda_r \to \lambda') T(\lambda_r \to w_1^{-1}\lambda_r)\pi = T(\lambda_r \to w_1\lambda') T^{\lambda_r}_{w_1}\pi.$$
On the other hand,
$$T_{s'}^{\lambda_r}T^{\lambda_r}_{w_1}\pi = T^{\lambda_r}_{w_1sw_1^{-1}w_1}\pi = T_{w_1}^{\lambda_r}T_s^{\lambda_r}\pi = -T_{w_1}^{\lambda_r}\pi.$$
Thus, $T_{w_1}^{\lambda_r}\pi$ is contained in $\ker{T(\lambda_r \to w_1\lambda')}$ by Lemma \ref{lem:kers}. And so $T(\lambda_r \to \lambda')\pi = T(\lambda_r \to w_1\lambda')T^{\lambda_r}_{w_1}\pi=0$, as desired. This completes the proof.
\end{proof}

This allows us to define 

\begin{definition}\label{def:Lcoherentcontinuation}
Suppose $\lambda \in \fh^*$ is in the good range for $\fl$ (Definition \ref{def:weaklyfair}) and let $\lambda' \in \lambda + X^*(L)$. Then
$$T_L(\lambda \to \lambda')\colon KM_{\lambda}^z(\fg,K) \to KM^z_{\lambda'}(\fg,K)$$
is the $\ZZ$-linear map defined by
$$T_L(\lambda \to \lambda')M = \Theta^M_L(\lambda'),$$
where $\Theta_L^M\colon \lambda + X^*(L) \to KM^z(\fg,K)$ is the unique $L$-coherent family such that $\Theta_L^M(\lambda)=M$.
\end{definition}

We conclude by mentioning an important construction of $L$-coherent families. Assume that $L$ is $\theta$-stable. Then there is a group homomorphism
$$R^{\fg}_{\fq}\colon KM^{zz(\rho_{\fu})}(\fl,L \cap K) \to KM^z(\fg,K).$$

\begin{prop}\label{prop:cohindLcoherent}
Let $\lambda \in \fh^*$ and let $M_L \in KM_{\lambda}^{zz(\rho_{\fu})}(\fl, L \cap K)$. Then the function
$$\Theta_L\colon \lambda + X^*(L) \to KM^z(\fg,K), \qquad \Theta_L(\lambda+\mu) = R^{\fg}_{\fq}(M_L \otimes \mu)$$
is an $L$-coherent family.
\end{prop}

\begin{proof}
By the existence claim in Theorem \ref{thm:coherentunique}, there is a (classical) coherent family for $L$
$$\Theta_0\colon \lambda + X \to KM^{zz(\rho_{\fu})}(\fl,L \cap K)$$
such that $\Theta_0(\lambda)=M_L$. For any $\mu \in X^*(L)$, we have (by the definition of coherent family)
$$\Theta_0(\lambda+\mu) = \Theta_0(\lambda) \otimes \mu.$$
Now let $\Theta := R^{\fg}_{\fq} \circ \Theta_0$.
By Proposition \ref{prop:coherentinduction}, this is a coherent family for $G$, so its restriction to $\lambda+X^*(L)$ is an $L$-coherent family. But for any $\mu \in X^*(L)$ we have
$$\Theta(\lambda+\mu) = R^{\fg}_{\fq} \Theta_0(\lambda+\mu) = R^{\fg}_{\fq}(M_L \otimes \mu) = \Theta_L(\lambda+\mu).$$
This completes the proof. 
\end{proof}

\section{The Langlands classification via geometric parameters}\label{sec:ABV}

In this section, we will recall the ABV formulation of the Langlands classification and the definition of Arthur packets. We will also define in Section \ref{sec:convolution} some convolution functors on categories of sheaves on spaces of Langlands parameters. We will show in Section \ref{sec:convolutiondualtocoherentcontinuation} that these functors are dual, in a certain precise sense, to coherent continuation.

\subsection{Geometric parameters}\label{sec:ABVparameters}

In this subsection, we will recall the definition of the ABV parameter space. It is helpful to introduce some general notation for reductive groups. Suppose $A$ is a complex reductive group with Lie algebra $\mathfrak{a}$ and let $\lambda \in \mathfrak{a}$ be a semisimple element. For each $z \in \CC$, let $\fa_z$ denote the $z$-eigenspace of the operator $\ad(\lambda)$. Consider the nilpotent subalgebra
$$\mathfrak{n}(\lambda) = \bigoplus_{n \in \ZZ_{>0}} \fa_n \subset \fa.$$
The \emph{canonical flat through $\lambda$} is the affine subspace of $\fa$ defined by
$$\mathcal{F}(\lambda) = \lambda+\mathfrak{n}(\lambda).$$
If $\OO \subset \fa$ is a semisimple $A$-orbit, we write
$$\mathcal{F}(\OO) = \{\mathcal{F}(\lambda) \mid \lambda \in \OO\}.$$
It is easy to see that the set $\mathcal{F}(\OO)$ is a partition of $\OO$. Moreover, the map $e: \fa \to A$ given by $e(\lambda)=\exp(2\pi i \lambda)$ is constant on canonical flats. So if $\Lambda$ is a canonical flat, we can write $e(\Lambda)$ without ambiguity.

\begin{definition}
Let ${}^{\vee}G^{\Gamma}$ be a weak $E$-group. A \emph{geometric parameter} for ${}^{\vee}G^{\Gamma}$ is a pair $(y,\Lambda)$ consisting of a semisimple element $y \in {}^{\vee}G^{\Gamma} - {}^{\vee}G$ and a canonical flat $\Lambda \subset {}^{\vee}\fg$ such that $y^2 = e(\Lambda)$.
\end{definition}

The set of geometric parameters for ${}^{\vee}G^{\Gamma}$ is denoted $X({}^{\vee}G^{\Gamma})$. Note that ${}^{\vee}G$ acts by conjugation on $X({}^{\vee}G^{\Gamma})$. Two geometric parameters are said to be \emph{equivalent} if they are conjugate under this action.

The relation between geometric and classical Langlands parameters (Definition \ref{def:Langlandsparam}) is described in the following proposition.

\begin{prop}[{\cite[Corollary 5.9]{AdamsBarbaschVogan}}]\label{prop:classicaltogeometric}
For any Langlands parameter $\varphi \in P({}^{\vee}G^{\Gamma})$, there is a (unique) semisimple element $\lambda = \lambda(\varphi) \in {}^{\vee}\fg$ (called the `infinitesimal character' of $\varphi$) such that
$$\varphi(\exp(t)) = \exp(t\lambda + \bar{t}\Ad(\varphi(j))\lambda), \qquad t \in \CC.$$
The ${}^{\vee}G$-equivariant map 
$$P({}^{\vee}G^{\Gamma}) \to X({}^{\vee}G^{\Gamma}), \qquad \varphi \mapsto (\exp(\pi i \lambda)\varphi(j), \mathcal{F}(\lambda)).$$
induces a bijection on ${}^{\vee}G$-orbits
$$\Phi({}^{\vee}G^{\Gamma}) = P({}^{\vee}G^{\Gamma})/{}^{\vee}G \xrightarrow{\sim} X({}^{\vee}G^{\Gamma})/{}^{\vee}G.$$
\end{prop}

Next, we recall some facts from \cite[Section 6]{AdamsBarbaschVogan} about the local geometry of $X({}^{\vee}G^{\Gamma})$. Fix a semisimple ${}^{\vee}G$-orbit $\OO \subset {}^{\vee}\fg$, and consider the sets
$$X({}^{\vee}G^{\Gamma},\OO) = \{(y,\Lambda) \in X({}^{\vee}G^{\Gamma}) \mid \Lambda \in \mathcal{F}(\OO)\}, \qquad Y(\OO) = \{y \in {}^{\vee}G^{\Gamma} - {}^{\vee}G \mid y^2 \in e(\OO)\}.$$
There are ${}^{\vee}G$-equivariant surjections
$$(\bullet)^2\colon Y(\OO) \twoheadrightarrow e(\OO), \qquad e\colon \mathcal{F}(\OO) \twoheadrightarrow e(\OO),$$
and hence a ${}^{\vee}G$-equivariant isomorphism
$$X({}^{\vee}G^{\Gamma},\OO) \simeq Y(\OO) \times_{e(\OO)} \mathcal{F}(\OO).$$
By \cite[Proposition 6.13]{AdamsBarbaschVogan}, the set $Y(\OO)$ decomposes into finitely many ${}^{\vee}G$-orbits, denoted
$$Y_1,\ Y_2, ..., \ Y_k \subset Y(\OO).$$
(of course, the integer $k$ and the sets $Y_i$ depend on the choice of $\OO$, even though this is not reflected in the notation). Fix one such ${}^{\vee}G$-orbit $Y_i \subset Y(\OO)$, and consider the subset
\begin{equation}\label{eq:XYO}X({}^{\vee}G^{\Gamma},\OO,Y_i)  = \{(y_i,\Lambda) \in X({}^{\vee}G^{\Gamma}) \mid y_i \in Y_i, \ \Lambda \in \mathcal{F}(\OO)\} \subset X({}^{\vee}G^{\Gamma},\OO).\end{equation}
There is a ${}^{\vee}G$-equivariant surjection
$$ (\bullet)^2\colon Y_i \twoheadrightarrow e(\OO),$$
and hence a ${}^{\vee}G$-equivariant isomorphism
$$X({}^{\vee}G^{\Gamma},\OO,Y_i) \simeq Y_i \times_{e(\OO)} \mathcal{F}(\OO).$$
Now choose an element $(y_i, \Lambda) \in X({}^{\vee}G^{\Gamma},\OO,Y_i)$. Define the following subgroups of ${}^{\vee}G$:
$${}^{\vee}R = Z_{{}^{\vee}G}(e(\Lambda))^0, \qquad 
{}^{\vee}K_i = Z_{{}^{\vee}R}(y_i), 
\qquad
{}^{\vee}P = \mathrm{Stab}_{^{\vee}R}(\Lambda).$$
Note that ${}^{\vee}R$ is (the identity component of) a pseudo-Levi subgroup of ${}^{\vee}G$ and ${}^{\vee}K_i$ (resp.~${}^{\vee}P$) is a symmetric (resp.~parabolic) subgroup of ${}^{\vee}R$. Sometimes we will write
$${}^{\vee}R = {}^{\vee}R(\Lambda), \qquad {}^{\vee}K_i = {}^{\vee}K(y_i), \qquad {}^{\vee}P={}^{\vee}P(\Lambda)$$
to indicate the dependence of these subgroups on the geometric parameter $(y_i,\Lambda)$.

There are ${}^{\vee}G$-equivariant isomorphisms
$$Y \simeq {}^{\vee}G/{}^{\vee}K_i, \qquad \mathcal{F}(\OO) \simeq {}^{\vee}G/{}^{\vee}P, \qquad e(\OO) \simeq {}^{\vee}G/{}^{\vee}R$$
and the maps $(\bullet)^2\colon Y_i \to e(\OO)$ and $e\colon \mathcal{F}(\OO) \twoheadrightarrow e(\OO)$ correspond to the inclusions ${}^{\vee}K_i \subset {}^{\vee}R$ and ${}^{\vee}P \subset {}^{\vee}R$. Thus
\begin{align*}
X({}^{\vee}G^{\Gamma},\OO,Y_i) &\simeq ({}^{\vee}G/{}^{\vee}K_i) \times_{({}^{\vee}G/{}^{\vee}R)} ({}^{\vee}G/{}^{\vee}P)\\
&\simeq {}^{\vee}G \times_{^{\vee}K_i} ({}^{\vee}R/{}^{\vee}P)
\end{align*}
In particular, there is an equivalence of categories (restrition to the fiber)
\begin{equation}\label{eq:equivalence}
D^bC^{{}^{\vee}G^{alg}}X({}^{\vee}G^{\Gamma},\OO,Y_i) \xrightarrow{\sim} D^bC^{{}^{\vee}K_i^{alg}}({}^{\vee}R/{}^{\vee}P)
\end{equation}

\subsection{Complete geometric parameters}

Let ${}^{\vee}G^{\Gamma}$ be a weak $E$-group. Fix a semisimple ${}^{\vee}G$-orbit $\OO \subset {}^{\vee}\fg$ and a ${}^{\vee}G$-orbit $Y_i \subset Y(\OO)$. Form the algebraic variety $X({}^{\vee}G^{\Gamma},\OO,Y_i)$ as in (\ref{eq:XYO}), and let ${}^{\vee}G^{alg}$ denote the algebraic universal covering group of ${}^{\vee}G$, i.e., the projective limit of all finite coverings of ${}^{\vee}G$.

\begin{definition}[{\cite[Definition 7.1]{AdamsBarbaschVogan}}]\label{def:completeparam}
A \emph{complete geometric parameter} for $({}^{\vee}G^{\Gamma},\OO,Y_i)$ is a pair $(x,\tau)$ consisting of a geometric parameter $x \in X({}^{\vee}G^{\Gamma},\OO,Y_i)$ and an irreducible representation $\tau$ of the profinite group
$$A^{loc}_x := \mathrm{Com}(\mathrm{Stab}_{{}^{\vee}G^{alg}}(x)).$$
\end{definition}
Note that ${}^{\vee}G^{alg}$ acts by conjugation on the set of complete geometric parameters for $({}^{\vee}G^{\Gamma},\OO,Y_i)$. Two complete geometric parameters are said to be \emph{equivalent} if they are conjugate under this action. Write $\Xi({}^{\vee}G^{\Gamma},\OO,Y_i)$ for the set of equivalence classes of complete geometric parameters for $({}^{\vee}G^{\Gamma},\OO,Y_i)$. Note that an element of $\Xi({}^{\vee}G^{\Gamma},\OO,Y_i)$ can be regarded as a pair $(S,\tau)$ consisting of a ${}^{\vee}G$-orbit $S \subset X({}^{\vee}G^{\Gamma},\OO,Y_i)$ (i.e. an equivalence class of geometric parameters) and an irreducible ${}^{\vee}G^{alg}$-equivariant local system $\tau$ on $S$.

Write 
\begin{align*}C^{{}^{\vee}G^{alg}}X({}^{\vee}G^{\Gamma},\OO,Y_i) &= \text{category of ${}^{\vee}G^{alg}$-equivariant constructible sheaves on $X({}^{\vee}G^{\Gamma},\OO,Y_i)$}\\
P^{{}^{\vee}G^{alg}}X({}^{\vee}G^{\Gamma},\OO,Y_i) &= \text{category of ${}^{\vee}G^{alg}$-equivariant perverse sheaves on $X({}^{\vee}G^{\Gamma},\OO,Y_i)$}
\end{align*}
Denote by $K^{{}^{\vee}G^{alg}}X({}^{\vee}G^{\Gamma},\OO,Y_i)$ the shared Grothendieck group of these two categories. To each complete geometric parameter $(S,\tau) \in \Xi({}^{\vee}G^{\Gamma},\OO,Y_i)$, we can associate two natural extensions: 
\begin{align*}
\mu(S,\tau) &= j_!\tau \in C^{{}^{\vee}G^{alg}}X({}^{\vee}G^{\Gamma},\OO,Y_i)\\
P(s,\tau) &= j_{!*}\tau[\dim(S)] \in P^{{}^{\vee}G^{alg}}X({}^{\vee}G^{\Gamma},\OO,Y_i)
\end{align*}
(here, $j$ denotes the inclusion of $S$ in its closure and $j_!$ (resp.~$j_{!*}$) denotes the extension-byzero (resp.~intermediate extension)). For details, we refer the reader to the discussion following \cite[Theorem 7.9]{AdamsBarbaschVogan}. The sets
$$\{\mu(S,\tau) \mid (S,\tau) \in \Xi({}^{\vee}G^{\Gamma},\OO,Y_i)\}, \qquad \{P(S,\tau) \mid (S,\tau) \in \Xi({}^{\vee}G^{\Gamma},\OO,Y_i)\}$$
form two natural $\ZZ$-bases, in canonical bijection, for $K^{{}^{\vee}G^{alg}}X({}^{\vee}G^{\Gamma},\OO,Y_i)$. 

It is convenient to introduce some analogous notation for $X({}^{\vee}G^{\Gamma},\OO)$:

\begin{align*}
C^{{}^{\vee}G^{alg}}X({}^{\vee}G^{\Gamma},\OO) &:= \bigoplus_{i} C^{{}^{\vee}G^{alg}}X({}^{\vee}G^{\Gamma},\OO,Y_i) \\
P^{{}^{\vee}G^{alg}}X({}^{\vee}G^{\Gamma},\OO) &:= \bigoplus_{i} P^{{}^{\vee}G^{alg}}X({}^{\vee}G^{\Gamma},\OO,Y_i) \\
K^{{}^{\vee}G^{alg}}X({}^{\vee}G^{\Gamma},\OO) &:= \bigoplus_{i} K^{{}^{\vee}G^{alg}}X({}^{\vee}G^{\Gamma},\OO,Y_i)\\
\Xi({}^{\vee}G^{\Gamma},\OO) &:= \bigsqcup_{i} \Xi({}^{\vee}G^{\Gamma},\OO,Y_i)
\end{align*}
\subsection{The Langlands classification via geometric parameters}\label{sec:LLCABV}

In this section, we will recall the statement of the Langlands classification of irreducible representations, as formulated in \cite{AdamsBarbaschVogan}.

\begin{theorem}[{\cite[Theorem 10.4]{AdamsBarbaschVogan}}]\label{thm:LLCABV}
Let $(G^{\Gamma},\mathcal{W})$ be an extended group for $G$ and let $({}^{\vee}G^{\Gamma},\mathcal{S})$ be a corresponding $E$-group with second invariant $z$. Fix an orbit $\OO \subset {}^{\vee}\fg$ of semisimple elements. Then there is a natural bijection
\begin{equation}\label{eq:XiLz}\Xi({}^{\vee}G^{\Gamma},\OO) \xrightarrow{\sim} L^z_{\OO}(G^{\Gamma})\end{equation}
from complete geometric parameters to equivalence classes of nonzero final type-$z$ standard limit characters. Composing with the bijection of Theorem \ref{thm:LLC}(iv), we obtain a further bijection
$$\pi\colon \Xi({}^{\vee}G^{\Gamma},\OO) \xrightarrow{\sim} \Pi^z_{\OO}(G^{\Gamma})$$
from complete geometric parameters to equivalence classes of irreducible type-$z$ representations of strong real forms of $G$.
\end{theorem}

\begin{rmk}
Let $\varphi\colon W_{\RR} \to {}^{\vee}G^{\Gamma}$ be a Langlands parameter (Definition \ref{def:Langlandsparam}). Then the \emph{$L$-packet} corresponding to $\varphi$ is the set of irreducible projective type-$z$ representations
$$\Pi^z(G^{\Gamma})_{\varphi} := \{\pi(x,\tau)\}$$
where $x$ is the geometric parameter corresponding to $\varphi$ (Proposition \ref{prop:classicaltogeometric}) and $\tau$ runs over all irreducible representations of $A_x^{loc}$.
\end{rmk}

For each complete geometric parameter $\xi \in \Xi({}^{\vee}G^{\Gamma},\OO)$, let $M(\xi) \in K\Pi^z_{\OO}(G^{\Gamma})$ denote the standard representation with Langlands quotient $\pi(\xi)$. Note that the set
$$\{M(\xi) \mid \xi \in \Xi({}^{\vee}G^{\Gamma},\OO)\}$$
is a basis for $K\Pi^z_{\OO}(G^{\Gamma})$. Also define $d(\xi) := \dim(S)$, where $\xi = (S,\tau)$. 

\begin{theorem}[{\cite[Theorem 1.24]{AdamsBarbaschVogan}}]\label{thm:Voganduality}
Let $(G^{\Gamma},\mathcal{W})$ be an extended group for $G$ and let $({}^{\vee}G^{\Gamma},\mathcal{S})$ be a corresponding $E$-group with second invariant $z$. Fix an infinitesimal character $\OO \subset {}^{\vee}\fg$. Then there is a unique perfect pairing
$$\langle -, - \rangle\colon K\Pi^z_{\OO}(G^{\Gamma}) \times KX^{{}^{\vee}G^{alg}}({}^{\vee}G^{\Gamma},\OO) \to \ZZ$$
such that
$$\langle M(\xi), \mu(\xi')\rangle = e(\xi) \delta_{\xi,\xi'}, \qquad \xi, \xi' \in \Xi({}^{\vee}G^{\Gamma})$$
and
$$\langle \pi(\xi), P(\xi')\rangle = e(\xi) (-1)^{d(\xi)} \delta_{\xi,\xi'}, \qquad \xi, \xi' \in \Xi({}^{\vee}G^{\Gamma}). $$
Here $e(\xi) \in \{\pm 1\}$ is the Kottwitz sign associated to the real form of $G$ represented by $\pi(\xi)$, see \textnormal{\cite{Kottwitzsign}} or \textnormal{\cite[Definition 15.8]{AdamsBarbaschVogan}}.
\end{theorem}

Since $\{M(\xi)\}$ (resp.~$\{\mu(\xi)\}$) is a basis for $K\Pi^z_{\OO}(G^{\Gamma})$ (resp.~$K^{{}^{\vee}G^{alg}}X({}^{\vee}G^{\Gamma},\OO)$), the uniqueness of the pairing defined by {\em one} of the two properties is obvious. The content of the theorem is that the pairings defined by both conditions coincide.

\subsection{Langlands functoriality}\label{subsec:ABVfunctoriality}

In this section, we will review the general formalism of Langlands functoriality, as developed in \cite{AdamsBarbaschVogan}. Suppose $(G^{\Gamma},\mathcal{W})$ and $(H^{\Gamma},\mathcal{W}_H)$ are extended groups. Let $({}^{\vee}G^{\Gamma},\mathcal{S})$ and $({}^{\vee}H^{\Gamma},\mathcal{S}_H)$ be corresponding $E$-groups with second invariants $z$ and $z_H$, respectively. Suppose we are given an $L$-homomorphism $\epsilon\colon {}^{\vee}H^{\Gamma} \to {}^{\vee}G^{\Gamma}$. Then there is a natural map
\begin{equation}\label{eq:epsilon0}\epsilon\colon X({}^{\vee}H^{\Gamma}) \to X({}^{\vee}G^{\Gamma}), \qquad \epsilon(y_H,\Lambda_H) = (\epsilon(y),\epsilon(\Lambda_H)).\end{equation}
Here, $\epsilon(\Lambda_H)$ denotes the unique canonical flat in ${}^{\vee}\fg$ containing $d\epsilon(\Lambda_H)$. This map is clearly ${}^{\vee}H$-equivariant. Now fix
$$\OO_H = \text{semisimple ${}^{\vee}H$-orbit on ${}^{\vee}\fh$}, \qquad Y_H = \text{semisimple ${}^{\vee}H$-conjugacy class in ${}^{\vee}H^{\Gamma}\setminus{}^{\vee}H$}$$
so that $Y_H^2 = e(\OO_H)$. Let $\OO = {}^{\vee}G \cdot \OO_H$ and $Y := {}^{\vee}G \cdot Y_H$. Then (\ref{eq:epsilon0}) restricts to a closed ${}^{\vee}H$-equivariant morphism of algebraic varieties
\begin{equation}\label{eq:epsilon2}\epsilon\colon X({}^{\vee}H^{\Gamma},Y_H,\OO_H) \to X({}^{\vee}G^{\Gamma},Y,\OO).\end{equation}

Choose an element $(y_H,\Lambda_H) \in X({}^{\vee}H^{\Gamma},Y_H,\OO_H)$ and let $(y,\Lambda) = \epsilon(y_H,\Lambda_H) \in X({}^{\vee}G^{\Gamma},Y,\OO)$. Define
$${}^{\vee}R_H = Z_{{}^{\vee}H}(e(\Lambda_H))^0, \quad {}^{\vee}K_H = Z_{{}^{\vee}H}(y_H), \quad {}^{\vee}P_H = \mathrm{Stab}_{^{\vee}H}(\Lambda_H)$$
and
$${}^{\vee}R = Z_{{}^{\vee}G}(e(\Lambda))^0, \quad {}^{\vee}K = Z_{{}^{\vee}G}(\epsilon(y_H)), \quad {}^{\vee}P = \mathrm{Stab}_{^{\vee}G}(\Lambda).$$
As explained in Section \ref{sec:ABVparameters}, there are natural identifications
$$X({}^{\vee}H^{\Gamma},Y_H,\OO_H) \simeq {}^{\vee}H \times_{{}^{\vee}K_H} ({}^{\vee}R_H/{}^{\vee}P_H), \qquad X({}^{\vee}G^{\Gamma},Y,\OO) \simeq {}^{\vee}G \times_{{}^{\vee}K} ({}^{\vee}R/{}^{\vee}P).$$
Under these identifications, (\ref{eq:epsilon2}) corresponds to the ${}^{\vee}H$-equivariant morphism
$$\epsilon\colon {}^{\vee}H \times_{{}^{\vee}K_H} ({}^{\vee}R_H/{}^{\vee}P_H) \to  {}^{\vee}G \times_{{}^{\vee}K} ({}^{\vee}R/{}^{\vee}P)$$
induced by the maps $\epsilon\colon{}^{\vee}H \to {}^{\vee}G$, $\epsilon\colon {}^{\vee}R_H \to {}^{\vee}R$, and so on.

Restriction along $\epsilon$ defines a functor
$$\epsilon^*\colon D^bC^{{}^{\vee}G^{alg}}X({}^{\vee}G^{\Gamma},Y,\OO) \to D^bC^{{}^{\vee}H^{alg}}X({}^{\vee}H^{\Gamma},Y_H,\OO_H)$$
of bounded derived categories of equivariant constructible sheaves, and hence a group homomorphism
$$\epsilon^*\colon K^{{}^{\vee}G^{alg}}X(^{\vee}G^{\Gamma},Y,\OO) \to K^{{}^{\vee}H^{alg}}X(^{\vee}H^{\Gamma},Y_H,\OO_H)$$
of Grothendieck groups. Assembling these homomorphisms over the various $Y_H$ and $\OO_H$, we get a group homomorphism
$$\epsilon^*\colon K^{{}^{\vee}G^{alg}}X(^{\vee}G^{\Gamma}) \to K^{{}^{\vee}H^{alg}}X(^{\vee}H^{\Gamma})$$
Taking transpose, we get a group homomorphism in the opposite direction
\begin{equation}\label{eq:epsilonstar1}\epsilon_*\colon K^{{}^{\vee}H^{alg}}X(^{\vee}H^{\Gamma})^* \to K^{{}^{\vee}G^{alg}}X(^{\vee}G^{\Gamma})^*.\end{equation}
(here, $(\bullet)^*$ stands for $\ZZ$-linear dual). Theorem \ref{thm:Voganduality} provides identifications
$$\overline{K}\Pi^{z_H}(H^{\Gamma}) \simeq  K^{{}^{\vee}H^{alg}}X(^{\vee}H^{\Gamma})^*, \qquad \overline{K}\Pi^z(G^{\Gamma}) \simeq K^{{}^{\vee}G^{alg}}X(^{\vee}G^{\Gamma})^*$$
of these dual spaces with formal virtual representations of strong real forms for $H$ and $G$ respectively. Under these identifications, the map (\ref{eq:epsilonstar1}) corresponds to a group homomorphism
\begin{equation}\label{eq:functoriality}\epsilon_*\colon \overline{K}\Pi^{z_H}(H^{\Gamma}) \to \overline{K}\Pi^z(G^{\Gamma})\end{equation}
from formal virtual representations of strong real forms of $H$ to formal virtual representations of strong real forms of $G$. This final map is called \emph{Langlands functoriality}. It is easy to describe its behavior on standard representations.

If $S_H \subset X({}^{\vee}H^{\Gamma})$ is a ${}^{\vee}H$-orbit (of geometric parameters for $H$) and $S = {}^{\vee}G \cdot \epsilon(S_H)$, there is a natural homomorphism of local component groups
$$A^{loc}_{\epsilon}\colon A^{loc}_{S_H} \to A^{loc}_S,$$
(notation from \ref{def:completeparam}). Thus, for each complete geometric parameter $(S_H,\tau_H) \in \Xi({}^{\vee}H^{\Gamma})$ there is an associated subset $\epsilon_*(S_H,\tau_H) \subset \Xi({}^{\vee}G^{\Gamma})$ of complete geometric parameters for $G$, defined by
\begin{equation}\label{eq:epsilonparameters}\epsilon_*(S_H,\tau_H) := \{({}^{\vee}G \cdot S_H,\tau) \mid \epsilon^*\tau = \tau_H\}\end{equation}

\begin{prop}[{\cite[Proposition 23.7]{AdamsBarbaschVogan}}]\label{prop:liftingstandards}
Let $(S_H, \tau_H) \in \Xi({}^{\vee}H^{\Gamma})$ be a complete geometric parameter for $H$. Recall from Theorem \ref{thm:Voganduality}  the corresponding (extension by zero) constructible sheaf $\mu(S_H,\tau_H)$;  the corresponding standard representation $M(S_H,\tau_H)$; and the Kottwitz sign $e(S_H,\tau_H)$. Suppose that $(S,\tau) \in \epsilon_*(S_H,\tau_H)$ is a complete geometric parameter for $G$. Then
$$
\epsilon^*(\mu(S,\tau)) = \mu(S_H,\tau\circ A^{loc}_{\epsilon})
$$
(Here $\tau\circ A^{loc}_{\epsilon}$ is a one-dimensional character of $A_{S_H}^{loc}$.) Equivalently,
$$\epsilon_* M(S_H,\tau_H) = e(S_H,\tau_H) \sum_{(S,\tau) \in \epsilon_*(S_H,\tau_H)} e(S,\tau) M(S,\tau).$$
\end{prop}

The first formula is the content of {\cite[Proposition 23.7]{AdamsBarbaschVogan}}. The second follows from the first and the definition of duality in Theorem \ref{thm:Voganduality}.

\subsection{Convolution functors}\label{sec:convolution} 

Let ${}^{\vee}G^{\Gamma}$ be a weak $E$-group. Fix a semisimple ${}^{\vee}G$-orbit $\OO \subset {}^{\vee}\fg$ and a ${}^{\vee}G$-orbit $Y \subset Y(\OO)$. Form the algebraic variety $X({}^{\vee}G^{\Gamma},\OO,Y)$ as in (\ref{eq:XYO}).

\begin{definition}
Suppose $\OO' \subset {}^{\vee}\fg$ is a semisimple ${}^{\vee}G$-orbit. A \emph{translation datum} from $\OO$ to $\OO'$ is a pair $(\Lambda,\Lambda')$ such that
\begin{itemize}
    \item[(i)] $\Lambda \subset \OO$ and $\Lambda' \subset \OO'$ are canonical flats, and
    \item[(ii)] $e(\Lambda)=e(\Lambda')$; this condition ensures that ${}^{\vee}R(\Lambda)={}^{\vee}R(\Lambda') =: {}^{\vee}R$. 
\end{itemize}
Two translation data are \emph{equivalent} if they are conjugate under ${}^{\vee}G$.
\end{definition}

\begin{rmk}\label{rmk:translationdatum}
If we have a fixed maximal torus $H \subset G$, then an equivalence class of translation data determines, and is determined by, a $W(G,H)$-orbit of pairs $(\lambda,\lambda')$ in $\fh^*$ such that $\lambda-\lambda' \in X^*(H)$. Given such a pair $(\lambda,\lambda')$, we obtain a translation datum $(\Lambda,\Lambda')$ by identifying $\fh^* = \fh^{\vee}$ and setting $\Lambda = \mathcal{F}(\lambda)$ and $\Lambda'=\mathcal{F}(\lambda')$. Conversely, given a translation datum $(\Lambda, \Lambda')$, we obtain a pair $(\lambda,\lambda')$ by identifying $\fh^* = \fh^{\vee}$ and choosing $\lambda \in \Lambda \cap \fh^*$ and $\lambda' \in \Lambda' \cap \fh^*$. 
\end{rmk}

Now fix a translation datum $(\Lambda,\Lambda')$ from $\OO$ to $\OO'$ and an element $(y,\Lambda) \in X({}^{\vee}G^{\Gamma},\OO,Y)$.  Below we will define a functor
$$I(\Lambda \to \Lambda')\colon K^{{}^{\vee}G^{alg}}X({}^{\vee}G^{\Gamma},\OO,Y) \to K^{{}^{\vee}G^{alg}}X({}^{\vee}G^{\Gamma},\OO',Y).$$
First, form the subgroups
$${}^{\vee}K = {}^{\vee}K(y) \subset {}^{\vee}R, \quad {}^{\vee}P={}^{\vee}P(\Lambda) \subset {}^{\vee}R, \quad {}^{\vee}P' = {}^{\vee}P(\Lambda') \subset {}^{\vee}R$$
as in Section \ref{sec:ABVparameters}. There are natural maps
$$q\colon {}^\vee R/(^\vee P\cap {}^\vee P') \to {}^\vee R/{}^\vee P, \qquad q'\colon {}^\vee R/({}^\vee P\cap {}^\vee P') \to {}^\vee R/{}^\vee P'.$$
Consider the `convolution' functor 
$$I(\Lambda \to \Lambda'):=q'_!q^*\colon D^bC({}^\vee R/^\vee P) \to D^bC({}^\vee R/^\vee P')$$ 
%
(the direct image $q'_!$ preserves constructibility by \cite[Exercise VIII.4]{KashiwaraSchapira}). This functor is compatible with left ${}^{\vee}K^{alg}$-actions. Thus it induces a functor
$$I(\Lambda \to \Lambda')\colon D^bC^{{}^{
\vee}K^{alg}}({}^\vee R/^\vee P) \to D^bC^{{}^{
\vee}K^{alg}}({}^\vee R/^\vee P').$$
Composing with the equivalences (\ref{eq:equivalence}), this induces a functor
$$I(\Lambda \to \Lambda')\colon D^bC^{{}^{\vee}G^{alg}}X({}^{\vee}G^{\Gamma},\OO,Y) \to D^bC^{{}^{\vee}G^{alg}}X({}^{\vee}G^{\Gamma},\OO',Y),$$
and hence a homomorphism (by taking Euler characteristics)
$$I(\Lambda \to \Lambda')\colon K^{{}^{\vee}G^{alg}}X({}^{\vee}G^{\Gamma},\OO,Y) \to K^{{}^{\vee}G^{alg}}X({}^{\vee}G^{\Gamma},\OO',Y).$$
The next lemma establishes a composition property for such homomorphisms. 

\begin{lemma}\label{lem:compositionintertwiners}
Let $\OO$, $\OO'$, and $\OO_r$ be semisimple ${}^{\vee}G$-orbits such that $\OO_r$ regular and let $(\Lambda,\Lambda_r)$ and $(\Lambda_r,\Lambda')$ be translation data from $\OO$ to $\OO_r$ and $\OO_r$ to $\OO'$, respectively. Suppose in addition that 
$${}^{\vee}P_r = {}^{\vee}P(\Lambda_r) \subseteq {}^{\vee}P(\Lambda) = {}^{\vee}P.$$
Then on the level of Grothendieck groups, the following are true:
    $$
    \chi((^\vee P\cap ^\vee P')/(^\vee P_r\cap ^\vee P')) I(\Lambda \to \Lambda')= I(\Lambda_r \to \Lambda')\circ I(\Lambda \to \Lambda_r)$$
    and 
    $$
    \chi(({}^\vee P\cap {}^\vee P')/({}^\vee P_r\cap {}^\vee P')) I(\Lambda' \to \Lambda)= I(\Lambda_r \to \Lambda)\circ I(\Lambda' \to \Lambda_r),$$
    where $\chi$ denotes the Euler characteristic, which is moreover a positive integer. 
\end{lemma}

\begin{proof} We have the following commutative diagram
    \begin{center}
\begin{tikzcd}
    ^\vee R/^\vee P_r \ar[d,"q_r"] & ^\vee R/(^\vee P_r\cap {}^\vee P') \ar[d,"\tilde{q}'_r"]\ar[l,"\tilde{q}_r"] &  \\
    ^\vee R/^\vee P   & ^\vee R/(^\vee P\cap {}^\vee P') \ar[r,"q'"] \ar[l,"q"] & ^\vee R/^\vee P'. 
\end{tikzcd}
\end{center}
Note that by definition $I(\Lambda \to \Lambda_r)=q_r^*$ and $I(\Lambda_r \to \Lambda)=(q_r)_!$. So
$$I(\Lambda_r \to \Lambda')\circ I(\Lambda \to \Lambda_r)\simeq q'_!(\tilde{q}'_r)_!(\tilde{q}_r)^*q_r^*\simeq q'_!(\tilde{q}'_r)_!(\tilde{q}'_r)^*q^*$$
and 
$$I(\Lambda_r \to \Lambda)\circ I(\Lambda' \to \Lambda_r)\simeq (q_r)_!(\tilde{q}_r)_!(\tilde{q}'_r)^*(q')^*\simeq q_!(\tilde{q}'_r)_!(\tilde{q}'_r)^*(q')^*.$$
Moreover, by the projection formula
$$(\tilde{q}'_r)_!(\tilde{q}'_r)^* \simeq (\tilde{q}'_r)_!(\underline{\CC}_{^\vee R/(^\vee P_r\cap ^\vee P')}
)\otimes -.$$
Note that $\tilde{q}'_r$ is a fibration with fiber $(^\vee P\cap ^\vee P')/(^\vee P_r\cap ^\vee P')$, which is itself an affine fibration over a flag variety of $^\vee P\cap ^\vee P'$. Hence, the constant sheaf $\underline{\CC}_{^\vee R/(^\vee P_r\cap ^\vee P')}$ is quasi-isomorphic to its cohomology sheaf, which is concentrated in even degrees and forms a trivial local system over $^\vee R/(^\vee P\cap ^\vee P')$ with stalk $H^*_c((^\vee P\cap ^\vee P')/(^\vee P_r\cap ^\vee P'))$. The claim follows. 
\end{proof}

Now assume that $\OO_r \subset {}^{\vee}\fg$ is a regular semisimple ${}^{\vee}G$-orbit. Choose a maximal torus ${}^{\vee}H \subset {}^{\vee}R$ and a Weyl group element $w \in W(\OO_r) = N_{{}^{\vee}R}({}^{\vee}H)/{}^{\vee}H$. Let $\dot{w}$ be a lift of $w$ to $N_{{}^{\vee}R}({}^{\vee}H)$ and let $\Lambda'=\Ad(\dot{w})\Lambda_r \subset \OO_r$. Then ${}^{\vee}P={}^{\vee}B$ and ${}^{\vee}P' = {}^{\vee}B'$ are Borel subgroups of ${}^{\vee}R$ satisfying ${}^{\vee}B' = \Ad(\dot{w}){}^{\vee}B$. Consider the functor
\begin{equation}\label{eq:convolutionw}
I_w\colon D^bC^{{}^{
\vee}K^{alg}}({}^\vee R/^\vee B) \to D^bC^{{}^{
\vee}K^{alg}}({}^\vee R/^\vee B)
\end{equation}
obtained by composing the convolution functor 
$$I(\Lambda_r \to \Lambda')\colon D^bC^{{}^{
\vee}K^{alg}}({}^\vee R/^\vee B) \to D^bC^{{}^{
\vee}K^{alg}}({}^\vee R/^\vee B')$$ 
with the equivalence 
$$D^bC^{{}^{
\vee}K^{alg}}({}^\vee R/^\vee B') \to D^bC^{{}^{
\vee}K^{alg}}({}^\vee R/^\vee B)$$
obtained by pullback along the isomorphism ${}^{\vee}R/{}^{\vee}B' \xrightarrow{\sim} {}^{\vee}R/{}^{\vee}B$, $r{}^{\vee}B' \mapsto r{}^{\vee}B'\dot{w} = r \dot{w}{}^{\vee}B$. It is well known that the functors $\{I_w \mid w \in W(\OO)\}$ define an action of $W(\OO)$ on $K^{^\vee K^{alg}}(^\vee R/^\vee B)$ and hence on $K^{{}^{\vee}G^{alg}}X({}^{\vee}G^{\Gamma},\OO,Y)$.


\begin{lemma}\label{lem:tothewallandback}
Let $\OO$ and $\OO_r$ be semisimple ${}^{\vee}G$-orbits such that $\OO_r$ is regular and let $(\Lambda,\Lambda_r)$ be a translation datum from $\OO$ to $\OO_r$. Then on the level of Grothendieck groups we have 
$$I(\Lambda\to\Lambda_r) \circ I(\Lambda_r\to\Lambda)=\sum_{w\in W_{\Lambda}} I_w,$$
where $W_{\Lambda}$ is the singular Weyl group of $\Lambda$, a subgroup of $W(\OO)$.
\end{lemma}
\begin{proof}
    Consider the Bruhat decomposition $^\vee P/^\vee B=\bigsqcup_{w\in W_P} \, ^\vee B\dot{w}^\vee B/^\vee B$. Note that $^\vee B\dot{w}^\vee B/^\vee B\simeq ^\vee B/(^\vee B\cap \Ad(\dot{w})^\vee B)$ as $^\vee B$-spaces. Thus the functor $I(\Lambda\to\Lambda_r) \circ I(\Lambda_r\to\Lambda)$ admits a filtration with adjoint quotients $I_w$. The claim follows. 
\end{proof}

\subsection{Convolution is dual to coherent continuation}\label{sec:convolutiondualtocoherentcontinuation}

Let $(G^{\Gamma},\mathcal{W})$ be an extended group for $G$ and let $({}^{\vee}G^{\Gamma},\mathcal{S})$ be a corresponding $E$-group with second invariant $z$. Let $(\Lambda,\Lambda')$ be a translation datum from $\OO$ to $\OO'$. As explained in Section \ref{sec:convolution}, there is a convolution map
\begin{equation}\label{eq:convolution}I(\Lambda \to \Lambda')\colon K^{{}^{\vee}G^{alg}}X({}^{\vee}G^{\Gamma},\OO) \to K^{{}^{\vee}G^{alg}}X({}^{\vee}G^{\Gamma},\OO').\end{equation}
Choose a maximal torus $H \subset G$. Recall, Remark \ref{rmk:translationdatum}, that $(\Lambda,\Lambda')$ corresponds to a pair $(\lambda,\lambda')$ of elements of $\fh^*$, well-defined up to simultaneous $W(\OO)$-conjugacy, such that $\lambda-\lambda' \in X^*(T)$. Let us assume that $H$ belongs to a Levi subgroup $L$ such that

\begin{itemize}
    \item[(GR1)] $\lambda'$ is in the good range for $\fl$ (Definition \ref{def:weaklyfair}), and
    \item[(GR2)] $\lambda'-\lambda \in X^*(L)$.
\end{itemize}
Under these conditions, there is an $L$-coherent continuation map
\begin{equation}\label{eq:Lcoherentcontinuation}T_L(\lambda' \to \lambda)\colon K\Pi_{\lambda'}^z(G^{\Gamma}) \to K\Pi_{\lambda}^z(G^{\Gamma})\end{equation}
see Definition \ref{def:Lcoherentcontinuation}. In this section, we will show that the maps (\ref{eq:convolution}) and (\ref{eq:Lcoherentcontinuation}) are dual under the Vogan duality pairings
\begin{equation}\label{eq:Vogandualities}K\Pi^z_{\lambda}(G^{\Gamma}) \times K^{{}^{\vee}G^{alg}}X({}^{\vee}G^{\Gamma},\OO) \to \ZZ, \qquad K\Pi^z_{\lambda'}(G^{\Gamma}) \times K^{{}^{\vee}G^{alg}}X({}^{\vee}G^{\Gamma},\OO') \to \ZZ,\end{equation}
We begin with some special cases. The simplest case is when both $\OO$ and $\OO'$ are regular. In this case, the duality is well known.

\begin{prop}[{\cite[Proposition 17.16]{AdamsBarbaschVogan}}, see also {\cite[Lemma 14.5]{VoganIC4}}]\label{prop:Heckematch}
Assume $\OO$ and $\OO'$ are regular and let $w \in W(\lambda)$ be the 
relative position of the Borel subgroups $^\vee B= {}^\vee P(\Lambda)$ and $^\vee B'={}^\vee P(\Lambda')$ of ${}^{\vee}R$. Then the operators $T(\lambda' \to \lambda)$ and $I(\Lambda \to \Lambda')$ are dual under the Vogan duality pairings (\ref{eq:Vogandualities}).
\end{prop}

The next simplest case is when there is an inclusion of parabolics ${}^{\vee}P(\Lambda') \subseteq {}^{\vee}P(\Lambda)$.

\begin{lemma}\label{lem:convolutionpreliminary}
In the setting described in the beginning of this section (in particular, assuming conditions (GR1) and (GR2)), the following are true:
\begin{itemize}
    \item[(i)] Suppose that 
    $${}^\vee P' = {}^{\vee}P(\Lambda') \subseteq {}^{\vee}P(\Lambda) = {}^\vee P.$$
    Then the operators $T_L(\lambda' \to \lambda)$ and $I(\Lambda \to \Lambda')$ are dual under the Vogan duality pairings (\ref{eq:Vogandualities}).
 
\item[(ii)]  Suppose in addition that $^\vee P'$ is a Borel. Let $u\in K\Pi^z_{\lambda'}(G^{\Gamma})$ be invariant under $T_w$ for all $w\in W_\lambda$ and let $v\in KX^{{}^{\vee}G^{alg}}({}^{\vee}G^{\Gamma},\OO')$ be arbitrary.
Then 
$$\langle T_L(\lambda'\to\lambda)u, I(\Lambda'\to\Lambda)v\rangle=|W_\lambda|\langle u,v\rangle,$$
where, recall, $W_{\lambda}$ is the singular Weyl group of $\lambda$.
\end{itemize}
\end{lemma}

\begin{proof}
By Proposition \ref{prop:translationtowall}, $T_L(\lambda'\to\lambda)$ takes irreducibles in $K\Pi^z_{\lambda'}(G^{\Gamma})$ to irreducibles or $0$. On the other hand,  $I(\Lambda \to \Lambda')$ is just the pullback along the fiber bundle $X({}^{\vee}G^{\Gamma},\OO') \to X({}^{\vee}G^{\Gamma},\OO)$ (with fiber ${}^{\vee}P/{}^{\vee}P'$). This functor takes irreducible perverse sheaves (shifted by the dimension of their support) to irreducible perverse sheaves (shifted by the dimension of their support). This proves (i).

Using part (i), Lemma \ref{lem:tothewallandback} and Proposition \ref{prop:Heckematch}, we get:
     \begin{align*}
\langle T_L(\lambda'\to\lambda)u, I(\Lambda'\to\Lambda)v\rangle &= \langle u, I(\Lambda\to\Lambda')I(\Lambda'\to\Lambda)v\rangle\\
&=\langle u, \sum_{w\in W_P} I_w v\rangle\\
&=\langle \sum_{w\in W_\lambda} T_w u,  v\rangle\\
&=|W_\lambda|\langle u,v\rangle.
\end{align*}
This proves (ii).
\end{proof}

\begin{theorem}\label{thm:convolutionisdualtocoherentcontinuation}
In the setting described in the beginning of this section (in particular, assuming conditions (GR1) and (GR2)), the operators $T_L(\lambda' \to \lambda)$ and $I(\Lambda \to \Lambda')$ are dual under the Vogan duality pairings (\ref{eq:Vogandualities}). That is, for any $u\in K\Pi^z_{\lambda'}(G^{\Gamma})$ and $v\in KX^{{}^{\vee}G^{alg}}({}^{\vee}G^{\Gamma},\OO)$, we have
$$\langle T_L(\lambda' \to \lambda)u,v\rangle=\langle u,I(\Lambda \to \Lambda')v\rangle.$$
\end{theorem}

\begin{proof}
Choose conjugacy classes $\OO_r, \OO'_r 
\subset {}^{\vee}\fg$ of regular semisimple elements and translation data $(\Lambda,\Lambda_r)$ and $(\Lambda',\Lambda'_r)$ from $\OO$ to $\OO_r$ and $\OO'$ to $\OO'_r$, respectively, such that ${}^{\vee}P_r = {}^{\vee}B$ is a Borel subgroup contained in ${}^{\vee}P$ and ${}^{\vee}P'_r = {}^{\vee}B'$ is a Borel subgroup contained in ${}^{\vee}P$. Also choose $u_r\in K\Pi^z_{\OO'_r}(G^{\Gamma})$, invariant under $\{T_w \mid w \in W_{\lambda'}\}$, such that $u=T_L(\lambda'_r\to\lambda')u_r$ (such a $u_r$ exists by Proposition \ref{prop:translationtowall}). 

Note that
$$\chi((^\vee P\cap ^\vee B')/(^\vee B\cap ^\vee B'))=\chi((^\vee B\cap ^\vee P')/(^\vee B\cap ^\vee B'))=1,$$ since in both cases the numerator and denominator are solvable groups sharing a common maximal torus, and so the quotients are affine spaces. Moreover, 
$$\chi((^\vee P\cap ^\vee P')/(^\vee P\cap ^\vee B'))=\chi((^\vee P\cap ^\vee P')/(^\vee B\cap ^\vee P'))=|W_{\lambda'}|.$$ 
To see this, note that both quotients are affine fibrations over the flag variety of the Levi factor of $^\vee P\cap ^\vee P'$, and conditions (GR1) and (GR2) guarantee that this Levi factor coincides with the Levi factor of $P'$, which has Weyl group $W_{\lambda'}$.

%
%
%
%
%
Now we compute
\begin{align*}
\langle u,I(\Lambda' \to \Lambda)v\rangle = &\frac{1}{|W_{\lambda'}|}\langle T_L(\lambda'_r\to\lambda')u_r,I(\Lambda'_r \to \Lambda')I(\Lambda_r \to \Lambda'_r)I(\Lambda \to \Lambda_r)v\rangle &&\text{(Lemma \ref{lem:compositionintertwiners})}\\
= &\frac{|W_{\lambda'}|}{|W_{\lambda'}|} \langle u_r,I(\Lambda_r \to \Lambda'_r)I(\Lambda \to \Lambda_r)v\rangle &&\text{(Lemma \ref{lem:convolutionpreliminary}(ii))}\\
= &\langle T_L(\lambda'_r \to \lambda_r)u_r,I(\Lambda \to \Lambda_r)v\rangle &&\text{(Proposition \ref{prop:Heckematch})}\\
= &\langle T_L(\lambda_r \to \lambda)T(\lambda'_r \to \lambda_r)u_r,v\rangle && \text{(Lemma \ref{lem:convolutionpreliminary}(i))}\\
= &\langle T_L(\lambda'_r \to \lambda)u_r,v\rangle &&\text{(Lemma \ref{lem:cohfacts})}\\
= &\langle T_L(\lambda' \to \lambda)u,v\rangle &&\text{(Proposition \ref{prop:Lcoherentunique})}.
\end{align*}

\end{proof}


\subsection{Arthur packets}\label{sec:Arthur}

Let $(G^{\Gamma},\mathcal{W})$ be an extended group for $G$ and let $({}^{\vee}G^{\Gamma},\mathcal{S})$ be a corresponding $E$-group with second invariant $z$. Fix an Arthur parameter $\psi\colon W_{\RR} \times SL_2(\CC) \to {}^{\vee}G^{\Gamma}$ (Definition \ref{def:Arthurparam}). Recall that $\psi$ determines a Langlands parameter $\varphi_{\psi}$ (\ref{eq:ArthurtoLanglands}) which in turn determines a point $(y,\Lambda) \in X({}^{\vee}G^{\Gamma})$ (Proposition \ref{prop:classicaltogeometric}). This point can be described a bit more explicitly as follows. 

Let $\psi_1$ denote the restriction of $\psi$ to $W_{\RR}$, a tempered Langlands parameter, and let $(y_1,\lambda_1)$ denote the corresponding pair as in Proposition \ref{prop:classicaltogeometric}. Next, let $\psi_2$ denote the restriction of $\psi$ to $SL_2(\CC)$, an algebraic homomorphism into ${}^{\vee}G$. Define the elements
$$e = d\psi_2\begin{pmatrix} 0 & 1\\0 & 0\end{pmatrix}, \qquad f = d\psi_2\begin{pmatrix} 0 & 0\\1 & 0\end{pmatrix}, \qquad h = d\psi_2\begin{pmatrix} 1 & 0\\0 & -1\end{pmatrix},\qquad y_2 = \psi_2\begin{pmatrix} i & 0\\0 & -i\end{pmatrix}$$
Then the point $(y,\Lambda) \in X({}^{\vee}G^{\Gamma})$ is given by the formulas
$$y = y_1y_2, \qquad \lambda = \lambda_1 + \frac{1}{2}h, \qquad \Lambda = \mathcal{F}(\lambda).$$
Now use $(y,\Lambda)$ to define the subgroups ${}^{\vee}R \subset {}^{\vee}G$, ${}^{\vee}K \subset {}^{\vee}R$, and ${}^{\vee}P \subset {}^{\vee}R$ and the algebraic variety $X({}^{\vee}G^{\Gamma},\OO,Y)$ as in Section \ref{sec:ABVparameters}. Let ${}^{\vee}\theta = \Ad(y)$, an involution of ${}^{\vee}\mathfrak{r}$ with fixed points ${}^{\vee}\fk$. It is easy to see that $e$ belongs to the subspace
$$\mathrm{nil}({}^{\vee}\fp)^{-{}^{\vee}\theta} \subset {}^{\vee}\mathfrak{r}$$

\begin{prop}[{\cite[Proposition 21.3]{AdamsBarbaschVogan}}]\label{prop:Arthuroconormal}
In the setting described above, let $S \subset X({}^{\vee}G^{\Gamma},\OO,Y)$ denote the ${}^{\vee}G$-orbit of $(y,\Lambda)$. Then the co-normal space $T^*_{X({}^{\vee}G^{\Gamma},\OO,Y), (y,\Lambda)}S$ to $S \subset X({}^{\vee}G^{\Gamma},\OO,Y)$ at the point $(y,\Lambda)$ is canonically identified with the subspace
$$({}^{\vee}\mathfrak{r}/{}^{\vee}\fp + {}^{\vee}\fk)^* \subset {}^{\vee}\mathfrak{r}^*.$$
If we fix a non-degenerate ${}^{\vee}G^{\Gamma}$-invariant symmetric bilinear form on ${}^{\vee}\fg$, then the induced isomorphism
$${}^{\vee}\mathfrak{r}^* \xrightarrow{\sim} {}^{\vee}\mathfrak{r}$$
restricts to an isomorphism
$$({}^{\vee}\mathfrak{r}/{}^{\vee}\fp + {}^{\vee}\fk)^* \xrightarrow{\sim} \mathrm{nil}({}^{\vee}\fp)^{-{}^{\vee}\theta}.$$
In this way, the element $e \in \mathfrak{nil}({}^{\vee}\fp)^{-{}^{\vee}\theta}$ corresponds to a vector in $T^*_{X({}^{\vee}G^{\Gamma},\OO,Y), (y,\Lambda)}S$. This vector is generic, in that its ${}^{\vee}G$-orbit is open and dense in $T^*_{X({}^{\vee}G^{\Gamma},\OO,Y), (y,\Lambda)}S$. 
\end{prop}

\begin{definition}[{\cite[Definition 22.4]{AdamsBarbaschVogan}}]\label{def:Arthurcomponentgroup}
The \emph{Arthur component group} of $\psi$ is the pro-finite group
$$A_{\psi}^{mic} := \mathrm{Com}(\mathrm{Stab}_{{}^{\vee}G^{alg}}(\psi)).$$
\end{definition}

\begin{prop}[{\cite[Proposition 22.9]{AdamsBarbaschVogan}}]\label{prop:Arthurcomponent}
In the setting described above, the inclusion
$$\mathrm{Stab}_{{}^{\vee}G^{alg}}(\psi) \subset Z_{{}^{\vee}G^{alg}}(e)$$
induces an isomorphism
$$A_{\psi}^{mic} \xrightarrow{\sim} \mathrm{Com}(Z_{{}^{\vee}G^{alg}}(e)).$$
Moreover, 
$$\mathrm{Com}(Z_{{}^{\vee}K^{alg}}(e)) = \mathrm{Com}(Z_{{}^{\vee}G^{alg}}(e)).$$
Thus, there is a canonical isomorphism 
$$A_{\psi}^{mic} \xrightarrow{\sim} \mathrm{Com}(Z_{{}^{\vee}K^{alg}}(e)).$$
\end{prop}

If $P \in P^{{}^{\vee}G^{alg}}X({}^{\vee}G^{\Gamma},\OO,Y)$, then we can define the \emph{microstalk} of $P$ at the cotangent vector $((y,\Lambda),e) \in T^*X({}^{\vee}G^{\Gamma},\OO,Y)$. This is a finite-dimensional representation of $A_{\psi}^{mic}$, and defines an exact functor
$$\chi^{mic}_{\psi}\colon P^{{}^{\vee}G^{alg}}X({}^{\vee}G^{\Gamma},\OO,Y) \to \Rep(A_{\psi}^{mic}).$$
from the category of equivariant perverse sheaves on $X({}^{\vee}G^{\Gamma},\OO,Y)$ to finite-dimensional representations of $A^{mic}_{\psi}$, see \cite[Theorem 24.8]{AdamsBarbaschVogan}. For any $P \in P^{{}^{\vee}G^{alg}}X({}^{\vee}G^{\Gamma},\OO,Y)$, the cotangent vector $((y,\Lambda),e)$ belongs to the singular support of $P$ if and only if $\chi^{mic}_{\psi}(P) \neq 0$.

\begin{rmk}\label{rmk:vanishingcycles}
Below, we will need an alternative description of the vector space $\chi^{mic}_{\psi}(P)$ in terms of the vanishing cycles functor. Let $F$ be a complex analytic function defined on an open neighborhood of $(y,\Lambda)\in X({}^{\vee}G^{\Gamma},\OO,Y)$ such that 
\begin{itemize}
    \item $F$ vanishes on the $^\vee G$-orbit of $(y,\Lambda)$ , and
    \item $dF|_{(y,\Lambda)} = e$.
\end{itemize}
Then there is a vanishing cycles functor
$$\Phi_F\colon PX({}^{\vee}G^{\Gamma},\OO,Y) \to P(F^{-1}(0)),$$
see e.g. \cite[Section 6]{GinsburgVanishing} and \cite[Section 8.6]{KashiwaraSchapira}. For any $P \in PX({}^{\vee}G^{\Gamma},\OO,Y)$, there is an isomorphism of vector spaces 
$$\chi^{mic}_{\psi}(P) \simeq \Phi_F(P)_{(y,\Lambda)},$$
where $\Phi_F(P)_{(y,\Lambda)}$ denotes the stalk of $\Phi_F(P)$ at $(y,\Lambda)$. In particular, we see that $((y,\Lambda),e) \in \mathrm{SS}(P)$ if and only if $\Phi_F(P)_{(y,\Lambda)} \neq 0$.  
\end{rmk}

For any $s \in A^{mic}_{\psi}$, we get a linear functional
$$m^{mic}_{\psi}(s)\colon K^{{}^{\vee}G^{alg}}X({}^{\vee}G^{\Gamma},\OO,Y) \to \ZZ, \qquad m^{mic}_{\psi}(s)(P) = \mathrm{Tr}(\chi^{mic}_{\psi}(P)(s)).$$

\begin{definition}[{\cite[Definition 22.6]{AdamsBarbaschVogan}}]\label{def:Arthurpacket}
Let $\psi\colon W_{\RR} \times SL_2(\CC) \to {}^{\vee}G^{\Gamma}$ be an Arthur parameter. The \emph{Arthur packet} attached to $\psi$ is the set of irreducible projective type-$z$ representations of strong real forms of $G^{\Gamma}$ given by
$$\Pi^z(G^{\Gamma})^{mic}_{\psi} := \{\pi(\xi) \mid \chi^{mic}_{\psi}(P(\xi)) \neq 0\}.$$
For an element $s \in A^{mic}_{\psi}$, the \emph{$s$-stable sum} associated to $\psi$ is the virtual representation
$$\eta^{mic}_{\psi}(s) := \sum _{\xi \in \Xi({}^{\vee}G^{\Gamma},Y,\OO)} e(\xi) (-1)^{d(\xi)-d(\varphi_{\psi})}m^{mic}_{\psi}(s)(P(\xi))\pi(\xi)$$
in $\overline{K}\Pi^z_{\OO}(G^{\Gamma})$.
\end{definition}

The following is immediate from Theorem \ref{thm:Voganduality}.

\begin{lemma}\label{lem:eta}
Let $\psi\colon W_{\RR} \times SL_2(\CC) \to {}^{\vee}G^{\Gamma}$ be an Arthur parameter and let $s \in A_{\psi}^{mic}$. Then the $s$-stable sum $\eta_{\psi}^{mic}(s) \in \overline{K}\Pi_{\OO}^z(G^{\Gamma})$ corresponds to the functional $m^{mic}_{\psi}(s) \in K^{{}^{\vee}G^{alg}}X({}^{\vee}G^{\Gamma},\OO)^*$ under the Vogan duality pairing
$$\langle -,-\rangle\colon K\Pi^z_{\OO}(G^{\Gamma}) \times K^{{}^{\vee}G^{alg}}X({}^{\vee}G^{\Gamma},\OO) \to \ZZ$$
\end{lemma}

\begin{prop}\label{prop:unipotentimpliesweaklyunipotent}
Let $\psi\colon W_{\RR} \times SL_2(\mathbb{C}) \to {}^{\vee}G^{\Gamma}$ be a unipotent Arthur parameter. Then the members of $\Pi^z(G^{\Gamma})^{mic}_{\psi}$ are weakly unipotent (Definition \ref{def:weaklyunipotent}).
\end{prop}

\begin{proof}
Definition \ref{def:weaklyunipotent} allows us to reduce the case when $\fg$ is semisimple. In this case, by \cite[Corollary 27.13]{AdamsBarbaschVogan}, a representation $\pi \in \Pi^z(G^{\Gamma})$ belongs to a unipotent Arthur packet if and only if
$$\Ann_{U(\fg)}(\pi) = I_{max}(\frac{1}{2}h),$$
where $h$ is the semisimple element of an $\mathfrak{sl}_2$-triple $(e,f,h)$ in ${}^{\vee}\fg$ and $I_{max}(\frac{1}{2}h)$ is the unique maximal ideal in $U(\fg)$ of infinitesimal character $\frac{1}{2}h \in {}^{\vee}\fh/W = \fh^*/W$. By \cite[Corollary 5.18]{BarbaschVogan1985}, all such representations are weakly unipotent.
\end{proof}

\section{Decompositions of general Arthur packets}\label{sec:functoriality}

Let $(G^{\Gamma},\mathcal{W})$ be an extended group and let $({}^{\vee}G^{\Gamma},\mathcal{S})$ be a corresponding $E$-group with second invariant $z$. Fix an Arthur parameter $\psi\colon W_{\RR} \times SL_2(\CC) \to {}^{\vee}G^{\Gamma}$. In this section, we will explain how to realize the Arthur packet $\Pi^z(G^{\Gamma})^{mic}_{\psi}$ via real parabolic and cohomological induction from a unipotent Arthur packet for a Levi subgroup of $G^{\Gamma}$. The structure which facilitates this realization is called a \emph{decomposition} of $\psi$. Formally, this is a sextuple
$$\mathcal{D} = ((M_1^{\Gamma},\mathcal{W}_1),(M_2^{\Gamma},\mathcal{W}_2),({}^{\vee}M_1^{\Gamma},\mathcal{S}_1),({}^{\vee}M_2^{\Gamma},\mathcal{S}_2), Q_1, Q_2)$$
where 
\begin{itemize}
    \item $(M_1^{\Gamma},\mathcal{W}_1)$ and $(M_2^{\Gamma},\mathcal{W}_2)$ are extended groups;
    \item $({}^{\vee}M_1^{\Gamma},\mathcal{S}_1)$ and $({}^{\vee}M_2^{\Gamma},\mathcal{S}_2)$ are corresponding $E$-groups;
    \item $M_2^{\Gamma} \subset M_1^{\Gamma} \subset G^{\Gamma}$ and ${}^{\vee}M_2^{\Gamma} \subset {}^{\vee}M_1^{\Gamma} \subset {}^{\vee}G^{\Gamma}$ are Levi subgroups;
    \item $Q_1 \subset G$ (resp. $Q_2 \subset M_1$) is a real parabolic subgroup (resp. $\theta$-stable parabolic subgroup) with Levi factor $M_1$ (resp. $M_2$),
\end{itemize}
subject to the conditions
\begin{itemize}
    \item[(i)] The image of $\psi$ belongs to ${}^{\vee}M_2^{\Gamma}$. We write $\psi_2: W_{\RR} \times SL_2(\CC) \to {}^{\vee}M_2^{\Gamma}$ and $\psi_1: W_{\RR} \times SL_2(\CC) \to {}^{\vee}M_1^{\Gamma}$for the resulting Arthur parameters;
    \item[(ii)] $\psi_2$ is unipotent, i.e., $\psi_2(\CC^{\times}) \subset Z({}^{\vee}M_2)$;
    \item[(iii)] the infinitesimal character $\lambda:=\lambda(\varphi_{\psi})$ is in the weakly fair range for $\fq_2$ (Definition \ref{def:weaklyfair});
    \item[(iv)] the structures $\mathcal{W}$, $\mathcal{W}_1$ and $\mathcal{W}_2$ (resp. $\mathcal{S}$, $\mathcal{S}_1$, and $\mathcal{S}_2$) are compatible in a certain precise sense which we will not describe here (see Definitions \ref{def:realinductiondatum} and \ref{def:thetastableinductiondatum}).
\end{itemize}

Given such a structure, we can define maps (normalized real parabolic and cohomological induction)
$$R^{\fg}_{\fq_1}\colon K\Pi^{zz(\rho_{\fu_1})}_{\lambda}(M_1^{\Gamma}) \to K\Pi^z_{\lambda}(G^{\Gamma}), \qquad R^{\fm_1}_{\fq_2}\colon K\Pi^{zz(\rho_{\fu_1})z(\rho_{\fu_2})}_{\lambda}(M_2^{\Gamma}) \to K\Pi^{zz(\rho_{\fu_2})}_{\lambda}(M_1^{\Gamma}),$$
(see Section \ref{sec:parabolicinductionextended}) as well as maps (Langlands functoriality)
$$(\epsilon_1)_* \colon K\Pi^{zz(\rho_{\fu_1})}_{\lambda}(M_1^{\Gamma}) \to K\Pi^z_{\lambda}(G^{\Gamma}), \qquad (\epsilon_2)_*\colon K\Pi^{zz(\rho_{\fu_1})z(\rho_{\fu_2})}_{\lambda}(M_2^{\Gamma}) \to K\Pi^{zz(\rho_{\fu_2})}_{\lambda}(M_1^{\Gamma}),$$
(see Section \ref{subsec:ABVfunctoriality}). In section \ref{subsec:realinductiondata}, we show that 
\begin{equation}\label{eq:1}R^{\fg}_{\fq_1} \equiv (\epsilon_1)_*\end{equation}
We also establish an identity in $K\Pi^z(G^{\Gamma})$
\begin{equation}\label{eq:2}\eta^{mic}_{\psi}(1) = (\epsilon_1)_*\eta^{mic}_{\psi_1}(1).\end{equation}
Combining (\ref{eq:1}) and (\ref{eq:2}), we deduce
\begin{equation}\label{eq:3}\eta^{mic}_{\psi}(1) = R^{\fg}_{\fq_1}\eta^{mic}_{\psi_1}(1).\end{equation}
Under an additional condition on the decomposition (namely, that ${}^{\vee}M_1 = Z_{{}^{\vee}G}(\psi(\RR_+))$), we get analogous identities for every $s \in A_{\psi_1}^{mic}$.

The situation for $R^{\fm_1}_{\fq_2}$ is considerably more subtle. In general, there is no simple relation between $R^{\fm_1}_{\fq_2}$ and $(\epsilon_2)_*$. However, in Section \ref{subsec:thetastableinductiondata} we show that
\begin{equation}\label{eq:4}R^{\fm_1}_{\fq_2} \equiv \kappa (\epsilon_2)_* \end{equation}
for an explicitly determined sign $\kappa \in \{\pm 1\}$, provided the infinitesimal character $\lambda$ is in the good range for $\fq_2$. Under the same condition on $\lambda$, we also get identities in $K\Pi^{zz(\rho_{\fu_1})}(M_1^{\Gamma})$
\begin{equation}\label{eq:5}
\eta_{\psi_1}^{mic}(s) = (\epsilon_2)_* \eta^{mic}_{\psi_2}(s), \qquad s \in A^{mic}_{\psi_2}\end{equation}
Combining (\ref{eq:4}) and (\ref{eq:5}), we deduce that
\begin{equation}\label{eq:6}
\eta_{\psi_1}^{mic}(s) = \kappa R^{\fm_1}_{\fq_2} \eta^{mic}_{\psi_2}(s), \qquad s \in A_{\psi_2}^{mic}.\end{equation}
For an arbitrary decomposition (i.e., without the assumption that $\lambda$ is good range), both (\ref{eq:4}) and (\ref{eq:5}) are false. However, miraculously, we find that the identities (\ref{eq:6}) still hold. The proof of this fact relies on a somewhat delicate Lie theory calculation, which is carried out in Appendix \ref{sec:appendix}. 

Combining (\ref{eq:3}) and (\ref{eq:6}), we finally arrive at the following identity:
\begin{equation}\label{eq:7}\eta^{mic}_{\psi}(1) = \kappa R^{\fg}_{\fq_1}R^{\fm_1}_{\fq_2}\eta^{mic}_{\psi_2}(1)\end{equation}
Under the aforementioned condition on the decomposition (i.e., that ${}^{\vee}M_1 = Z_{{}^{\vee}G}(\psi(\RR_+))$), we also have the identities
$$\eta^{mic}_{\psi}(s) = \kappa R^{\fg}_{\fq_1}R^{\fm_1}_{\fq_2}\eta^{mic}_{\psi_2}(s), \qquad s \in A^{mic}_{\psi_2}.$$
Recall that $\eta^{mic}_{\psi}(1)$ (resp. $\eta^{mic}_{\psi_2}(1)$) is a linear combination over $\Pi^z(G^{\Gamma})_{\psi}^{mic}$ (resp. $\Pi^{zz(\rho_{\fu_1})z(\rho_{\fu_2})}(M_2^{\Gamma})_{\psi_2}^{mic}$), each member appearing with nonzero multiplicity. Thus, the identity (\ref{eq:7}) implies the following relation on the level of Arthur packets
$$\Pi^z(G^{\Gamma})_{\psi}^{mic} \subseteq [R^{\fg}_{\fq_1}R^{\fm_1}_{\fq_2}\Pi^{zz(\rho_{\fu_1})z(\rho_{\fu_2})}(M_2^{\Gamma})_{\psi_2}^{mic}],$$
where, for $S \subset K\Pi^z(G^{\Gamma})$, $[S] \subset \Pi^z(G^{\Gamma})$ denotes the set of irreducibles appearing in the elements of $S$ when expanded in the basis of $\Pi^z(G^{\Gamma})$. This is the sense in which a decomposition of $\psi$ gives a realization of $\Pi^z(G^{\Gamma})_{\psi}^{mic}$ via parabolic induction from a unipotent Arthur packet.

\subsection{Induction data: general definitions}\label{subsec:inductiongeneral}

Let $(G^{\Gamma},\mathcal{W})$ be an extended group and let $({}^{\vee}G^{\Gamma},\mathcal{S})$ be a corresponding $E$-group with second invariant $z$.

\begin{definition}\label{def:inductiondatumgeneral}
A \emph{weak induction datum} (for the pair $(G^{\Gamma},{}^{\vee}G^{\Gamma})$) is a triple
$$I=(M^{\Gamma},{}^{\vee}M^{\Gamma},Q)$$
where
\begin{itemize}
    \item $M^{\Gamma}$ is a weak extended group;
    \item ${}^{\vee}M^{\Gamma}$ is a corresponding weak $E$-group;
    \item $M^{\Gamma}$ is a Levi subgroup of $G^{\Gamma}$. This means that there is an inclusion $M^{\Gamma} \subset G^{\Gamma}$, compatible with the projections $M^{\Gamma},G^{\Gamma} \to \Gamma$, such that $M \subset G$ is a Levi subgroup (in the ordinary sense).
    \item $^{\vee}M^{\Gamma}$ is a Levi subgroup of ${}^{\vee}G^{\Gamma}$. This means that there is an inclusion $\epsilon\colon{}^{\vee}M^{\Gamma} \subset {}^{\vee}G^{\Gamma}$, compatible with the projections ${}^{\vee}M^{\Gamma},{}^{\vee}G^{\Gamma} \to \Gamma$, such that ${}^{\vee}M \subset {}^{\vee}G$ is a Levi subgroup (in the ordinary sense).
    \item $Q \subset G$ is a parabolic subgroup with Levi factor $M$. We write $U \subset G$ for its unipotent radical and ${}^{\vee}Q = {}^{\vee}M{}^{\vee}U$ for the dual parabolic subgroup of ${}^{\vee}G$.
\end{itemize}
An \emph{induction datum} (for the pair $((G^{\Gamma},\mathcal{W}),({}^{\vee}G^{\Gamma},\mathcal{S}))$) is a triple
$$\mathcal{I}=((M^{\Gamma},\mathcal{W}_0),({}^{\vee}M^{\Gamma},\mathcal{S}_0),Q)$$
where 
\begin{itemize}
    \item $(M^{\Gamma},\mathcal{W}_0)$ is an extended group; 
    \item $({}^{\vee}M^{\Gamma},\mathcal{S}_0)$ is an $E$-group;
    \item $(M^{\Gamma},{}^{\vee}M^{\Gamma},Q)$ is a weak induction datum (for the pair $(G^{\Gamma},{}^{\vee}G^{\Gamma})$).
\end{itemize}
\end{definition}

Often, we will have a (weak) induction datum $\mathcal{I}$ and an Arthur parameter $\psi_0: W_{\RR} \times SL_2(\CC) \to {}^{\vee}M^{\Gamma}$. In this setting, we can introduce all of the structure defined in Section \ref{subsec:ABVfunctoriality}. We repeat the definitions here in the present notation. The $L$-embedding $\epsilon\colon {}^{\vee}M^{\Gamma} \subset {}^{\vee}G^{\Gamma}$ induces a ${}^{\vee}M$-equivariant embedding of geometric parameter spaces
\begin{equation}\label{eq:epsilon1}\epsilon\colon  X({}^{\vee}M^{\Gamma}) \subseteq X({}^{\vee}G^{\Gamma}) .\end{equation}
Let $(y_0,\Lambda_0) \in X({}^{\vee}M^{\Gamma})$ denote the geometric parameter corresponding to the Langlands parameter $\varphi_{\psi_0}$ (Proposition \ref{prop:classicaltogeometric}), and let 
$$(y,\Lambda)= \epsilon(y_0,\Lambda_0).$$
Also let
$$\OO_0 = {}^{\vee}M \cdot \Lambda_0, \qquad \OO = {}^{\vee}G \cdot \Lambda,$$ 
and 
$$Y_0 = {}^{\vee}M \cdot y_0, \qquad Y = {}^{\vee}G \cdot y.$$
The embedding (\ref{eq:epsilon1}) restricts to a closed ${}^{\vee}M$-equivariant embedding of algebraic varieties
\begin{equation}\label{eq:epsilon11}\epsilon\colon  X({}^{\vee}M^{\Gamma},Y_0,\OO_0) \subseteq X({}^{\vee}G^{\Gamma},Y,\OO).\end{equation}
Form the subgroups ${}^{\vee}R_0 \subset {}^{\vee}M$, ${}^{\vee}R \subseteq {}^{\vee}G$, ${}^{\vee}K_0 \subset {}^{\vee}R_0$,${}^{\vee}K \subseteq {}^{\vee}R$, ${}^{\vee}P_0 \subset {}^{\vee}M$, and ${}^{\vee}P \subseteq {}^{\vee}G$ as in Section \ref{subsec:ABVfunctoriality}. Then the embedding (\ref{eq:epsilon11}) corresponds to the natural map
$$
\epsilon\colon {}^{\vee}M \times_{{}^{\vee}K_0} ({}^{\vee}R_0/{}^{\vee}P_0) \to  {}^{\vee}G \times_{{}^{\vee}K} ({}^{\vee}R/{}^{\vee}P)
$$
induced by the inclusions ${}^{\vee}M \subset {}^{\vee}G$, ${}^{\vee}R_0 \subset {}^{\vee}R$, and so on.

\subsection{Real induction data}\label{subsec:realinductiondata}

Let $(G^{\Gamma},\mathcal{W})$ be an extended group and let $({}^{\vee}G^{\Gamma},\mathcal{S})$ be a corresponding $E$-group with second invariant $z$. 

\begin{definition}\label{def:realinductiondatum}
Let $\mathcal{I}$ be an induction datum for the pair $((G^{\Gamma},\mathcal{W}),({}^{\vee}G^{\Gamma},\mathcal{S}))$ (Definition \ref{def:inductiondatumgeneral}). We say that $\mathcal{I}$ is `real' if the following conditions are satisfied
\begin{itemize}
    \item[$(a_\RR)$] Given $(\delta,N_0,\chi_0) \in \mathcal{W}_0$, there exists $(\delta,N,\chi) \in \mathcal{W}$ such that $(N_0,\chi_0) \subset (N,\chi)$ (this notation means: $N_0 \subset N$ and $\chi_0 = \chi|_{N_0^{\Ad(\delta)}}$).
    \item[$(b_\RR)$] given $({}^{\vee}\delta,{}^{\vee}B_0) \in \mathcal{S}_0$, there exists $({}^{\vee}\delta,{}^{\vee}B) \in \mathcal{S}$ such that ${}^{\vee}B = {}^{\vee}B_0{}^{\vee}U$.
\end{itemize}
\end{definition}

We note that condition $(b_{\RR})$ of Definition \ref{def:realinductiondatum} implies that $Q$ is real (in the sense of Lemma \ref{lem:realisreal}) and that $({}^{\vee}M^{\Gamma},\mathcal{S}_0)$ is of type $zz(\rho_{\fu})$. Thus, we obtain two homomorphisms with the same source and target
$$\epsilon_*, R^{\fg}_{\fq} \colon K\Pi^{zz(\rho_{\fu})}(M^{\Gamma}) \to K\Pi^z(G^{\Gamma}).$$
The homomorphism $\epsilon_*$ is Langlands functoriality (Section \ref{subsec:ABVfunctoriality}). The homomorphism $R^{\fg}_{\fq}$ is (normalized) real parabolic induction (Section \ref{sec:parabolicinductionextended}). Our first task is to show that these homomorphisms coincide.

For the following lemma, we work in the setting of \cite[Chapter 13]{AdamsBarbaschVogan}.

\begin{lemma}\label{lem:basedCartansreal}
Let $\mathcal{I}$ be a real induction datum for the pair $((G^{\Gamma},\mathcal{W}),({}^{\vee}G^{\Gamma},\mathcal{S}))$.
\begin{itemize}
    \item[(i)] Let $(H^{\Gamma},\mathcal{W}_0(H^{\Gamma}),\Delta^+_{\RR},\Delta^+_{i\RR})$ be a based Cartan subgroup of $(M^{\Gamma},\mathcal{W}_0)$ (\cite[Definition 13.5]{AdamsBarbaschVogan}). Then $(H^{\Gamma},\mathcal{W}_0(H^{\Gamma}),\Delta^+_{\RR} \sqcup \Delta_{\RR}(\fu),\Delta^+_{i\RR})$ is a based Cartan subgroup of $(G^{\Gamma},\mathcal{W})$.
    \item[(ii)] Let $({}^{\vee}H^{\Gamma},\mathcal{S}_0({}^{\vee}H^{\Gamma}),{}^{\vee}\Delta^+_{\RR},{}^{\vee}\Delta^+_{i\RR})$ be a based Cartan subgroup of $({}^{\vee}M^{\Gamma},\mathcal{S}_0)$ (\cite[Definition 13.7]{AdamsBarbaschVogan}). Then $({}^{\vee}H^{\Gamma},\mathcal{S}_0({}^{\vee}H^{\Gamma}),{}^{\vee}\Delta^+_{\RR},{}^{\vee}\Delta^+_{i\RR} \sqcup \Delta_{i\RR}({}^{\vee}\fu))$ is a based Cartan subgroup of $({}^{\vee}G^{\Gamma},\mathcal{S})$.
\end{itemize}
\end{lemma}

\begin{proof}
For (i), choose $\delta \in \mathcal{W}_0(H^{\Gamma})$. By the definition of a based Cartan subgroup, there is some $(\delta,N_0,\chi_0) \in \mathcal{W}_0$ and hence by $(a_{\RR})$ some $(\delta,N,\chi) \in \mathcal{W}$ such that $N_0 \subset N$ and $\chi_0 = \chi|_{N_0(\RR)}$. Choose a final limit character $(\delta,\Lambda) = (\delta,(H(\RR),\Lambda^{can},\Delta^+_{i\RR},\Delta^+_{\RR} \sqcup \Delta_{\RR}(\fu))) \in L^{z}(G^{\Gamma})$. Then $(\delta,(H(\RR),\Lambda^{can},\Delta^+_{i\RR},\Delta^+_{\RR})) \in L^{zz(\rho_{\fu})}(M^{\Gamma})$. Write $(\delta,X) \in K\Pi^z(G^{\Gamma})$ and $(\delta,X_0) \in K\Pi^{zz(\rho_{\fu})}(M^{\Gamma})$ for the corresponding standard modules. By the transitivity of induction, we have $X = \mathcal{R}^{\fg}_{\fq}X_0$. On the other hand, since $(H^{\Gamma},\mathcal{W}_0(H^{\Gamma}),\Delta^+_{\RR},\Delta^+_{i\RR})$ is a based Cartan subgroup of $(M^{\Gamma},\mathcal{W}_0)$, we know that $X_0$ has a Whittaker model of type $(N_0,\chi_0)$. Thus, by Proposition \ref{prop:Whittakerinduction}(i), $X=R^{\fg}_{\fq}X_0$ has a Whittaker model of type $(N,\chi)$. This proves that $(H^{\Gamma},\mathcal{W}_0(H^{\Gamma}),\Delta^+_{\RR} \sqcup \Delta_{\RR}(\fu),\Delta^+_{i\RR})$ is a based Cartan subgroup of $(G^{\Gamma},\mathcal{W})$.

For (ii), choose ${}^{\vee}\delta \in \mathcal{S}_0({}^{\vee}H^{\Gamma})$. By the definition of a based Cartan subgroup, there is some $({}^{\vee}\delta,{}^{\vee}B_0) \in \mathcal{S}_0$ and hence by $(b_{\RR})$ some $({}^{\vee}\delta,{}^{\vee}B) \in \mathcal{S}$ such that ${}^{\vee}B={}^{\vee}B_0{}^{\vee}U$. Let ${}^{\vee}K_0 = {}^{\vee}M^{{}^{\vee}\delta}$ and ${}^{\vee}K = {}^{\vee}G^{{}^{\vee}\delta}$. Since $({}^{\vee}H^{\Gamma}, \mathcal{S}_0({}^{\vee}H^{\Gamma}),{}^{\vee}\Delta_{\RR}^+,{}^{\vee}\Delta_{i\RR}^+)$ is a based Cartan subgroup of $({}^{\vee}M^{\Gamma},\mathcal{S}_0)$, we know that ${}^{\vee}\Delta_{i\RR}^+$ is special with respect to ${}^{\vee}B_0$. Thus by Proposition \ref{prop:specialinduction}(i), ${}^{\vee}\Delta_{i\RR}^+ \sqcup \Delta_{i\RR}({}^{\vee}\fu)$ is special with respect to ${}^{\vee}B$. This proves that $({}^{\vee}H^{\Gamma},\mathcal{S}_0({}^{\vee}H^{\Gamma}),{}^{\vee}\Delta_{\RR}^+,{}^{\vee}\Delta_{i\RR}^+ \sqcup \Delta_{i\RR}({}^{\vee}\fu))$ is a based Cartan subgroup of $({}^{\vee}G^{\Gamma},\mathcal{S})$.
\end{proof}

\begin{definition}\label{def:hashparameter}
Suppose $\sigma$ is a real form of $G$ and let $Q=MU \subset G$ be a parabolic subgroup such that $M$ is $\sigma$-stable. If
$$\Lambda =(H(\RR),\Lambda^{can},\Delta_{i\RR}^+,\Delta^+_{\RR}) \in L^{zz(\rho_{\fu})}(M(\RR))$$
we define
$$\Lambda^{\#} := (H(\RR),\Lambda^{can},\Delta_{i\RR}^+ \cup \Delta_{i\RR}(\fu),\Delta^+_{\RR}\cup \Delta_{\RR}(\fu)).$$
This is a (possibly non-standard, non-final) limit character for $G^{\Gamma}$. 
\end{definition}

\begin{prop}\label{prop:estarrealinduction}
Let $\mathcal{I}$ be a real induction datum for the pair $((G^{\Gamma},\mathcal{W}),({}^{\vee}G^{\Gamma},\mathcal{S}))$. Then:
\begin{itemize}
    \item[(i)] For each final limit character $(\delta,\Lambda) \in L^{zz(\rho_{\fu})}(M^{\Gamma})$, we have
$$\epsilon_*M(\delta,\Lambda) = \begin{cases*}
                    M(\delta,\Lambda^{\#})
                     & if  $\delta^2 \in Z(G)$  \\
                     0 & else
                 \end{cases*}$$
In this formula, $(\delta,\Lambda)$ is a (possibly non-final) non-zero standard type-$z$ limit character for $G^{\Gamma}$.
\item[(ii)]  Real parabolic induction $R^{\fg}_{\fq}$ satisfies an identical formula. 
\item[(iii)] Consequently,
$$\epsilon_* \equiv R^{\fg}_{\fq}: K\Pi^{zz(\rho_{\fu})}(M^{\Gamma}) \to K\Pi^z(G^{\Gamma}).$$ 
\end{itemize}
\end{prop}

\begin{proof}
(ii) is immediate from the definition of $R^{\fg}_{\fq}$ (Equation \ref{eq:indextendedgroups}), the construction of standard modules from standard limit characters, and transitivity of induction. (iii) follows from (i) and (ii) since the set $\{M(\delta,\Lambda) \mid (\delta,\Lambda) \in L^{zz(\rho_{\fu})}(M^{\Gamma})\}$ forms a basis for $KM^{zz(\rho_{\fu})}(M^{\Gamma})$. We concentrate on (i). 

Write
$$\alpha_G\colon \Xi({}^{\vee}G^{\Gamma}) \xrightarrow{\sim} L^z(G^{\Gamma})$$
for the bijection of Theorem \ref{thm:LLCABV} with respect to the pair $(G^{\Gamma},\mathcal{W}),({}^{\vee}G^{\Gamma},\mathcal{S})$. Write $\alpha_{M}$ for the corresponding bijection for $(M^{\Gamma},\mathcal{W}_0),({}^{\vee}M^{\Gamma},\mathcal{S}_0)$. Tracing through the construction of $\alpha_G$ in \cite[Chapter 13]{AdamsBarbaschVogan}, and using Lemma \ref{lem:basedCartansreal} to match based Cartan subgroups, we see that
$$\alpha_G(\epsilon_*(\alpha_M^{-1}(\delta, \Lambda))) = \begin{cases*} \{(\delta,\Lambda') \mid \Lambda' \in \fin(\Lambda^{\#})\}
                     & if $\delta^2 \in Z(G)$  \\
                     0 & else
                 \end{cases*}$$
where $\epsilon_*(\alpha_M^{-1}(\delta, \Lambda)) \subset \Xi({}^{\vee}G^{\Gamma})$ is the subset defined in (\ref{eq:epsilonparameters}) and $\fin(\Lambda^{\#})$ is as defined in Theorem \ref{thm:LLC}. So by Proposition \ref{prop:liftingstandards}, we have
$$\epsilon_* M(\delta,\Lambda) =  \begin{cases*}
                e(M(\RR,\delta))e(G(\RR,\delta)) \sum_{\Lambda' \in \fin(\Lambda^{\#})}  M(\delta,\Lambda')
                     & if $\delta^2 \in Z(G)$  \\
                     0 & else
                 \end{cases*}.$$
Since $M(\RR,\delta)$ is the Levi subgroup of a real parabolic subgroup of $G(\RR,\delta)$, we have that  $e(M(\RR,\delta))=e(G(\RR,\delta))$ by \cite[(6) of Corollary on p.~295]{Kottwitzsign}. And by Theorem \ref{thm:LLC}(iii), we have
$$ \sum_{\Lambda' \in \fin(\Lambda^{\#})} M(\delta,\Lambda') = M(\delta,\Lambda^{\#})$$
This completes the proof of (i).

%
%
%
\end{proof}

Our next task is to analyze the behavior of $\epsilon_*$ with respect to Arthur packets.

\begin{lemma}\label{lem:s}
Let $I$ be a weak induction datum for the pair $(G^{\Gamma},{}^{\vee}G^{\Gamma})$ such that $Q \subset G$ is real, and let $\psi_0: W_{\RR} \times SL_2(\CC) \to {}^{\vee}M^{\Gamma}$ be an Arthur parameter. 
Then there is a semisimple element $s$ in the identity component of  $Z_{{}^{\vee}M}(\psi_0)$ such that
$${}^{\vee}M^{\Gamma} = Z_{{}^{\vee}G^{\Gamma}}(s).$$
\end{lemma}

\begin{proof}
For each $t \in \RR$, let $s(t) := \exp(t\rho_{\fu}) \in Z({}^{\vee}M)$. Choose any element ${}^{\vee}\delta \in {}^{\vee}M^{\Gamma} \setminus {}^{\vee}M$ such that ${}^{\vee}\delta^2 \in Z({}^{\vee}G)$. Since $Q$ is real, $\rho_{\fu} \in Z({}^{\vee}\fm)$ is ${}^{\vee}\delta$-fixed. Hence
\begin{equation}\label{eq:11}{}^{\vee}M^{\Gamma} \subseteq Z_{{}^{\vee}G^{\Gamma}}(s(t))\end{equation}
for every $t \in \RR$. Moreover, since ${}^{\vee}M = Z_{{}^{\vee}G}(\rho_{\fu})$, we have
\begin{equation}\label{eq:22}{}^{\vee}M = Z_{{}^{\vee}G}(s(t_0))\end{equation}
for $t_0 \in \RR$ generic. (\ref{eq:11}) and (\ref{eq:22}) imply
$${}^{\vee}M^{\Gamma} = Z_{{}^{\vee}G^{\Gamma}}(s(t_0)).$$
Finally, note that $\{s(t)\}$ is a one parameter subgroup in $Z({}^{\vee}M^{\Gamma}) \subset Z_{{}^{\vee}G}(\psi)$ and hence belongs to the identity component of $Z_{{}^{\vee}G}(\psi)$.
\end{proof}

\begin{prop}\label{prop:realpackets}
Let $\mathcal{I}$ be a real induction datum for the pair $((G^{\Gamma},\mathcal{W}),({}^{\vee}G^{\Gamma},\mathcal{S}))$, and let $\psi_0: W_{\RR} \times SL_2(\CC) \to {}^{\vee}M^{\Gamma}$ be an Arthur parameter. Let $\psi = \epsilon \circ \psi_0$, an Arthur parameter for ${}^{\vee}G^{\Gamma}$. Then 
\begin{itemize}
    \item[(i)] There is an identity in $K\Pi^z(G^{\Gamma})$
$$\eta^{mic}_{\psi}(1) = \epsilon_*\eta^{mic}_{\psi_0}(1);$$
    \item[(ii)] There is an identity in $K\Pi^z(G^{\Gamma})$
$$\eta^{mic}_{\psi}(1) = R^{\fg}_{\fq}\eta^{mic}_{\psi_0}(1).$$
\end{itemize}
\end{prop}

\begin{proof}
Choose $s$ as in Lemma \ref{lem:s}. Then $({}^{\vee}M^{\Gamma},\mathcal{S}_0,s)$ is an endoscopic datum for $({}^{\vee}G^{\Gamma},\mathcal{S})$ in the sense of \cite[Definition 26.15]{AdamsBarbaschVogan}. Now (i) follows from  \cite[Theorem 26.25]{AdamsBarbaschVogan}. (ii) follows from (i) together with Proposition \ref{prop:estarrealinduction}(iii).
 \end{proof}

Under an additional condition on $\mathcal{I}$, we can prove a slightly stronger result. Recall the notation of Section \ref{subsec:inductiongeneral}, e.g., ${}^{\vee}R$, ${}^{\vee}R_0$, ${}^{\vee}K$, ${}^{\vee}K_0$, $X({}^{\vee}M^{\Gamma},Y_0,\OO_0)$, $X({}^{\vee}G^{\Gamma},Y,\OO)$, and so on.

\begin{lemma}\label{lem:RM1equalsR}
Let $I$ be a weak induction datum for the pair $(G^{\Gamma},{}^{\vee}G^{\Gamma})$ such that $Q \subset G$ is real, and let $\psi_0: W_{\RR} \times SL_2(\CC) \to {}^{\vee}M^{\Gamma}$ be an Arthur parameter. Assume that ${}^{\vee}M = Z_{{}^{\vee}G}(\psi_0(\RR_+))$. Then 
${}^{\vee}R_0 = {}^{\vee}R$.
\end{lemma}

\begin{proof}
Write $\psi_1 = \psi_0|_{W_{\RR}}$, a tempered Langlands parameter, and $\psi_2 = \psi_0|_{SL_2(\CC)}$. Choose a maximal torus ${}^{\vee}T \subset {}^{\vee}G$ containing both $\psi_1(\CC^{\times})$ and the image under $\psi_2$ of the diagonal torus in $SL_2(\CC)$. The restriction of $\psi_1$ to $\CC^{\times}$ is a continuous group homomorphism $\CC^{\times} \to {}^{\vee}T$, which is necessarily of the form
$$\psi_1(z) = z^{\lambda_1}\bar{z}^{\mu_1}, \qquad z \in \CC^{\times},$$
for $\lambda_1,\mu_1 \in {}^{\vee}\ft$ satisfying $\lambda_1 - \mu_1 \in X^*(T)$ (since the homomorphism $\CC^{\times} \to {}^{\vee}T$ is restricted from a Langlands parameter $\psi_1$, we in fact have that $\mu_1=\Ad(\psi_1(j))\lambda_1$, but we will not use this fact). If $z \in \RR_+$, we have
$$\psi_1(z) = z^{\lambda_1}\bar{z}^{\mu_1} = z^{\lambda_1 + \mu_1}.$$
Since $\psi_1$ is tempered, $\psi_1(\RR_+)$ is bounded. Hence, $\lambda_1+\mu_1 \in i\RR X_*({}^{\vee}T)$. 

Next, let $h$ denote the image of $\mathrm{diag}(1,-1)$ under $d\psi_2$. Then
$$\lambda = \lambda(\varphi_{\psi}) = \lambda_1 + \frac{1}{2}h \in {}^{\vee}\ft.$$
By definition ${}^{\vee}R$ is the connected subgroup of ${}^{\vee}G$ with Lie algebra
$${}^{\vee}\mathfrak{r} = \bigoplus_{n \in \ZZ} {}^{\vee}\fg_n, \qquad {}^{\vee}\fg_n = \{X \in {}^{\vee}\fg \mid [\lambda,X] = nX\}.$$
On the other hand, ${}^{\vee}R_0$ is the connected subgroup of ${}^{\vee}G$ with Lie algebra
$${}^{\vee}\mathfrak{r}_0 = Z_{\psi_1(\RR_+)}({}^{\vee}\mathfrak{r}).$$
Thus, it suffices to show that 
$$\alpha \in \Delta({}^{\vee}\fg,{}^{\vee}T), \ \alpha(\lambda) \subset \ZZ \implies \alpha(\psi_1(\RR_+)) = \{1\}.$$
Write
$$\lambda = \lambda_{1,re} + \lambda_{1,im} + \frac{1}{2}h, \qquad \lambda_{1,re} = \frac{1}{2}(\lambda_1 - \mu_1), \quad \lambda_{1,im} = \frac{1}{2}(\lambda_1+\mu_1).$$
Since $\lambda_1 - \mu_1 \in X_*({}^{\vee}T)$, we have
$$\alpha(\lambda_{1,re}) = \alpha(\frac{1}{2}(\lambda_1-\mu_1)) \in \frac{1}{2}\ZZ, \qquad \forall \alpha \in \Delta({}^{\vee}\fg,{}^{\vee}\ft).$$
Since $\lambda_1 + \mu_1 \in i\RR X_*({}^{\vee}T)$, we have
$$\alpha(\lambda_{1,im}) = \alpha(\frac{1}{2}(\lambda_1+\mu_1)) \in i\RR, \qquad \forall \alpha \in \Delta({}^{\vee}\fg,{}^{\vee}\ft).$$
Finally, by the representation theory of $\mathfrak{sl}_2(\CC)$
$$\alpha(\frac{1}{2}h) \in \frac{1}{2}\ZZ, \qquad \forall \alpha \in \Delta({}^{\vee}\fg,{}^{\vee}\ft).$$
So if $\alpha(\lambda) \in \ZZ$, then $\alpha(\lambda_{1,im}) = 0$. Since $\psi_1(\RR_+) = \exp(\RR\lambda_{1,im})$, this implies that $\alpha(\psi_1(\RR_+)) = \{1\}$.
\end{proof}

\begin{lemma}\label{lem:real1}
Let $I$ be a weak induction datum for the pair $(G^{\Gamma},{}^{\vee}G^{\Gamma})$ such that $Q \subset G$ is real, and let $\psi_0: W_{\RR} \times SL_2(\CC) \to {}^{\vee}M^{\Gamma}$ be an Arthur parameter. Assume that ${}^{\vee}M=Z_{{}^{\vee}G}(\psi_0(\RR_+))$. Then $\epsilon$ induces a ${}^{\vee}G$-equivariant isomorphism
$$\epsilon\colon {}^{\vee}G \times_{{}^{\vee}M} X({}^{\vee}M^{\Gamma},Y_0,\OO_0) \xrightarrow{\sim} X({}^{\vee}G^{\Gamma},Y,\OO).$$
\end{lemma}

\begin{proof}
By Lemma \ref{lem:RM1equalsR}, we have that ${}^{\vee}R_0 = {}^{\vee}R$ and hence ${}^{\vee}P_0 = {}^{\vee}P$, ${}^{\vee}K_0 = {}^{\vee}K$. Thus, $\epsilon$ is the map
$$\epsilon\colon {}^{\vee}M \times_{{}^{\vee}K} ({}^{\vee}R/{}^{\vee}P) \to {}^{\vee}G \times_{{}^{\vee}K} ({}^{\vee}R/{}^{\vee}P)$$
induced by the inclusion ${}^{\vee}M \subset {}^{\vee}G$. The claim follows.
\end{proof}

\begin{prop}\label{prop:realpacketsJordan}
Let $\mathcal{I}$ be a real induction datum for the pair $((G^{\Gamma},\mathcal{W}),({}^{\vee}G^{\Gamma},\mathcal{S}))$, and let $\psi_0: W_{\RR} \times SL_2(\CC) \to {}^{\vee}M^{\Gamma}$ be an Arthur parameter. Let $\psi := \epsilon \circ \psi_0$, an Arthur parameter for ${}^{\vee}G^{\Gamma}$. Assume that ${}^{\vee}M = Z_{{}^{\vee}G}(\psi_0(\RR_+))$. Then the following are true:
\begin{itemize}
    \item[(i)] $\epsilon$ induces an isomorphism of Arthur component groups
    $$A_{\epsilon}^{mic}\colon A^{mic}_{\psi_0} \xrightarrow{\sim} A^{mic}_{\psi}$$
    \item[(ii)] For any $s \in A^{mic}_{\psi_0}$, there is an identity in $K\Pi^z(G^{\Gamma})$
    $$\eta^{mic}_{\psi}(A^{mic}_{\epsilon}(s)) = \epsilon_*\eta^{mic}_{\psi_0}(s).$$
    \item[(iii)] For any $s \in A^{mic}_{\psi_0}$, there is an identity in $K\Pi^z(G^{\Gamma})$
    $$\eta^{mic}_{\psi}(A^{mic}_{\epsilon}(s)) = R^{\fg}_{\fq}\eta^{mic}_{\psi_0}(s).$$
\end{itemize}
\end{prop}

\begin{proof}
By Proposition \ref{prop:Arthurcomponent}, $A^{mic}_{\epsilon}$ is the natural map
$$\mathrm{Com}(Z_{{}^{\vee}K_0^{alg}}(e)) \to \mathrm{Com}(Z_{{}^{\vee}K^{alg}}(e)),$$
But by Lemma \ref{lem:RM1equalsR}, ${}^{\vee}K_0^{alg} = {}^{\vee}K^{alg}$, so this map is an isomorphism. We note that (i) is also immediate from Lemma \ref{lem:real1} and the genericity of $e$ (Proposition \ref{prop:Arthuroconormal}).

In view of Lemma \ref{lem:eta}, (ii) is equivalent to the following identity in $K^{{}^{\vee}G^{alg}}X({}^{\vee}G^{\Gamma},Y,\OO)^*$:
\begin{equation}\label{eq:matchingmicrostalksreal}m^{mic}_{\psi_0}(s) \circ \epsilon^* = m^{mic}_{\psi}(s)\end{equation}
By Lemma \ref{lem:real1}, the functor
$$\epsilon^*\colon D^bC^{{}^{\vee}G^{alg}}X({}^{\vee}G^{\Gamma},Y,\OO) \to D^bC^{{}^{\vee}M^{alg}} X({}^{\vee}M^{\Gamma},Y_0,\OO_0)$$
is an equivalence, with inverse the functor
$$p^*\colon D^bC^{{}^{\vee}M^{alg}} X({}^{\vee}M^{\Gamma},Y_0,\OO_0) \to D^bC^{{}^{\vee}G^{alg}}X({}^{\vee}G^{\Gamma},Y,\OO),$$
where $p\colon X({}^{\vee}G^{\Gamma},Y,\OO) \to X({}^{\vee}M^{\Gamma},Y_0,\OO_0)$ is the fiber bundle (with fiber ${}^{\vee}G/{}^{\vee}M$). Thus (\ref{eq:matchingmicrostalksreal}) is equivalent to the identity in $K^{{}^{\vee}M^{alg}}X({}^{\vee}M^{\Gamma},Y_0,\OO_0)^*$
$$m^{mic}_{\psi_0}(s) =m^{mic}_{\psi}(s) \circ p^*.$$
Since $p$ is a smooth morphism, this is immediate from \cite[Proposition 4.3.5]{KashiwaraSchapira}.

(iii) follows from (ii) together with Proposition \ref{prop:estarrealinduction}(iii).
\end{proof}

Finally, we address the question of existence of real induction data.

\begin{lemma}\label{lem:realinductiondataexist}
Let
$$I = (M^{\Gamma},{}^{\vee}M^{\Gamma},Q)$$
be a weak induction datum for the pair $(G^{\Gamma},{}^{\vee}G^{\Gamma})$ such that $Q \subset G$ real (Lemma \ref{lem:realisreal}). Then there is an extended group structure $\mathcal{W}_0$ on $M^{\Gamma}$ and an $E$-group structure $\mathcal{S}_0$ on ${}^{\vee}M^{\Gamma}$ such that 
$$\mathcal{I} = ((M^{\Gamma},\mathcal{W}_0),({}^{\vee}M^{\Gamma},\mathcal{S}_0),Q)$$
is a real induction datum for the pair $((G^{\Gamma},\mathcal{W}),({}^{\vee}G^{\Gamma},\mathcal{S}))$.
\end{lemma}

\begin{proof}
Choose any Cartan subgroup $H^{\Gamma} \subset M^{\Gamma}$. By \cite[Proposition 13.6]{AdamsBarbaschVogan}, there is a based Cartan subgroup $(H^{\Gamma},\mathcal{W}(H^{\Gamma}),\Delta_{i\RR}^+,\Delta_{\RR}^+)$ of $(G^{\Gamma},\mathcal{W})$. Choose any strong real form $\delta \in \mathcal{W}(H^{\Gamma})$. Then by definition, $\Ad(\delta)$ is a (quasi-split) real form of $G$ and there is a Whittaker datum $(N,\chi)$ for $G(\RR,\delta)$ such that  $(\delta,N,\chi) \in \mathcal{W}$. After possibly replacing $(N,\chi)$ with a $G(\RR,\delta)$-conjugate, we may assume that $M$ is standard with respect to $N$. Then, $N_0 := N \cap M$ is a maximal unipotent subgroup of $M$, normalized by $\delta$, and $\chi_0 := \chi|_{N_0(\RR)}$ is a non-degenerate unitary character. Define
$$\mathcal{W}_0 := M \cdot (\delta,N_0,\chi_0).$$
To construct $\mathcal{S}_0$, choose a fundamental Cartan subgroup ${}^{\vee}H^{\Gamma} \subset {}^{\vee}M^{\Gamma}$. Since ${}^{\vee}Q$ is $\theta$-stable, ${}^{\vee}H^{\Gamma}$ is also a fundamental Cartan subgroup of  ${}^{\vee}G^{\Gamma}$. By \cite[Proposition 13.8]{AdamsBarbaschVogan}, there is a based Cartan subgroup $({}^{\vee}H^{\Gamma},\mathcal{S}({}^{\vee}H^{\Gamma}),{}^{\vee}\Delta_{i\RR}^+,{}^{\vee}\Delta_{\RR}^+)$ of $({}^{\vee}G^{\Gamma},\mathcal{S})$. Choose any strong involution ${}^{\vee}\delta' \in \mathcal{S}({}^{\vee}H^{\Gamma})$. By definition, there is a large ${}^{\vee}\delta'$-stable Borel subgroup ${}^{\vee}B' \subset {}^{\vee}G$ such that $({}^{\vee}\delta,{}^{\vee}B) \in \mathcal{S}$. Choose (as is possible) a ${}^{\vee}\delta'$-stable Borel subgroup ${}^{\vee}B \subset {}^{\vee}Q$ containing ${}^{\vee}H$ and a Weyl group element $w \in W(G,H)$ such that $w\Delta({}^{\vee}B') = \Delta({}^{\vee}B)$. Since both ${}^{\vee}B$ and ${}^{\vee}B'$ are ${}^{\vee}\delta'$-stable, $w$ is ${}^{\vee}\delta'$-fixed. Choose any lift $\dot{w} \in N_{{}^{\vee}G}({}^{\vee}H)$ of $w$, and consider the strong involution ${}^{\vee}\delta := \dot{w}{}^{\vee}\delta\dot{w}^{-1}$ of ${}^{\vee}G^{\Gamma}$. Since $w$ is ${}^{\vee}\delta'$-fixed, we have that $\Ad({}^{\vee}\delta)=\Ad({}^{\vee}\delta')$ as involutions of ${}^{\vee}H$. Hence, ${}^{\vee}\delta \in {}^{\vee}\delta'H \subset {}^{\vee}H^{\Gamma} \subset {}^{\vee}M^{\Gamma}$. Moreover, ${}^{\vee}B_0 := {}^{\vee}B \cap {}^{\vee}M$ is a large ${}^{\vee}\delta$-stable Borel subgroup of ${}^{\vee}M$. Define
$$\mathcal{S}_0 := {}^{\vee}M \cdot ({}^{\vee}\delta,{}^{\vee}B_0).$$
By construction, the induction datum $\mathcal{I}:=((G^{\Gamma},\mathcal{W}),(M^{\Gamma},\mathcal{W}_0),({}^{\vee}G^{\Gamma},\mathcal{S}),({}^{\vee}M^{\Gamma},\mathcal{S}_0),Q)$ satisfies properties $(a_{\RR})$ and $(b_{\RR})$ of Definition \ref{def:realinductiondatum}, and is therefore real.
\end{proof}

\subsection{$\theta$-stable induction data}\label{subsec:thetastableinductiondata}

Let $(G^{\Gamma},\mathcal{W})$ be an extended group and let $({}^{\vee}G^{\Gamma},\mathcal{S})$ be a corresponding $E$-group with second invariant $z$. Recall the maps $\mathcal{B}(-)$ and $\mathcal{N}(-)$ of Definition \ref{d:bijection}.

\begin{definition}\label{def:thetastableinductiondatum}
Let $\mathcal{I}$ be an induction datum for the pair $((G^{\Gamma},\mathcal{W}),({}^{\vee}G^{\Gamma},\mathcal{S}))$. We say that $\mathcal{I}$ is `$\theta$-stable' if the following conditions are satisfied
\begin{itemize}
\item[$(a_\theta)$] Given $(\delta,N_0,\chi_0)\in \mathcal{W}_0$ there exists  $(\delta,N,\chi)\in \mathcal{W}$
  such that if $\theta$ is 
a Cartan involution  compatible with $\Ad(\delta)$ and preserving $M$ then  $\mathcal B(N_0,\chi_0,\theta|_M)U\subset \mathcal B(N,\chi,\theta)$.
\item[$(b_\theta)$] Given $(\delta,B_0)\in \mathcal{S}_0$ there exists $(\delta,B)\in \mathcal{S}$ such that if $\sigma$ is a real form of $G$, compatible with $\Ad(\delta)$ and preserving $M$, then $\mathcal N(B_0,\sigma|_{M})\subset \mathcal N(B,\sigma)$ (notation as in Definition \ref{def:realinductiondatum}).
\end{itemize}
\end{definition}

We note that condition $(a_{\theta})$ of Definition \ref{def:thetastableinductiondatum} implies that $Q$ is $\theta$-stable (in the sense of Lemma \ref{lem:realisreal}) and condition $(b_{\theta})$ implies that $({}^{\vee}M^{\Gamma},\mathcal{S}_0)$ is of type $zz(\rho_{\fu})$. Thus, we obtain two homomorphisms with the same source and target
$$\epsilon_*, R^{\fg}_{\fq} \colon K\Pi^{zz(\rho_{\fu})}(M^{\Gamma}) \to K\Pi^z(G^{\Gamma}).$$
The homomorphism $\epsilon_*$ is Langlands functoriality (Section \ref{subsec:ABVfunctoriality}). The homomorphism $R^{\fg}_{\fq}$ is (normalized) cohomological induction (Section \ref{sec:parabolicinductionextended}). Our first task is to compare these homomorphisms (in contrast to the case of real induction data, these homomorphisms do not coincide in general).

For the following lemma, we work in the setting of \cite[Chapter 13]{AdamsBarbaschVogan}.

\begin{lemma}\label{lem:basedCartansthetastable}
Let $\mathcal{I}$ be a $\theta$-stable induction datum for the pair $((G^{\Gamma},\mathcal{W}),({}^{\vee}G^{\Gamma},\mathcal{S}))$.
\begin{itemize}
    \item[(i)] Let $(H^{\Gamma},\mathcal{W}_0(H^{\Gamma}),\Delta^+_{\RR},\Delta^+_{i\RR})$ be a based Cartan subgroup of $(M^{\Gamma},\mathcal{W}_0)$. Then $(H^{\Gamma},\mathcal{W}_0(H^{\Gamma}),\Delta^+_{\RR},\Delta^+_{i\RR} \sqcup \Delta_{i\RR}(\fu))$ is a based Cartan subgroup of $(G^{\Gamma},\mathcal{W})$.
    \item[(ii)] Let $({}^{\vee}H^{\Gamma},\mathcal{S}_0({}^{\vee}H^{\Gamma}),{}^{\vee}\Delta^+_{\RR},{}^{\vee}\Delta^+_{i\RR})$ be a based Cartan subgroup of $({}^{\vee}M^{\Gamma},\mathcal{S}_0)$. Then $({}^{\vee}H^{\Gamma},\mathcal{S}_0({}^{\vee}H^{\Gamma}),{}^{\vee}\Delta^+_{\RR} \sqcup \Delta_{\RR}({}^{\vee}\fu),{}^{\vee}\Delta^+_{i\RR})$ is a based Cartan subgroup of $({}^{\vee}G^{\Gamma},\mathcal{S})$.
\end{itemize}
\end{lemma}

\begin{proof}
For (i), choose $\delta \in \mathcal{W}_0(H^{\Gamma})$. By the definition of a based Cartan subgroup, there is some $(\delta,N_0,\chi_0) \in \mathcal{W}_0$ and hence, by $(a_{\theta})$, a $(\delta,N,\chi) \in \mathcal{W}$, involution $\theta$ of $G$, preserving $M$ and compatible with $\Ad(\delta)$, and large $\theta$-stable Borels $B_0 \subset M$ and $B \subset G$, corresponding $(N_0,\chi_0)$ and $(N,\chi)$, such that $B=B_0U$. Choose a final limit parameter $(\delta,\Lambda) = (\delta, (H(\RR),\Lambda^{can},\Delta_{\RR}^+,\Delta_{i\RR}^+ \sqcup \Delta_{i\RR}(\fu))) \in L^z(G^{\Gamma})$. Then $(\delta,(H(\RR),\Lambda^{can},\Delta_{\RR}^+,\Delta_{i\RR}^+)) \in L^{zz(\rho_{\fu})}(M^{\Gamma})$. Write $(\delta,X) \in K\Pi^z(G^{\Gamma})$ and $(\delta,X_0) \in K\Pi^{zz(\rho_{\fu})}(M^{\Gamma})$ for the corresponding standard modules. By the transitivity of induction, we have $X = L_{\dim(\fu^{\theta})}\mathcal{R}^{\fg}_{\fq}X_0$ (cohomological induction in the good range). Since $(H^{\Gamma},\mathcal{W}_0(H^{\Gamma}),\Delta^+_{\RR},\Delta^+_{i\RR})$ is a based Cartan subgroup of $(M^{\Gamma},\mathcal{W}_0)$, we know that $X_0$ has a Whittaker model of type $(N_0,\chi_0)$. Thus by Proposition \ref{prop:Whittakerinduction}(ii), $X$ has a Whittaker model of type $(N,\chi)$. This proves that $(H^{\Gamma},\mathcal{W}_0(H^{\Gamma}),\Delta^+_{\RR},\Delta^+_{i\RR}\sqcup \Delta_{i\RR}(\fu))$ is a based Cartan subgroup of $(G^{\Gamma},\mathcal{W})$.

For (ii), choose ${}^{\vee}\delta \in \mathcal{S}_0({}^{\vee}H^{\Gamma})$. By the definition of a based Cartan subgroup, there is some $({}^{\vee}\delta,{}^{\vee}B_0) \in \mathcal{S}_0$, and hence, by $(b_{\theta})$, a $({}^{\vee}\delta,{}^{\vee}B)\in\mathcal{S}$, a real form $\sigma$ of ${}^{\vee}G$ preserving ${}^{\vee}M$ and compatible with $\Ad({}^{\vee}\delta)$, and Whittaker data $({}^{\vee}N_0,\chi_0)$ and $({}^{\vee}N,\chi)$ corresponding to ${}^{\vee}B_0$ and ${}^{\vee}B$, such that ${}^{\vee}N_0 \subset {}^{\vee}N$ and $\chi_0 = \chi|_{{}^{\vee}N_0(\RR)}$. Since $({}^{\vee}H^{\Gamma},\mathcal{S}_0({}^{\vee}H^{\Gamma}),{}^{\vee}\Delta_{\RR}^+,{}^{\vee}\Delta_{i\RR}^+)$ is a based Cartan subgroup of $({}^{\vee}M^{\Gamma},\mathcal{S}_0)$, we know that $\Delta_{i\RR}^+$ is special with respect to ${}^{\vee}B_0$. Thus by Proposition \ref{prop:specialinduction}(ii), $\Delta_{i\RR}^+$ is special with respect to ${}^{\vee}B$. This proves that $({}^{\vee}H^{\Gamma},\mathcal{S}_0({}^{\vee}H^{\Gamma}),{}^{\vee}\Delta^+_{\RR} \sqcup \Delta_{\RR}({}^{\vee}\fu),{}^{\vee}\Delta^+_{i\RR})$ is a based Cartan subgroup of $({}^{\vee}G^{\Gamma},\mathcal{S})$. 
\end{proof}

Recall the operation $\Lambda \mapsto \Lambda^{\#}$ on limit characters (Definition \ref{def:hashparameter}).

\begin{prop}\label{prop:estarthetastableinduction}
Let $\mathcal{I}$ be a $\theta$-stable induction datum for the pair $((G^{\Gamma},\mathcal{W}),({}^{\vee}G^{\Gamma},\mathcal{S}))$, and let ${}^{\vee}\OO_0 \subset {}^{\vee}\fm$ be a semisimple ${}^{\vee}M$-orbit in the good range for $\fq$ (Definition \ref{def:weaklyfair}).
\begin{itemize}
    \item[(i)] For each final limit character $(\delta,\Lambda) \in L^{zz(\rho_{\fu})}_{{}^{\vee}\OO_0}(M^{\Gamma})$, we have
$$\epsilon_*M(\delta,\Lambda) = \begin{cases*}
                    e(M(\RR,\delta))e(G(\RR,\delta))M(\delta,\Lambda^{\#})
                     & if  $\delta^2 \in Z(G)$  \\
                     0 & else
                 \end{cases*}$$
In this formula, $(\delta,\Lambda^{\#})$ is a final non-zero standard type-$z$ limit character for $G^{\Gamma}$.
\item[(ii)]  Cohomological induction induction $R^{\fg}_{\fq}$ satisfies a similar formula, but with a different sign:
$$R^{\fg}_{\fq}M(\delta,\Lambda) = \begin{cases*}
                    (-1)^{\dim(\fu^{\theta})}M(\delta,\Lambda^{\#})
                     & if  $\delta^2 \in Z(G)$  \\
                     0 & else
                 \end{cases*}$$
\item[(iii)] Consequently
$$\epsilon_* \equiv \kappa R^{\fg}_{\fq}\colon K\Pi^{zz(\rho_{\fu})}_{{}^{\vee}\OO_0}(M^{\Gamma}) \to K\Pi^z(G^{\Gamma}),$$
where $\kappa$ is a sign, depending only on $\mathcal{I}$, given by the formula
$$\kappa = (-1)^{\dim(\fu) + q(G^*) - q(M^*)}$$
where $G^*$ (resp. $M^*$) is the quasisplit real form in the inner class defined by $G^{\Gamma}$ (resp. $M^{\Gamma}$) and $q(G^*)$ is half the dimension of the $-1$-eigenspace of a Cartan involution for $G^*$.
\end{itemize}
\end{prop}

\begin{proof}
(ii) is immediate from the definition of $R^{\fg}_{\fq}$ (Equation (\ref{eq:indextendedgroups}), the construction of standard modules from standard limit characters, and the transitivity of induction. (iii) follows from (i), (ii), and the following calculation with Kottwitz signs: recall that by definition $e(G(\RR,\delta)) = (-1)^{q(G(\RR,\delta)) - q(G^*)}$, where $q(G(\RR)) = \frac{1}{2}\dim(\fg^{-\theta})$. Note that $q(G(\RR,\delta))-q(M(\RR,\delta)) = \dim(\fu^{-\theta})$. So
$$e(G(\RR,\delta))e(M(\RR,\delta))(-1)^{\dim(\fu^{\theta})} = (-1)^{q(G^*)-q(M^*) + \dim(\fu^{-\theta}) + \dim(\fu^{\theta})} = (-1)^{\dim(\fu) + q(G^*)-q(M^*)}$$
The proof of (i) is similar to the proof of Proposition \ref{prop:estarrealinduction}(i), except that we use Lemma \ref{lem:basedCartansthetastable} instead of Lemma \ref{lem:basedCartansreal} and the limit character $(\delta,\Lambda^{\#})$ is final (so no finalization is required).
\end{proof}

Our next task is to analyze the behavior of $\epsilon_*$ with respect to Arthur packets. Recall the notation of Section \ref{subsec:inductiongeneral}, e.g., ${}^{\vee}R$, ${}^{\vee}R_0$, ${}^{\vee}K$, ${}^{\vee}K_0$, $X({}^{\vee}M^{\Gamma},Y_0,\OO_0)$, $X({}^{\vee}G^{\Gamma},Y,\OO)$, and so on.

\begin{lemma}\label{lem:parabolics}
Let $I$ be a weak induction datum for the pair $(G^{\Gamma},{}^{\vee}G^{\Gamma})$ such that $Q \subset G$ is $\theta$-stable, and let $\psi_0: W_{\RR} \times SL_2(\CC) \to {}^{\vee}M^{\Gamma}$ be an Arthur parameter. Then
\begin{itemize}
    \item[(i)] ${}^{\vee}Q \cap {}^{\vee}R$ is a parabolic subgroup of ${}^{\vee}R$ with Levi factor ${}^{\vee}R_0$. 
    \item[(ii)] Suppose that $\lambda(\varphi_{\psi_0})$ is in the good range for $\fq$. Then ${}^{\vee}P \subset {}^{\vee}Q$. More precisely, ${}^{\vee}Q \cap {}^{\vee}R={}^{\vee}R_0 {}^{\vee}P$.
\end{itemize}
\end{lemma}

\begin{proof}
Since $S^1$ commutes with $\CC^{\times} \times SL_2(\CC)$ in $W_{\RR} \times SL_2(\CC)$, $\psi(S^1)$ commutes with $\lambda = \lambda(\varphi_{\psi_0})$ in ${}^{\vee}G$. Since ${}^{\vee}R = Z_{{}^{\vee}G}(e(\lambda))^0$, this implies that $\psi_0(S^1) \subset {}^{\vee}R$. Thus, ${}^{\vee}Q \cap {}^{\vee}R$ is the connected subgroup of ${}^{\vee}R$ defined by the co-character $\psi_0|_{S^1}: S^1 \to {}^{\vee}R$. This is a parabolic subgroup of ${}^{\vee}R$ with Levi factor 
$$Z_{\psi_0(S^1)}({}^{\vee}R) = Z_{\psi_0(S^1)}(Z_{{}^{\vee}G}(e(\lambda))^0) = Z_{Z_{\psi_0(S^1)}({}^{\vee}G)}(e(\lambda))^0 = Z_{{}^{\vee}M}(e(\lambda))^0 = {}^{\vee}R_0.$$
This proves (i)

For (ii), choose any maximal torus ${}^{\vee}H$ containing $\psi_0(\CC^{\times})$ so that ${}^{\vee}H \subset {}^{\vee}P$ and ${}^{\vee}H \subset {}^{\vee}Q$. We have
   $$\Delta(\mathrm{nil}({}^{\vee}\fp),{}^{\vee}\fh) = \{\alpha^{\vee} \in \Delta \mid \langle  \psi_0|_{\RR^+}, \alpha^{\vee}\rangle =0 \text{ and } \langle \lambda, \alpha^{\vee}\rangle \in \ZZ_{> 0}\},$$
   and
   $$\Delta(\mathrm{nil}({}^{\vee}\fq \cap {}^{\vee}\mathfrak{r}),{}^{\vee}\fh) = \{\alpha^{\vee} \in \Delta \mid \langle \psi_0|_{\RR^+},  \alpha^{\vee}\rangle =0, \ \langle \lambda, \alpha^{\vee}\rangle \in \ZZ,  \text{ and } \langle  \psi_0|_{S^1}, \alpha^{\vee}\rangle \in \ZZ_{>0}\}$$
So, assuming $\psi_0$ is in the good range, we have that $\Delta(\mathrm{nil}({}^{\vee}\fq \cap {}^{\vee}\mathfrak{r}),{}^{\vee}\fh) \subset \Delta(\mathrm{nil}({}^{\vee}\fp),{}^{\vee}\fh)$, and hence $\mathrm{nil}({}^{\vee}\fq \cap {}^{\vee}\mathfrak{r}) \subset \mathrm{nil}({}^{\vee}\fp)$. For parabolic subalgebras $\fa$ and $\fb$ containing a common Cartan subalgebra, $\mathrm{nil}(\fa) \subset \mathrm{nil}(\fb)$ implies $\fb \subset \fa$. So we have ${}^{\vee}\fp \subset {}^{\vee}\fq \cap {}^{\vee}\mathfrak{r}$ and therefore ${}^{\vee}P \subset {}^{\vee}Q \cap {}^{\vee}R$. The claim that ${}^{\vee}Q \cap {}^{\vee}R = {}^{\vee}R_0 {}^{\vee}P$ follows immediately.
\end{proof}

Now consider the ${}^{\vee}M$-invariant subset 
$$\tilde{U} := {}^{\vee}M \cdot \epsilon (X({}^{\vee}M^{\Gamma}, Y_0, \OO_0)) \subset X({}^{\vee}G^{\Gamma}, Y, \OO).$$

\begin{lemma}\label{lem:thetastable2}
Let $I$ be a weak induction datum for the pair $(G^{\Gamma},{}^{\vee}G^{\Gamma})$ such that $Q \subset G$ is $\theta$-stable, and let $\psi_0: W_{\RR} \times SL_2(\CC) \to {}^{\vee}M^{\Gamma}$ be an Arthur parameter. Assume that $\lambda(\varphi_{\psi_0})$ is in the good range for $\fq$. Then the following are true:
\begin{itemize}
    \item[(i)] $\epsilon$ induces a ${}^{\vee}G$-equivariant isomorphism
    $${}^{\vee}G \times_{{}^{\vee}M} X({}^{\vee}M^{\Gamma},Y_0,\OO_0) \xrightarrow{\sim} \widetilde{U}.$$
    \item[(ii)] $\tilde{U}$ is an open subset of $X({}^{\vee}G^{\Gamma}, Y, \OO)$.
\end{itemize}
\end{lemma}

\begin{proof}
For (i), it is equivalent to show that the natural map
\begin{equation}\label{eq:embedding1}
{}^{\vee}K \times_{{}^{\vee}K_0} {}^{\vee}R_0/{}^{\vee}P_0 \to {}^{\vee}R/{}^{\vee}P
\end{equation}
is injective, i.e. that the stabilizer in ${}^{\vee}K$ of each point in ${}^{\vee}R_0{}^{\vee}P/{}^{\vee}P \subset {}^{\vee}R/{}^{\vee}P$ belongs to ${}^{\vee}K_0$. But every such stabilizer is contained in
$${}^{\vee}K \cap {}^{\vee}R_0{}^{\vee}P,$$
and by Lemma \ref{lem:parabolics}, we have that ${}^{\vee}R_0{}^{\vee}P = {}^{\vee}Q \cap {}^{\vee}R$. So 
$${}^{\vee}K \cap {}^{\vee}R_0{}^{\vee}P = {}^{\vee}K \cap {}^{\vee}Q \cap {}^{\vee}R = {}^{\vee}K \cap {}^{\vee}R_0 = {}^{\vee}K_0.$$
This proves (i).

For (ii), we must show that the image of (\ref{eq:embedding1}), namely ${}^{\vee}K_{}^{\vee}R_0{}^{\vee}P/{}^{\vee}P$ is an open subset of ${}^{\vee}R/{}^{\vee}P$. By part (i), this subset is the image of an injective map from an irreducible variety, and hence locally closed and irreducible. So it suffices to show that 
$$\dim({}^{\vee}K_0{}^{\vee}R_0{}^{\vee}P/{}^{\vee}P) = \dim({}^{\vee}R/{}^{\vee}P).$$
This follows by comparing tangent spaces at the point $^\vee P/^\vee P$. The tangent space in $^\vee R/^\vee P$ is $^\vee\mathfrak{r}/^\vee\mathfrak{p}$, whereas the tangent space in ${}^{\vee}K{}^{\vee}R_0{}^{\vee}P/{}^{\vee}P$ is 
$$\left(^\vee\mathfrak{k}+^\vee\mathfrak{r}_0+^\vee\mathfrak{p}\right)/^\vee\mathfrak{p}=\left(^\vee\mathfrak{k}+(^\vee\mathfrak{q}\cap ^\vee\mathfrak{r})\right)/^\vee\mathfrak{p}.$$
Thus, it suffices to check
$$^\vee\mathfrak{k}+(^\vee\mathfrak{q}\cap ^\vee\mathfrak{r})=^\vee\mathfrak{r}.$$
This follows from the fact that the parabolic  $^\vee\mathfrak{q}\cap ^\vee\mathfrak{r} \subset {}^{\vee}\mathfrak{r}$ is real, i.e. taken to its opposite by the Cartan involution defining ${}^{\vee}\fk$.
\end{proof}

\begin{prop}\label{prop:thetastablepacketsgood}
Let $\mathcal{I}$ be a $\theta$-stable induction datum for the pair $((G^{\Gamma},\mathcal{W}),({}^{\vee}G^{\Gamma},\mathcal{S}))$, and let $\psi_0: W_{\RR} \times SL_2(\CC) \to {}^{\vee}M^{\Gamma}$ be an Arthur parameter. Assume that $\lambda(\varphi_{\psi_0})$ is in the good range for $\fq$. Let $\psi := \epsilon \circ \psi_0$, an Arthur parameter for ${}^{\vee}G^{\Gamma}$. Then the following are true:
\begin{itemize}
    \item[(i)] $\epsilon$ induces an isomorphism of Arthur component groups
    $$A_{\epsilon}^{mic}\colon A_{\psi_0}^{mic} \xrightarrow{\sim} A_{\psi}^{mic}$$
    \item[(ii)] For any $s \in A_{\psi_0}^{mic}$, there is an identity in $K\Pi^z(G^{\Gamma})$
    $$\eta^{mic}_{\psi}(A^{mic}_{\epsilon}(s)) = \epsilon_*\eta^{mic}_{\psi_0}(s).$$
    \item[(iii)] For any $s \in A_{\psi_0}^{mic}$, there is an identity in $K\Pi^z(G^{\Gamma})$
    $$\eta^{mic}_{\psi}(A^{mic}_{\epsilon}(s)) = \kappa R^{\fg}_{\fq}\eta^{mic}_{\psi_0}(s),$$
    where $\kappa$ is the sign of Proposition \ref{prop:estarthetastableinduction}(iii).
\end{itemize}
\end{prop}

\begin{proof}
In view of Lemmas \ref{prop:estarthetastableinduction}(iii) and \ref{lem:thetastable2}, the proof is basically identical to that of Proposition \ref{prop:realpacketsJordan}.
\end{proof}

We now consider the case of weakly fair range Arthur parameters. 

\begin{prop}\label{prop:thetastablepackets}
Let $\mathcal{I}$ be a $\theta$-stable induction datum for the pair $((G^{\Gamma},\mathcal{W}),({}^{\vee}G^{\Gamma},\mathcal{S}))$, and let $\psi_0: W_{\RR} \times SL_2(\CC) \to {}^{\vee}M^{\Gamma}$ be a unipotent Arthur parameter. Assume that $\lambda(\varphi_{\psi_0})$ is in the weakly fair range for $\fq$. Let $\psi := \epsilon \circ \psi_0$, an Arthur parameter for ${}^{\vee}G^{\Gamma}$.  Then the following are true:
\begin{itemize}
    \item[(i)] $\epsilon$ induces an isomorphism of Arthur component groups
    $$A_{\epsilon}^{mic}\colon A_{\psi_0}^{mic} \xrightarrow{\sim} A_{\psi}^{mic}$$
    \item[(ii)] For any $s \in A_{\psi_0}^{mic}$, there is an identity in $K\Pi^z(G^{\Gamma})$
    $$\eta^{mic}_{\psi}(A^{mic}_{\epsilon}(s)) = \kappa R^{\fg}_{\fq} \eta^{mic}_{\psi_0}(s).$$
    Here, $R^{\fg}_{\fq}$ is cohomological induction in the weakly fair range and $\kappa$ is the sign defined in Proposition \ref{prop:estarthetastableinduction}(iii).
\end{itemize}
\end{prop}

\begin{proof}
Let $2\rho_{\fu}$ denote the algebraic co-character $\CC^{\times} \to Z({}^{\vee}M)$ or algebraic character $M \to \CC^{\times}$ corresponding to the sum of the roots in $\fu$. For any integer $N$, we obtain an Arthur parameter $\psi_0^N\colon W_{\RR} \times SL_2(\CC) \to {}^{\vee}M^{\Gamma}$ according to the formula
$$\psi_0^N(j) = \psi_0(j), \quad \psi_0^N|_{SL_2(\CC)} = \psi_0|_{SL_2(\CC)}, \quad \psi_0^N(z) = \psi_0(z)(2\rho_{\fu})(z/\overline{z})^{2N} \quad z \in \CC^{\times}$$
The Arthur packets and $s$-stable sums for $\psi_0^N$ and $\psi_0$ are related by the formulas
\begin{equation}\label{eq:charactertwist}
\Pi^{zz(\rho_{\fu})}(M^{\Gamma})^{mic}_{\psi_0^N} = \Pi^{zz(\rho_{\fu})}(M^{\Gamma})^{mic}_{\psi_0} \otimes (2\rho_{\fu})^{2N}, \qquad \eta^{mic}_{\psi_0^N}(s) = \eta^{mic}_{\psi_0}(s) \otimes (2\rho_{\fu})^{2N}.
\end{equation}
Define an Arthur parameter $\psi^N\colon W_{\RR} \times SL_2(\CC) \to {}^{\vee}G^{\Gamma}$ for ${}^{\vee}G^{\Gamma}$ by $\psi^N := \epsilon \circ \psi^N_0$. 

Let $z_f \in Z({}^{\vee}\fm)$ denote the holomorphic part of the differential of $\psi_0|_{\CC^{\times}}$ and let 
$h \in {}^{\vee}\fm$ denote the image of $\mathrm{diag}(1,-1)$ under $d\psi|_{\mathfrak{sl}_2(\CC)}$. Let $z_g = z_f + 4N\rho_{\fu}$, the holomorphic part of the differential of $\psi^N_0|_{\CC^{\times}}$. Then
$$\lambda_f :=  \lambda(\varphi_{\psi}) = \frac{1}{2}h + z_f, \qquad \lambda_g := \lambda(\varphi_{\psi^N}) = \frac{1}{2}h + z_g.$$
(notation as in Proposition \ref{prop:classicaltogeometric}). Now choose $N>>0$ so that $\lambda_g$ is strictly dominant for $\fu$ (this is possible by Lemma \ref{lem:rhotwistgoodrange}). This means that $\psi^N$ is in the good range for $\mathcal{I}$. So by Proposition \ref{prop:thetastablepacketsgood}(i), we have $A^{mic}_{\psi^N_0} \xrightarrow{\sim} A^{mic}_{\psi^N}$. On the other hand, by Proposition \ref{prop:Arthurcomponent}, we have natural identifications
$$A^{mic}_{\psi_0} \simeq \mathrm{Com}(Z_{{}^{\vee}K^{alg}_{M}}(e)) \simeq A^{mic}_{\psi_0^N}, \qquad A^{mic}_{\psi} \simeq \mathrm{Com}(Z_{{}^{\vee}K^{alg}}(e)) \simeq A^{mic}_{\psi^N}.$$
Thus, all four component groups $A^{mic}_{\psi_0}$, $A^{mic}_{\psi_0^N}$, $A^{mic}_{\psi^N}$, and $A^{mic}_{\psi}$ are canonically identified (with each other, and with $\mathrm{Com}(Z_{{}^{\vee}K^{alg}}(e))$). This proves (i).

We proceed to proving (ii). Since $\psi^N$ is in the good range for $\mathcal{I}$, Proposition \ref{prop:thetastablepacketsgood}(iii) implies that for 
each $s \in A_{\psi_0}^{mic}$, there is an identity in $K\Pi^z(G^{\Gamma})$
\begin{equation}\label{eq:goodrangeidentity2}\eta^{mic}_{\psi^N}(A^{mic}_{\epsilon}(s)) = \kappa R^{\fg}_{\fq} \eta^{mic}_{\psi_0^N}(s) \end{equation}
Since $\lambda_g$ is good and $\lambda_f-\lambda_g \in X^*(M)$, there is a well-defined $M$-coherent continuation homomorphism (Definition \ref{def:Lcoherentcontinuation})
$$T_{M}(\lambda_g \to \lambda_f)\colon K\Pi^z_{\lambda_g}(G^{\Gamma}) \to K\Pi^z_{\lambda_f}(G^{\Gamma})$$
To prove the identity in (ii), we will apply this homomorphism to both sides of (\ref{eq:goodrangeidentity2}). By Proposition \ref{prop:cohindLcoherent} and (\ref{eq:charactertwist}), we have
\begin{equation}\label{eq:coherentfrominduced}T_M(\lambda_g \to \lambda_f)\kappa R^{\fg}_{\fq} \eta^{mic}_{\psi_0^N}(s) = \kappa R^{\fg}_{\fq}\eta^{mic}_{\psi_0}(s).\end{equation}
Thus, it suffices to show that
$$T_M(\lambda_g \to \lambda_f)\eta^{mic}_{\psi^N}(A^{mic}_{\epsilon}(s)) = \eta^{mic}_{\psi}(A^{mic}_{\epsilon}(s)).$$
In view of Theorem \ref{thm:convolutionisdualtocoherentcontinuation} and Lemma \ref{lem:eta}, this is equivalent to the following identity in $K^{{}^{\vee}G^{alg}}X({}^{\vee}G^{\Gamma},Y,\OO_f)^*$
$$m^{mic}_{((y,\Lambda_g),e)}(A^{mic}_{\epsilon}(s)) \circ I(\Lambda_f \to \Lambda_g) = m^{mic}_{((y,\Lambda_f),e)}(A^{mic}_{\epsilon}(s)),$$
where $\Lambda_f = \mathcal{F}(\lambda_f)$ and $\Lambda_g=\mathcal{F}(\lambda_g)$. This is immediate from the following lemma.

\begin{lemma}\label{lem:intandmic}
Let $\mathcal{G} \in P^{{}^{\vee}G^{alg}}X({}^{\vee}G^{\Gamma},Y,\OO_f)$. Then there is an isomorphism of $A^{mic}_{\psi}$-representations
$$\chi^{mic}_{((y,\Lambda_g),e)}  I(\Lambda_f \to \Lambda_g)(\mathcal{G}) \simeq \chi^{mic}_{((y,\Lambda_f),e)}(\mathcal{G}).$$
\end{lemma}

\begin{proof}
Let ${}^{\vee}R = {}^{\vee}R(y)$, ${}^{\vee}P_g = {}^{\vee}P(\Lambda_g)$, ${}^{\vee}P_f = {}^{\vee}P(\Lambda_f)$, and consider the correspondence
\begin{center}
    \begin{tikzcd}
       & {}^{\vee}R/({}^{\vee}P_g \cap {}^{\vee}P_f) \ar[dl,"q_g"] \ar[dr,"q_f"] & \\
       {}^{\vee}R/{}^{\vee}P_g & & {}^{\vee}R/{}^{\vee}P_f
    \end{tikzcd}
\end{center}
Recall that after identifying $P^{{}^{\vee}G^{alg}}X({}^{\vee}G^{\Gamma},Y,\OO_g) \simeq P^{{}^{\vee}K^{alg}}({}^{\vee}R/{}^{\vee}P_g)$ and $P^{{}^{\vee}G^{alg}}X({}^{\vee}G^{\Gamma},Y,\OO_f) \simeq P^{{}^{\vee}K^{alg}}({}^{\vee}R/{}^{\vee}P_f)$, the convolution functor $I(\Lambda_f \to \Lambda_g)$ is given by $(q_g)_! \circ q_f^*$. 

Factor $q_g$ as $q_g = p_g \circ \iota$, where 
$$\iota \colon {}^\vee R/({}^\vee P_g\cap {}^\vee P_f)\to {}^\vee R/{}^\vee P_g \times {}^\vee R/^\vee P_f$$
is the diagonal embedding, and 
$$p_g \colon {}^\vee R/{}^\vee P_g \times {}^\vee R/^\vee P_f \to {}^\vee R/^\vee P_g$$
is the first projection. Regard $\CG$ as an object in $P^{{}^{\vee}K^{alg}}({}^{\vee}R/{}^{\vee}P_f)$ and let $\CF=q_f^*\CG$. We wish to show that there is an isomorphism
\begin{equation}\label{eq:iso}\chi^{mic}_{(1{}^{\vee}P_g,e)}(p_g)_!\iota_!\mathcal{F} \simeq \chi^{mic}_{(1{}^{\vee}P_f,e)}\mathcal{G}\end{equation}
of representations of $A^{mic}_{\psi}$. By the general properties of microlocalization functors under pullbacks and pushforwards (\cite[Proposition 4.3.4, Proposition 4.3.5]{KashiwaraSchapira}) there is a natural morphism of $A^{mic}_{\psi}$-representations 
$$\chi^{mic}_{(1{}^{\vee}P_g,e)}(p_g)_!\iota_!\mathcal{F} \to \chi^{mic}_{(1{}^{\vee}P_f,e)}\mathcal{G}.$$
It remains to show that this morphism is an isomorphism of vector spaces.

For a $^\vee K$-orbit $S$ on $^\vee R/^\vee P_f$ and a $^\vee R$-orbit $T$ on $^\vee R/^\vee P_g\times ^\vee R/^\vee P_f$, let 
$$\CS_{S,T}=( ^\vee R/^\vee P_g \times S)\cap T.$$
Then $\{\CS_{S,T}\}$ is a stratification of $^\vee R/^\vee P_g\times ^\vee R/^\vee P_f$ by smooth locally closed subsets. Note that $\iota$ is simply an inclusion of a union of strata and $\CF$ is constructible with respect to this stratification. Thus, so is $\iota_!\CF$. In particular, 
\begin{equation}\label{eq:SS}SS(\iota_!\CF)\subset \bigcup_{S,\ T\subset \overline{^\vee R\cdot(1^\vee P_g,1^\vee P_f)}}T^*_{\CS_{S,T}}\left({}^\vee R/^\vee P_g\times ^\vee R/^\vee P_f\right),\end{equation}
where the union runs over all $^\vee R$-orbits $T$ in the closure of $^\vee R \cdot (1^\vee P_g,1^\vee P_f)$ and all ${}^{\vee}K$-orbits $S$ on ${}^{\vee}R/{}^{\vee}P_f$.

Let us study the conormal bundle $T^*_{\CS_{S,T}}\left({}^\vee R/^\vee P_g\times ^\vee R/^\vee P_f\right)$ in more detail. Choose a point $(g^\vee P_g, h^\vee P_f)\in \CS_{S,T}$. The conormal space of $ ^\vee R/^\vee P_g \times S \subset {}^\vee R/^\vee P_g\times ^\vee R/^\vee P_f$ at the point $(g^\vee P_g, h^\vee P_f)$ is given by 
\begin{equation}\label{eq:subspace1}
\{0\} \oplus(\Ad(h){}^\vee\fu_f)^{-^\vee\theta}\end{equation}
(after identifying the cotangent space $({}^{\vee}\mathfrak{r}/\Ad(g){}^{\vee}\fp_g)^* \oplus ({}^{\vee}\mathfrak{r}/\Ad(h){}^{\vee}\fp_f)^*$ with $\Ad(g)\fu_g \oplus \Ad(h)\fu_f$ using an invariant bilinear form as in Proposition \ref{prop:Arthuroconormal}). On the other hand, the conormal space of $T \subset {}^\vee R/^\vee P_g\times ^\vee R/^\vee P_f$ at $(g^\vee P_g, h^\vee P_f)$ is given by 
\begin{equation}\label{eq:subspace2}\{(x,-x) \mid x\in \Ad(g) ^\vee\fu_g\cap \Ad(h) ^\vee\fu_f \}\end{equation}
Since $\CS_{S,T}=( ^\vee R/^\vee P_g \times S)\cap T$, the conormal space to $\mathcal{S}_{S,T}$ at $(g^\vee P_g, h^\vee P_f)$ is the sum of the subspaces (\ref{eq:subspace1}) and (\ref{eq:subspace2}):
\begin{equation}\label{eq:conormal}\{(x,y) \mid  x\in \Ad(g)^\vee\fu_g\cap \Ad(h) ^\vee\fu_f, \ x+y\in (\Ad(h)^\vee\fu_f)^{-^\vee\theta}\}.\end{equation}
Let $F$ be a complex analytic function on $^\vee R/^\vee P_g$ such that $F|_{^\vee K^\vee P_g}=0$ and $dF(1^\vee P_g)=e$. Recall, Remark \ref{rmk:vanishingcycles}, that there is a natural isomorphism
\begin{equation}\label{eq:iso0}\chi^{mic}_{(1{}^{\vee}P_g,e)}(p_g)_!\iota_!\mathcal{F} \simeq  \left(\Phi_F\circ (p_g)_!\iota_!\mathcal{F}\right)_{1^\vee P_g}.
\end{equation}
where $\Phi_F$ denotes the vanishing cycles functor. Since $p_g$ is proper, \cite[Exercise 8.15]{KashiwaraSchapira} implies that there is an isomorphism of vector spaces
\begin{equation}\label{eq:iso1}\left(\Phi_F\circ (p_g)_!\iota_!\mathcal{F}\right)_{1^\vee P_g}\simeq (p_g|_{\{1^\vee P_g\}\times ^\vee R/^\vee P_f})_!  \Phi_{F\circ p_g}|_{\{1^\vee P_g\}\times ^\vee R/^\vee P_f}(\iota_!\mathcal{F}).
\end{equation}
We claim that $\Phi_{F\circ p_g}|_{\{1^\vee P_g\}\times ^\vee R/^\vee P_f}(\iota_!\CF)$ is a skyscraper sheaf supported at the point $(1^\vee P_g,1^\vee P_f)$. Indeed, suppose the stalk $\Phi_{F\circ p_g}(\iota_!\CF)_{(1{}^{\vee}P_g,h{}^{\vee}P_f)}$ is nonzero for some point $(1^\vee P_g, h^\vee P_f) \in \mathrm{Supp}(\iota_!\CF)$. Then there must be a cotangent vector of the form $(e,0)$ at the point $(1{}^{\vee}P_g,h{}^{\vee}P_f)$, which is also contained in $SS(\iota_!\CF)$. By (\ref{eq:SS}) and (\ref{eq:conormal}) this implies that $e=e+0\in (^\vee\fu_g\cap \Ad(h) ^\vee\fu_f)^{-^\vee \theta}$. Recall, Section \ref{sec:Arthur}, that $e \in ({}^{\vee}\fu_g)^{-{}^{\vee}\theta}$. So this condition is equivalent to $e\in \Ad(h)^\vee\fu_f$. On the other hand, since $(1^\vee P_g, h^\vee P_f) \in \mathrm{Supp}(\iota_!\CF)$ we must have that $h\in\overline{{}^{\vee}P_g{}^{\vee}P_f}$. In summary, $h$ must belong to the set
$$\{h \in \overline{{}^{\vee}P_g{}^{\vee}P_f} \mid e \in \Ad(h) {}^{\vee}\fu_f\} = \{h^{-1} \in \overline{{}^{\vee}P_f{}^{\vee}P_g} \mid \Ad(h^{-1})e \in {}^{\vee}\fu_f\}.$$
But Theorem \ref{thm:threeparabolics} tells us that this set is equal to ${}^{\vee}P_f$. Thus, $(1^\vee P_g, h^\vee P_f) = (1{}^{\vee}P_g,1{}^{\vee}P_f)$, as asserted above.

Since $\Phi_{F\circ p_g}|_{\{1^\vee P_g\}\times ^\vee R/^\vee P_f}(\iota_!\CF)$ is a skyscraper sheaf supported at the point $(1{}^{\vee}P_g,1{}^{\vee}P_f)$, there is an isomorphism of vector spaces
\begin{equation}\label{eq:iso2}(p_g|_{\{1^\vee P_g\}\times ^\vee R/^\vee P_f})_!  \Phi_{F\circ p_g}|_{\{1^\vee P_g\}\times ^\vee R/^\vee P_f}(\iota_!\mathcal{F})\simeq \Phi_{F\circ p_g}(\iota_!\mathcal{F})_{(1^\vee P_g,1^\vee P_f)}.\end{equation}
But $\iota(1(^\vee P_g\cap ^\vee P_f))=(1^\vee P_g,1^\vee P_f)$, so
    \begin{equation}\label{eq:iso3}\Phi_{F\circ p_g}(\iota_!\mathcal{F})_{(1^\vee P_g,1^\vee P_f)}\simeq  \left(\Phi_{F\circ p_g\circ \iota}\CF\right)_{1(^\vee P_g\cap ^\vee P_f)}.\end{equation}
    Finally, using the compatibility of microlocalization with smooth pullbacks (\cite[Proposition 4.3.5]{KashiwaraSchapira}) we have an isomorphism
    \begin{equation}\label{eq:iso4}\left(\Phi_{F\circ p_g\circ \iota}\CF\right)_{1(^\vee P_g\cap ^\vee P_f)}\simeq \chi^{mic}_{(1^\vee P_f,e)}(\mathcal{G}).\end{equation}
Now (\ref{eq:iso}) follows by combining the isomorphisms (\ref{eq:iso0}),(\ref{eq:iso1}), (\ref{eq:iso2}), (\ref{eq:iso3}), and (\ref{eq:iso4}).

\end{proof}
\end{proof}

Finally, we address the question of existence of $\theta$-stable induction data.

\begin{lemma}\label{lem:thetastableinductiondataexist}
Let
$$I=(M^{\Gamma},{}^{\vee}M^{\Gamma},Q)$$
be a weak induction datum for the pair $(G^{\Gamma},{}^{\vee}G^{\Gamma})$ such that $Q \subset G$ is $\theta$-stable (Lemma \ref{lem:realisreal}). Then there is an extended group structure $\mathcal{W}_0$ on $M^{\Gamma}$ and an $E$-group structure $\mathcal{S}_0$ on ${}^{\vee}M^{\Gamma}$ such that 
$$\mathcal{I}=((M^{\Gamma},\mathcal{W}_0),({}^{\vee}M^{\Gamma},\mathcal{S}_0),Q)$$
is a $\theta$-stable induction datum for the pair $((G^{\Gamma},\mathcal{W}),({}^{\vee}G^{\Gamma},\mathcal{S}))$.
\end{lemma}

\begin{proof}
Let $H^{\Gamma} \subset M^{\Gamma}$ be a fundamental Cartan subgroup. Since $Q$ is $\theta$-stable, $H^{\Gamma}$ is also a fundamental Cartan subgroup of $G^{\Gamma}$. By \cite[Proposition 13.6]{AdamsBarbaschVogan}, there is a based Cartan subgroup $(H^{\Gamma},\mathcal{W}(H^{\Gamma}),\Delta_{i\RR}^+,\Delta_{\RR}^+)$ of $(G^{\Gamma},\mathcal{W})$. Choose any strong real form $\delta' \in \mathcal{W}(H^{\Gamma})$. By definition, $\Ad(\delta')$ is a quasi-split real form of $G$ and there is a Whittaker datum $(N',\chi')$ for $G(\RR,\delta')$  such that $(\delta',N',\chi') \in \mathcal{W}$. Choose (as is possible), a Cartan involution $\theta'$ of $G$, corresponding to $\Ad(\delta')$, preserving $H$. Note that $\theta' = -\Ad(\delta')$ on $\Delta(\fg,\fh)$. In particular, $M$ and $Q$ are $\theta'$-stable subgroups. Choose a large $\theta'$-stable Borel subgroup $B' \subset G$ corresponding to $(N',\chi')$ under the bijection of Proposition \ref{prop:WhittakerAV}. After possibly replacing $B'$ with a $G^{\theta'}$-conjugate, we may assume that $B' \supset H$. Since $Q$ is $\theta'$-stable, we can find a $\theta'$-stable Borel subgroup $B \subset Q$ containing $H$. Choose $w \in W(G,H)$ such that $w\Delta(B') = \Delta(B)$. Since both $B$ and $B'$ are $\theta'$-stable, we have that $w \in W^{\theta'} = W^{\Ad(\delta')}$. Let $\dot{w} \in N_G(H)$ be any lift of $w$, so that $B = \dot{w}B'\dot{w}^{-1}$.  Consider the strong real form $\delta := \dot{w}\delta'\dot{w}^{-1}$ of $G^{\Gamma}$ and its Cartan involution $\theta := \Ad(\dot{w}) \circ \theta' \circ \Ad(\dot{w})^{-1}$. Since $w$ is $\delta'$-fixed, $\Ad(\delta) = \Ad(\delta')$ as involutions of $H$. Hence, $\delta \in H\delta' \subset H^{\Gamma} \subset M^{\Gamma}$. For the same reason, $\theta$ preserves the subgroups $H \subset M \subset G$. Since $B$ is a large $\theta$-stable Borel subgroup of $Q$, $B_0 := B \cap M$ is a large $\theta$-stable Borel subgroup of $M$. Let $(N,\chi)$ (resp. $(N_0,\chi_0)$) be the Whittaker datum for $G(\RR,\delta)$ (resp. $M(\RR,\delta)$), corresponding to $B$ (resp. $B_0$) under the bijection of Proposition \ref{prop:WhittakerAV}. Define
 $$\mathcal{W}_0 := M \cdot (\delta,N_0,\chi_0).$$
 To construct $\mathcal{S}_0$, we choose any Cartan subgroup ${}^{\vee}H^{\Gamma} \subset {}^{\vee}M^{\Gamma}$. By \cite[Proposition 13.8]{AdamsBarbaschVogan}, there is a based Cartan subgroup $({}^{\vee}H^{\Gamma},\mathcal{S}({}^{\vee}H^{\Gamma}),{}^{\vee}\Delta_{i\RR}^+,{}^{\vee}\Delta_{\RR}^+)$. Choose any strong involution ${}^{\vee}\delta \in \mathcal{S}({}^{\vee}H^{\Gamma})$. By definition, there is a large ${}^{\vee}\delta$-stable Borel subgroup ${}^{\vee}B \subset {}^{\vee}G$ such that $({}^{\vee}\delta,{}^{\vee}B) \in \mathcal{S}$. Choose (as is possible) a real form $\sigma$ of ${}^{\vee}G$, corresponding to $\Ad({}^{\vee}\delta)$, and preserving ${}^{\vee}M \subset {}^{\vee}G$. Note that $\sigma$ is quasi-split and ${}^{\vee}Q = {}^{\vee}M{}^{\vee}U$ is $\sigma$-stable. Let $({}^{\vee}N,\chi)$ be a Whittaker datum for ${}^{\vee}G(\RR,\sigma)$ corresponding to ${}^{\vee}B$ under the bijection of Proposition \ref{prop:WhittakerAV}. After replacing $({}^{\vee}N,\chi)$ with a ${}^{\vee}G(\RR,\sigma)$-conjugate, we may assume that ${}^{\vee}M$ is standard with respect to ${}^{\vee}N$. Thus, ${}^{\vee}N_0 := {}^{\vee}N \cap {}^{\vee}M$ is a $\sigma$-stable maximal unipotent subgroup of ${}^{\vee}M$ and $\chi_0 := \chi|_{{}^{\vee}N_0(\RR)}$ is a non-degenerate unitary character. Let ${}^{\vee}B_0 \subset {}^{\vee}M$ be a large ${}^{\vee}\delta$-stable Borel subgroup corresponding to $({}^{\vee}N_0,\chi_0)$ under the bijection of Proposition \ref{prop:WhittakerAV}, and let
 $$\mathcal{S}_0 := {}^{\vee}M \cdot ({}^{\vee}\delta,{}^{\vee}B_0).$$
 By construction, the induction datum $\mathcal{I}=((G^{\Gamma},\mathcal{W}),(M^{\Gamma},\mathcal{W}_0),({}^{\vee}G^{\Gamma},\mathcal{S}),({}^{\vee}M^{\Gamma},\mathcal{S}_0),Q)$ satisfies properties $(a_{\theta})$ and $(b_{\theta})$ of Definition \ref{def:thetastableinductiondatum}, and is therefore $\theta$-stable. 
\end{proof}

\subsection{Decompositions of Arthur parameters}\label{subsec:decompositionsparameters}

Let $(G^{\Gamma},\mathcal{W})$ be an extended group and let $({}^{\vee}G^{\Gamma},\mathcal{S})$ be a corresponding $E$-group with second invariant $z$.  Let $\psi\colon W_{\RR} \times SL_2(\CC) \to {}^{\vee}G^{\Gamma}$ be an Arthur parameter.  In this section, we will introduce the notion of a \emph{decomposition} of $\psi$. One should think of this as
the minimal structure required to realize the associated Arthur packet via real parabolic and (weakly fair range) cohomological induction from a unipotent Arthur packet for a Levi subgroup of $G^{\Gamma}$. Here is the definition.

\begin{definition}\label{def:completeinduction}
Let $\psi: W_{\RR} \times SL_2(\CC) \to {}^{\vee}G^{\Gamma}$ be an Arthur parameter. A \emph{decomposition} of $\psi$ is a sextuple
$$\mathcal{D} = ((M_1^{\Gamma},\mathcal{W}_1),(M_2^{\Gamma},\mathcal{W}_2),({}^{\vee}M_1^{\Gamma},\mathcal{S}_1),({}^{\vee}M_2^{\Gamma},\mathcal{S}_2), Q_1, Q_2)$$
where
\begin{itemize}
    \item $((M_1^{\Gamma},\mathcal{W}_1),({}^{\vee}M_1^{\Gamma},\mathcal{S}_1),Q_1)$ is a real induction datum for the pair $((G^{\Gamma},\mathcal{W}),({}^{\vee}G^{\Gamma},\mathcal{S}))$ (Definition \ref{def:realinductiondatum});
    \item $((M_2^{\Gamma},\mathcal{W}_2),({}^{\vee}M_2^{\Gamma},\mathcal{S}_2),Q_2)$ is a $\theta$-stable induction datum for the pair $((M_1^{\Gamma},\mathcal{W}_1),({}^{\vee}M_1^{\Gamma},\mathcal{S}_1))$ (Definition \ref{def:thetastableinductiondatum}),
\end{itemize}
subject to the conditions
\begin{itemize}
    \item[(i)] The image of $\psi$ is contained in ${}^{\vee}M_2^{\Gamma}$. We write $\psi_1$ (resp. $\psi_2$) for $\psi$, when regarded as an Arthur parameter for ${}^{\vee}M_1^{\Gamma}$ (resp. ${}^{\vee}M_2^{\Gamma}$).
    \item[(ii)] $\psi_2$ is unipotent, i.e. $\psi(\CC^{\times}) \subset Z({}^{\vee}M_2)$.
    \item[(iii)] $\lambda(\varphi_{\psi_2})$ is in the weakly fair range for $\fq_2$ (Definition \ref{def:weaklyfair}(ii)).
\end{itemize}
We say that $\mathcal{D}$ is a \emph{Jordan decomposition} if
\begin{itemize}
    \item[(iv)] ${}^{\vee}M_1 = Z_{{}^{\vee}G}(\psi(\RR_+))$ and ${}^{\vee}M_2 = Z_{{}^{\vee}G}(\psi(\CC^{\times}))$.
\end{itemize}
We say that $\mathcal{D}$ is \emph{distinguished} if
\begin{itemize}
    \item[(v)] $\psi(SL_2(\CC))$ is not contained in a proper Levi subgroup of ${}^{\vee}M_2$. 
\end{itemize}
\end{definition}

It is not immediately obvious that (Jordan, resp. distinguished) decompositions exist. We now show that they do. 

\begin{prop}\label{prop:Jordandecompositionsexist}
Let $\psi: W_{\RR} \times SL_2(\CC) \to {}^{\vee}G^{\Gamma}$ be an Arthur parameter. Then $\psi$ admits a Jordan decomposition.
\end{prop}

\begin{proof}
Define Levi subgroups
$${}^{\vee}M_1^{\Gamma} := Z_{{}^{\vee}G}(\psi(\RR_+)), \qquad {}^{\vee}M_2 := Z_{{}^{\vee}G}(\psi(\CC^{\times})) = Z_{{}^{\vee}M_1}(\psi(S^1)).$$
Since $j$ normalizes the subgroups $\RR_+$ and $\CC^{\times}$ of $W_{\RR}$, $\psi(j)$ normalizes the subgroups ${}^{\vee}M_1$ and ${}^{\vee}M_2$ of ${}^{\vee}G$. Thus, we obtain weak $E$-groups
$${}^{\vee}M_1^{\Gamma} := \langle {}^{\vee}M_1,\psi(j)\rangle, \qquad {}^{\vee}M_2^{\Gamma} := \langle {}^{\vee}M_2,\psi(j)\rangle.$$
Clearly, $\psi$ factors through a unipotent Arthur parameter $\psi_2: W_{\RR} \times SL_2(\CC) \to {}^{\vee}M_2^{\Gamma}$. 

Consider the $\ZZ$-grading of ${}^{\vee}\fm_1$ defined by the co-character $\psi|_{S^1}: S^1 \to {}^{\vee}M_1$
$${}^{\vee}\fm_1 = \bigoplus_n {}^\vee\fm_{1,n}, \qquad {}^{\vee}\fm_{1,n} = \{\xi \in {}^{\vee}\fm_1 \mid \Ad(\psi(z))\xi = z^n\xi, \ z \in S^1\},$$
Define the parabolic subalgebra ${}^{\vee}\fq_2 \subset {}^{\vee}\fm_1$ by
$${}^{\vee}\fq_2 := {}^\vee\fm_2 \oplus {}^{\vee}\fu_2, \qquad {}^{\vee}\fm_2 = {}^{\vee}\fm_{1,0}, \qquad {}^{\vee}\fu_2 := \bigoplus_{n>0} {}^{\vee}\fm_{1,n},$$
and let ${}^{\vee}Q_2 = {}^{\vee}M_2{}^{\vee}U_2$ denote the corresponding parabolic subgroup of ${}^{\vee}M_1$. To define ${}^{\vee}Q_1$, choose a regular element $x \in d\psi(\RR_+)$ and consider the corresponding eigenspace decomposition of ${}^{\vee}\fg$
$${}^{\vee}\fg = \bigoplus_z {}^{\vee}\fg_z, \qquad {}^{\vee}\fg_z = \{\xi \in {}^{\vee}\fg \mid [x,\xi] = z\xi\}.$$
Choose an $\RR$-linear functional $\nu: \CC \to \RR$ such that $\nu(z) \neq 0$ for every nonzero $z$ appearing in this decomposition. Let
$${}^{\vee}\fq_1:={}^{\vee}\fm_1 \oplus {}^{\vee}\fu_1, \qquad {}^{\vee}\fm_1 = {}^{\vee}\fg_0, \qquad {}^{\vee}\fu_1 := \bigoplus_{\nu(z) >0} {}^{\vee}\fg_z,$$
and let ${}^{\vee}Q_1 = {}^{\vee}M_1{}^{\vee}U_1$ denote the corresponding parabolic subgroup of ${}^{\vee}G$. Since $j$ acts by the identity (resp. inversion) on $\RR_+$ (resp. $S^1$), ${}^{\vee}Q_1$ (resp. ${}^{\vee}Q_2$) is preserved (resp. taken to its opposite) by $\Ad(\psi(j))$.

Now, choose a Cartan subgroup ${}^dH^{\Gamma} \subset {}^{\vee}G^{\Gamma}$ containing $\psi(W_{\RR})$. Choose any positive systems $^d\Delta_{\RR}^+$ and $^d\Delta_{i\RR}^+$ of real and imaginary roots. Then by \cite[Proposition 13.8]{AdamsBarbaschVogan}, there is a unique $E$-group structure $\mathcal{S}(^dH^{\Gamma})$ on $^dH^{\Gamma}$ such that
$(^dH^{\Gamma},\mathcal{S}(^dH^{\Gamma}), {}^d\Delta_{i\RR}^+, {}^d\Delta_{\RR}^+)$ is a based Cartan subgroup of $({}^{\vee}G^{\Gamma},\mathcal{S})$. Choose a based Cartan subgroup $(H^{\Gamma},\mathcal{W}(H^{\Gamma}), \Delta_{i\RR}^+,\Delta_{\RR}^+)$ of $(G^{\Gamma},\mathcal{W})$ and an identification $\zeta\colon {}^{\vee}H \to {}^dH$ as in \cite[Definition 13.9]{AdamsBarbaschVogan} (it is possible to do so by \cite[Proposition 13.10]{AdamsBarbaschVogan}). Let
$$M_2 \subset M_1 \subset G$$
be the Levi subgroups whose roots correspond under $\zeta$ to the co-roots of ${}^{\vee}M_2 \subset {}^{\vee}M_1 \subset {}^{\vee}G$, and let
$$Q_1 = M_1 U_1 \subset G, \qquad Q_2 = M_2U_2 \subset M_1$$
be the parabolic subgroups corresponding under $\zeta$ to ${}^{\vee}Q_1$, ${}^{\vee}Q_2$. Choose a strong real form $\delta \in \mathcal{W}(H^{\Gamma})$. Since the action of $\delta$ on the roots of $H$ in $G$ is the transpose of the action of ${}^{\vee}\delta$ on the roots of $^dH$ in ${}^{\vee}G$, it is clear that $M_1$ and $M_2$ are preserved by $\delta$. Thus, we obtain weak extended groups
$$M_1^{\Gamma} = \langle M_1, \delta\rangle, \qquad M_2^{\Gamma} = \langle M_2, \delta \rangle.$$
Moreover, since ${}^{\vee}Q_1$ (resp. ${}^{\vee}Q_2$) is preserved (resp. taken to its opposite) by $\Ad(\psi(j))$, $Q_1$ is real (resp. $\theta$-stable) in the sense of Lemma \ref{lem:realisreal}.

By Lemma \ref{lem:realinductiondataexist}, there is an extended group structure $\mathcal{W}_1$ on $M_1^{\Gamma}$ and an $E$-group structure $\mathcal{S}_1$ on ${}^{\vee}M_1^{\Gamma}$ such that $((M_1^{\Gamma},\mathcal{W}_1),({}^{\vee}M_1^{\Gamma},\mathcal{S}_1),Q_1)$ is a real induction datum for the pair $((G^{\Gamma},\mathcal{W}),({}^{\vee}G^{\Gamma},\mathcal{S}))$. Then by Lemma \ref{lem:thetastableinductiondataexist}, there is an extended group structure $\mathcal{W}_2$ on $M_2^{\Gamma}$ and an $E$-group structure $\mathcal{S}_2$ on ${}^{\vee}M_2^{\Gamma}$ such that $((M_2^{\Gamma},\mathcal{W}_2),({}^{\vee}M_2^{\Gamma},\mathcal{S}_2),Q_2)$ is a $\theta$-stable induction datum for the pair $((M_1^{\Gamma},\mathcal{W}_1),({}^{\vee}M_1^{\Gamma},\mathcal{S}_1))$. It is clear from definitions that the sextuple
$$\mathcal{D} = ((M_1^{\Gamma},\mathcal{W}_1),(M_2^{\Gamma},\mathcal{W}_2),({}^{\vee}M_1^{\Gamma},\mathcal{S}_1),({}^{\vee}M_2^{\Gamma},\mathcal{S}_2), Q_1, Q_2)$$
is a decomposition of $\psi$.
\end{proof}

\begin{prop}\label{prop:distinguisheddecompositionsexist}
Let $\psi\colon W_{\RR} \times SL_2(\CC) \to {}^{\vee}G^{\Gamma}$ be an Arthur parameter. Then $\psi$ admits a distinguished decomposition.
\end{prop}

\begin{proof}
Let
$${}^{\vee}E = Z_{{}^{\vee}G}(\psi(SL_2(\CC))),$$
a (disconnected) reductive subgroup of ${}^{\vee}G$, normalized by $\psi(j)$. Choose (as is possible) a $\psi(j)$-stable maximal torus ${}^{\vee}T \subset {}^{\vee}E$ containing $\psi(\CC^{\times})$. Define Levi subgroup
$${}^{\vee}M_1 := Z_{{}^{\vee}G}({}^{\vee}T^{\psi(j)}), \qquad {}^{\vee}M_2 := Z_{{}^{\vee}G}({}^{\vee}T).$$
Since $\psi(j)$ normalizes the subgroups ${}^{\vee}T, {}^{\vee}T^{\psi(j)} \subset {}^{\vee}G$, it also normalizes their centralizers ${}^{\vee}M_1, {}^{\vee}M_2 \subset {}^{\vee}G$. Thus, we obtain weak $E$-groups
$${}^{\vee}M_1^{\Gamma} := \langle {}^{\vee}M_1,\psi(j)\rangle, \qquad {}^{\vee}M_2^{\Gamma} := \langle {}^{\vee}M_2, \psi(j)\rangle.$$
Choose a parabolic subgroup ${}^{\vee}Q_1 = {}^{\vee}M_1{}^{\vee}U_1\subset {}^{\vee}G$ preserved by $\Ad(\psi(j))$ and a parabolic ${}^{\vee}Q_2 = {}^{\vee}M_2{}^{\vee}U_2 \subset {}^{\vee}M_1$ sent to its opposite by $\Ad(\psi(j))$. We may choose ${}^{\vee}Q_2$ so that $\lambda(\varphi_{\psi})_z \in Z({}^{\vee}\fm_2)$ is (weakly) dominant for ${}^{\vee}\fu_2$. Define the structures $\mathcal{S}_1$, $\mathcal{S}_2$, $(M_1^{\Gamma},\mathcal{W}_1)$, $(M_2^{\Gamma},\mathcal{W}_2)$, $Q_1$, and $Q_2$  exactly as in the proof of Proposition \ref{prop:Jordandecompositionsexist}. It is clear from definitions that the sextuple
$$\mathcal{D} = ((M_1^{\Gamma},\mathcal{W}_1),(M_2^{\Gamma},\mathcal{W}_2),({}^{\vee}M_1^{\Gamma},\mathcal{S}_1),({}^{\vee}M_2^{\Gamma},\mathcal{S}_2), Q_1, Q_2)$$
is a decomposition of $\psi$. Moreover, it is distinguished: if ${}^{\vee}L$ is a Levi subgroup of ${}^{\vee}M_2$ containing $\psi(SL_2(\CC))$, then $Z({}^{\vee}L)$ is a torus in ${}^{\vee}E$ containing ${}^{\vee}T$. Since ${}^{\vee}T$ is maximal, we have $Z({}^{\vee}L)={}^{\vee}T$ and hence ${}^{\vee}L = Z_{{}^{\vee}G}(Z({}^{\vee}L)) = Z_{{}^{\vee}G}({}^{\vee}T) = {}^{\vee}M_2$. 
\end{proof}

\subsection{Decompositions of Arthur packets}

Let $(G^{\Gamma},\mathcal{W})$ be an extended group and let $({}^{\vee}G^{\Gamma},\mathcal{S})$ be a corresponding $E$-group with second invariant $z$. Let $\psi\colon W_{\RR} \times SL_2(\CC) \to {}^{\vee}G^{\Gamma}$ be an Arthur parameter. In this section, we prove that every decomposition of $\psi$ gives a realization of the Arthur packet $\Pi^z(G^{\Gamma})^{mic}_{\psi}$ via real parabolic and (weakly fair range) cohomological induction from a unipotent Arthur packet of a Levi subgroup of $G^{\Gamma}$. The proof follows easily from the results of Sections \ref{subsec:realinductiondata} and \ref{subsec:thetastableinductiondata}.
It is convenient to introduce a bit of extra notation. For a virtual representation $\eta \in K\Pi^z(G^{\Gamma})$, let $[\eta] \subset \Pi^z(G^{\Gamma})$ denote the set of irreducible representations appearing (with nonzero multiplicity) in the formula for $\eta$ as a linear combination of irreducibles. For a set $S \subset K\Pi^z(G^{\Gamma})$, let $[S] = \bigcup_{\eta \in S} [\eta]$.

\begin{theorem}\label{thm:decompositiongeneral}
Let $\psi\colon W_{\RR} \times SL_2(\CC) \to {}^{\vee}G^{\Gamma}$ be an Arthur parameter with decomposition $\mathcal{D}$. Then 

\begin{itemize}
    \item[(i)] There is an identity in $K\Pi^z(G^{\Gamma})$
$$\eta^{mic}_{\psi}(1) =  \kappa R^{\fg}_{\fq_1} R^{\fm_1}_{\fq_2}\eta^{mic}_{\psi_2}(1)$$
    Here, $R^{\fg}_{\fq_1}$ is (normalized) real parabolic induction, $R^{\fm_1}_{\fq_2}$ is (normalized) cohomological induction in the weakly fair range, and $\kappa$ is the sign of Proposition \ref{prop:estarthetastableinduction}(iii).
   \item[(ii)] There is an inclusion of subsets of $\Pi^z(G^{\Gamma})$
   $$\Pi^z(G^{\Gamma})^{mic}_{\psi} \subseteq [R^{\fg}_{\fq_1} R^{\fm_1}_{\fq_2} \Pi^{zz(\rho_{\fu_1})z(\rho_{\fu_2})}(M_2^{\Gamma})^{mic}_{\psi_2}].$$
   \item[(iii)]  The map $R^{\fg}_{\fq_1} R^{\fm_1}_{\fq_2}$ takes unitary irreducibles in $\Pi^{zz(\rho_{\fu_1})z(\rho_{\fu_2})}(M_2^{\Gamma})^{mic}_{\psi_2}$ to (linear combinations of) unitary irreducibles. In particular
   $$\Pi^{zz(\rho_{\fu_1})z(\rho_{\fu_2})}(M_2^{\Gamma})^{mic}_{\psi_2} \text{ is unitary} \implies \Pi^z(G^{\Gamma})^{mic}_{\psi} \text{ is unitary}$$
\end{itemize}
If $\mathcal{D}$ is moreover a Jordan decomposition, then we also have
\begin{itemize}
    \item[(iv)] $\epsilon$ induces an isomorphism of Arthur component groups
    $$A_{\epsilon}^{mic}\colon A^{mic}_{\psi_2} \xrightarrow{\sim} A^{mic}_{\psi}.$$
    \item[(v)] For any $s \in  A^{mic}_{\psi_2}$, there is an identity in $K\Pi^z(G^{\Gamma})$
$$\eta^{mic}_{\psi}(A^{mic}_{\epsilon}(s)) =  \kappa R^{\fg}_{\fq_1}R^{\fm_1}_{\fq_2}\eta^{mic}_{\psi_2}(s)$$
\end{itemize}
\end{theorem}

\begin{proof}
(i) follows from Propositions \ref{prop:realpackets}(ii) and \ref{prop:thetastablepackets}(ii). Since $\eta_{\psi}^{mic}(1)$ is a linear combination of the members of $\Pi^z(G^{\Gamma})^{mic}_{\psi}$, each appearing with nonzero multiplicity, (ii) follows from (i). (iii) follows from (ii) and Corollary \ref{cor:unitaryinduction}. (iv) follows from Propositions \ref{prop:realpacketsJordan}(i) and \ref{prop:thetastablepacketsgood}(i). (v) follows from Propositions \ref{prop:realpacketsJordan}(iii) and \ref{prop:thetastablepackets}(ii). 
\end{proof}

\subsection{Some applications}

We now highlight two applications of Theorem \ref{thm:decompositiongeneral}. The first is a proof of the unitarity of Arthur representations for real reductive groups.

\begin{cor}\label{cor:unitary}
Let $(G^{\Gamma},\mathcal{W})$ be an extended group and let $({}^{\vee}G^{\Gamma},\mathcal{S})$ be a corresponding $E$-group with second invariant $z$. Let $\psi\colon W_{\RR} \times SL_2(\CC) \to {}^{\vee}G^{\Gamma}$ be an Arthur parameter. Then the members of the Arthur packet $\Pi^z(G^{\Gamma})_{\psi}^{mic}$ are unitary.
\end{cor}

\begin{proof}
By Proposition \ref{prop:Jordandecompositionsexist}, $\psi$ admits a Jordan decomposition 
$$\mathcal{D} = ((M_1^{\Gamma},\mathcal{W}_1),(M_2^{\Gamma},\mathcal{W}_2),({}^{\vee}M_1^{\Gamma},\mathcal{S}_1),({}^{\vee}M_2^{\Gamma},\mathcal{S}_2), Q_1, Q_2).$$
Then by Theorem \ref{thm:decompositiongeneral}(iii), we have
$$\Pi^{zz(\rho_{\fu_1})z(\rho_{\fu_2})}(M_2^{\Gamma})^{mic}_{\psi_2} \text{ is unitary} \implies \Pi^z(G^{\Gamma})^{mic}_{\psi} \text{ is unitary}$$
Since $\psi_2$ is a unipotent Arthur parameter, $\Pi^{zz(\rho_{\fu_1})z(\rho_{\fu_2})}(M_2^{\Gamma})^{mic}_{\psi_2}$ is unitary by Theorem \ref{thm:unipotentunitary}. 
\end{proof}

\begin{rmk}\label{rmk:reductiontodistinguished}
Here is another approach to Corollary \ref{cor:unitary}. By Proposition \ref{prop:distinguisheddecompositionsexist}, $\psi$ admits a distinguished decomposition. Thus by Theorem \ref{thm:decompositiongeneral}(iii), we can reduce to the case when $\psi$ is a distinguished (unipotent) Arthur parameter (Definition \ref{def:Arthurparam}(v)). It is known that the members of a distinguished Arthur packet have irreducible associated cycles (in classical types, this follows from \cite{BarbaschMaSunZhu}. In exceptional types, it is verified in \cite{yu2026} with the assistance of \texttt{atlas}). Thus, they are unitary by \cite[Theorem 5.22]{DavisMasonBrownOrbitMethod}. 
\end{rmk}

Our second application of Theorem \ref{thm:decompositiongeneral} is a proof of Jiang's conjecture, which gives an upper bound on the wavefront sets of the members of an Arthur packet. 

\begin{cor}[Jiang's Conjecture for Real Groups, {\cite[Conjecture 4.2]{Jiang}}]\label{cor:Jiang}
Let $(G^{\Gamma},\mathcal{W})$ be an extended group and let $({}^{\vee}G^{\Gamma},\mathcal{S})$ be a corresponding $E$-group with second invariant $z$. Let $\psi\colon W_{\RR} \times SL_2(\CC) \to {}^{\vee}G^{\Gamma}$ be an Arthur parameter and let $\OO$ be the Barbasch-Vogan dual of the nilpotent adjoint ${}^{\vee}G$-orbit corresponding to $\psi|_{SL_2(\CC)}$. Then every representation in the Arthur packet $\Pi^z(G^{\Gamma})_{\psi}^{mic}$ has associated variety contained in $\overline{\OO}$. 
\end{cor}

\begin{proof}
Let $\pi \in \Pi^z(G^{\Gamma})_{\psi}^{mic}$. By Theorem \ref{thm:decompositiongeneral}(ii), there is a representation $\pi_2 \in \Pi^{zz(\rho_{\fu_1})z(\rho_{\fu_2})}(M_2^{\Gamma})^{mic}_{\psi_2}$ such that $\pi$ is a summand of $R^{\fg}_{\fq_1}R^{\fm_1}_{\fq_2}\pi_2$. Let $\OO_2$ denote the Barbasch-Vogan dual of the nilpotent ${}^{\vee}M_2$-orbit corresponding to $\psi_2|_{SL_2(\CC)}$. Then $\OO$ and $\OO_2$ are related by parabolic induction, i.e. $\OO = \Ind^G_{M_2}\OO_2$ (see \cite[Proposition A2]{BarbaschVogan1985}). By \cite[Corollary 27.13]{AdamsBarbaschVogan}, the associated variety of $\pi_2$ is contained in $\overline{\OO}_2$. Thus it suffices to show that
$$M_1 \cdot \mathrm{AV}(R^{\fm_1}_{\fq_2}\pi_2) \subseteq \Ind^{M_1}_{M_2} (M_2 \cdot \mathrm{AV}(\pi_2)), \qquad G \cdot \mathrm{AV}(R^{\fg}_{\fq_1}\pi_1) \subseteq \Ind^{G}_{M_1} (M_1 \cdot \mathrm{AV}(\pi_1))$$
where $\pi_1$ is a constituent of $R^{\fm_1}_{\fq_2}\pi_2$. These inclusions are well-known. The first inclusion (involving cohomological induction) follows from Proposition \ref{prop:AVcohinduced}. A reference for the second inclusion (involving real parabolic induction) is \cite[Theorem 3.5]{BarbaschVoganChars}.
\end{proof}

We conclude by describing a conjectural refinement of Theorem \ref{thm:decompositiongeneral} in the Jordan case. 

\begin{conj}
Let $(G^{\Gamma},\mathcal{W})$ be an extended group and let $({}^{\vee}G^{\Gamma},\mathcal{S})$ be a corresponding $E$-group with second invariant $z$. Let $\psi\colon W_{\RR} \times SL_2(\CC) \to {}^{\vee}G^{\Gamma}$ be an Arthur parameter, and let $\mathcal{D}$ be a Jordan decomposition of $\psi$. Then
$$\Pi^z(G^{\Gamma})^{mic}_{\psi} = [R^{\fg}_{\fq_1} R^{\fm_1}_{\fq_2} \Pi^{zz(\rho_{\fu_1})z(\rho_{\fu_2})}(M_2^{\Gamma})^{mic}_{\psi_2}].$$
\end{conj}

Because of the signs appearing in the formulas for $\eta^{mic}_{\psi_2}(s)$, this does not seem to follow in an obvious way from the character identities in Theorem \ref{thm:decompositiongeneral}(v). However, we have verified that it is true in numerous low-rank examples using the \texttt{atlas} software. 

\appendix

\section{A Lie theory calculation}\label{sec:appendix}

\subsection{Recollections on prehomogeneous vector spaces}\label{subsec:PVS}

Recall that a \emph{prehomogeneous vector space} (PVS) is a pair $(L,V)$ consisting of a complex algebraic group $L$ and a finite-dimensional rational $L$-representation $V$ containing an open dense $L$-orbit $\OO$.

\begin{definition}
A \emph{relative invariant} of $(L,V)$ is a nonzero rational function $f$ on $V$ such that 
$$f(l \cdot v) = \chi(l) f(v), \qquad \forall l \in L, \ v \in V,$$
for some character $\chi\colon L \to \CC^{\times}$.
\end{definition}

We note that a relative invariant is determined by its character, up to multiplication by $\CC^{\times}$ (\cite[Proposition 3]{SatoKimura}). In particular, every relative invariant is a homogeneous function. 

Let $S_1,...,S_n$ denote the codimension one irreducible components of $V \setminus \OO$. Then each $S_i$ is the zero set of an irreducible polynomial $f_i \in \CC[V]$, unique up to nonzero scalar multiplication. We will need the following fundamental fact about relative invariants.

\begin{prop}[{\cite[Proposition 5]{SatoKimura}}]\label{prop:generators}
The polynomials $f_i$ are algebraically independent relative invariants of $(L,V)$. Moreover, every relative invariant $f$ of $(L,V)$ can be written in the form
$$f = c \prod_{i=1}^n f_i^{m_i},$$
for $c \in \CC^{\times}$ and $m_1,...,m_n\in \ZZ$.
\end{prop}

Let us now consider a very special kind of PVS which arises in this paper and in other Lie theory contexts. Let $G$ be a complex connected reductive algebraic group, and let $(e,f,h)$ be an $\mathfrak{sl}_2$-triple in $\fg$. The semisimple operator $\ad(h)$ defines a $\ZZ$-grading of $\fg$
$$\fg = \bigoplus_{k \in \ZZ} \fg_k, \qquad \fg_k = \{X \in \fg \mid [h,X]=kX\}.$$
Let $G_0=Z_G(h)$, a Levi subgroup of $G$ with Lie algebra $\fg_0$. Clearly $\Ad(G_0)$ preserves $\fg_i$, for all $i \in \ZZ$. It is a theorem of Kostant  (\cite{Kostant1959}) that $(G_0,\fg_2)$ is a PVS with open dense orbit $\Ad(G_0)e$. In fact, more is true. 

\begin{prop}[{\cite[Proposition 11.2.9]{Rubenthaler}}]\label{prop:codim1}
The complement $\fg_2 \setminus \Ad(G_0)e$ is of pure codimension one, i.e. it is the union of its codimension one irreducible components.
\end{prop}

Fix a non-degenerate invariant symmetric bilinear form $\langle -,-\rangle$ on $\fg$ with the following property: for any maximal torus $T \subset G$, $\langle -,-\rangle$ is positive definite on the co-character lattice $X_*(T)$. Note that if $\fg$ is semisimple, then the Killing form has this property (and is the unique such form, up to positive real scaling on each simple factor).

\begin{prop}\label{prop:positivityofh}
Let $\chi_1,...,\chi_n$ be the characters of $G_0$ corresponding to the irreducible polynomials $f_1,...,f_n$ defining the components of $\fg_2 \setminus \Ad(G_0)e$. Then there are nonnegative real numbers $c_1,...,c_n$ such that
$$\langle h, x\rangle = \sum_{i=1}^n c_i d\chi_i(x), \qquad x \in \fg_0.$$
\end{prop}


\begin{proof}
We can easily reduce to the case when $\fg$ is a simple Lie algebra and $\langle -,-\rangle$ is the Killing form
$$\langle x,y\rangle = \mathrm{Tr}(\ad(x) \circ \ad(y)).$$
We begin by constructing an explicit polynomial relative invariant $p$ on $\fg_2$ which transforms according to $\langle h,-\rangle$. Choose a basis $\{e_k^1,...,e^{l_k}_k\}$ of $\fg_k$, for each $k \in \ZZ$ (this choice of basis will not matter significantly, but is useful for definitions). For each $k \leq -1$, we define a matrix-valued polynomial $M_k$ on $\fg_2$ by the formula
$$M_k(v)_{ij} = \langle \ad(v)^{-k} e_k^i,e_k^j\rangle, \qquad v \in \fg_2.$$
Let $p_k$ be the polynomial on $\fg_2$ defined by taking the determinant of $M_k$
$$p_k(v) = \det(M_k(v)), \qquad v \in \fg_2,$$
and let
$$p = \prod_{k \leq -1} p_k(v)^{-k}.$$
We claim that $p$ is a relative invariant of $(G_0,\fg_2)$ and that the corresponding character $\tau\colon G_0 \to \CC^{\times}$ satisfies
$$d\tau(x) = \langle h, x\rangle, \qquad x \in \fg_0.$$
For $g \in G$, let $C_k(g)$ denote the matrix of $\Ad(g^{-1})|_{\fg_k}$ with respect to the basis $\{e_k^1,...,e^{l_k}_k\}$. Then for any $k \leq -1$ and $v \in \fg_2$ we have
\begin{align*}
M_k(\Ad(g)v)_{ij} &= \langle \ad(\Ad(g)v)^{-k} e_k^i, e_k^j\rangle\\
                  &= \langle \Ad(g)\ad(v)^{-k}\Ad(g)^{-1}e^i_k,e^j_k\rangle\\
                  &= \langle \ad(v)^{-k} \Ad(g)^{-1}e^i_k, \Ad(g)^{-1}e^j_k\rangle\\
                  &= (C_k(g)^t M_k(v) C_k(g))_{ij}
\end{align*}
Hence, $M_k(\Ad(g)v) = C_k(g)^t M_k(v) C_k(g)$ and so
\begin{align*}
p_k(\Ad(g)v) &= \det(M_k(\Ad(g)v))\\
&= \det(C_k(g))^2 p_k(v)\\
&= \det(\Ad(g^{-2})|_{\fg_k}) p_k(v)
\end{align*}
This shows that $p_k$ is a relative invariant of $(G_0,\fg_2)$, corresponding to the character
$$\tau_k\colon G_0 \to \CC^{\times}, \qquad \tau_k(g) = \det(\Ad(g^{-2}))|_{\fg_k}.$$
It follows that $p$ is a relative invariant, corresponding to the character
$$\tau\colon G_0 \to \CC^{\times}, \qquad \tau(g) = \prod_{k \leq -1} \tau_k(g)^{-k}.$$
Note that for each $k \in \ZZ$, we have the relation $\det(\Ad(g)|_{\fg_k}) = \det(\Ad(g)^{-1}|_{\fg_{-k}})$. Thus,
$$\tau(g) = \prod_k \det(\Ad(g)|_{\fg_k})^k, \qquad g \in G_0.$$
Taking differentials, we get
\begin{align*}
d\tau(X) &= \sum_k k\mathrm{Tr}(\ad(X)|_{\fg_k})\\
         &= \sum_k \mathrm{Tr}(\ad(h) \circ \ad(X)|_{\fg_k})\\
         &= \mathrm{Tr}(\ad(h) \circ \ad(X))\\
         &= \langle h,X\rangle
\end{align*}
Since $p$ is a relative invariant, Proposition \ref{prop:generators} implies that
$$p = c\prod_{i=1}^n f_i^{m_i}$$
for $c \in \CC^{\times}$ and $m_1,...,m_n \in \ZZ$. Hence
$$d\tau = \sum_{i=1}^n m_i d\chi_i.$$
Since $p$ is a polynomial, we have $m_i \geq 0$ for all $i$. This completes the proof.
\end{proof}

\subsection{Three parabolics}\label{sec:threeparabolics}

In this section, we will consider the following situation. Let $G$ be a complex connected reductive algebraic group with Borel subgroup $B \subset G$ and maximal torus $T \subset B$. Let $Q = LU$ be a standard parabolic subgroup of $G$ (in our application, $Q$ will arise as the parabolic dual to our `inducing' parabolic). Denote the roots for $T$ in $\fg$ and $\fl$ by $\Delta(\fg)$, $\Delta(\fl)$, and the positive roots by $\Delta^+(\fg)$, $\Delta^+(\fl)$. Let $(e,f,h)$ be an $\mathfrak{sl}_2$-triple in $\fl$ with $h \in \ft$, dominant for $\Delta^+(\fl)$. We will also assume that $(e,f,h)$ is `even' in $\fl$, in that
\begin{itemize}
    \item[(i)] $h$ has even weights on $\fl$
\end{itemize}
Choose $z_f,z_g \in X_*(T) \otimes_{\ZZ} \RR$ such that
$$\alpha(z_f)=\alpha(z_g)=0, \qquad \forall \alpha \in \Delta(\fl)$$
(the subscripts `f' and `g' stand for `fair' and `good', respectively), and consider the elements
$$\lambda_f := \frac{1}{2}h + z_f, \qquad \lambda_g = \frac{1}{2}h + z_g$$
(in our application, these elements will arise as the infinitesimal characters of Arthur packets with a common $SL_2(\CC)$). We will impose the following additional conditions on $h$, $z_f$, and $z_g$:

\begin{itemize}
    \item[(ii)] $\lambda_f$ and $\lambda_g$ have integral weights on $\fg$. 
    \item[(iii)] $z_f$ has nonnegative weights on $\fu$ (this is the `weakly fair range' assumption).
    \item[(iv)] $\lambda_g$ has strictly positive weights on $\fu$ (this is the `good range' assumption).
\end{itemize}

Because of condition (ii), the semisimple operators $\ad(\lambda_f)$ and $\ad(\lambda_g)$ define $\ZZ$-gradings of $\fg$, and hence parabolic subgroups
$$P_f = L_fU_f, \qquad P_g = L_gU_g.$$
Here, $L_f$ (resp.~$L_g$) is the centralizer of $\lambda_f$ (resp.~$\lambda_g$), and $U_f$ (resp.~$U_g$) is the unipotent subgroup corresponding to the strictly positive (integral) weight spaces for $\lambda_f$ (resp.~$\lambda_g$). Because of condition (iv), we have inclusions
\begin{equation}\label{eq:PginQ}L_g \subseteq L, \qquad U \subseteq U_g, \qquad P_g \subseteq Q.\end{equation}
Note that in general, there are no simple relations between $P_f$ and $Q$, or between $P_f$ and $P_g$, except that
$$P_f \cap L = P_g \cap L$$
(since $\lambda_f-\lambda_g$ is central in $\fl$). 

Consider the constructible subset $P_fP_g \subset G$ and its Zariski closure $\overline{P_fP_g}$. The main result of this section is

\begin{theorem}\label{thm:threeparabolics}
In the setting described above, we have
$$\{g \in \overline{P_fP_g} \mid \Ad(g)e \in \fu_f\} = P_f.$$
\end{theorem}

The proof of Theorem \ref{thm:threeparabolics} will occupy the remainder of this section. First, we will prove the corresponding statement for the `open' Bruhat cell $P_fP_g \subset \overline{P_fP_g}$.

\begin{prop}\label{prop:opencell}
In the setting described above, we have
$$\{g \in P_fP_g \mid \Ad(g)e \in \fu_f\} = P_f.$$
\end{prop}

\begin{proof}
The proof has several steps.

{\it Step 1.} Note that $[\lambda_f,e] = [\frac{1}{2}h,e] + [z_f,e] = e$. So $e \in \fu_f$. The inclusion $P_f \subseteq \{g \in P_fP_g \mid \Ad(g)e \in \fu_f\}$ follows immediately. It remains to prove the reverse inclusion $\{g \in P_fP_g \mid \Ad(g)e \in \fu_f\} \subseteq P_f$. Since $\{g \in P_fP_g \mid \Ad(g)e \in \fu_f\}$ is stable under left-multiplication by $P_f$, it suffices to show that $\{g \in P_g \mid \Ad(g)e \in \fu_f\} \subseteq P_f$.

{\it Step 2.} By the inclusions in (\ref{eq:PginQ}), there is a decomposition
\begin{equation}\label{eq:decomp1}P_g = (L \cap P_g) U.\end{equation}
Since $U \subset G$ is unipotent subgroup of $G$, it is a product of its root groups (multiplied in any order). Hence
\begin{equation}\label{eq:decomp2}U = \prod_{\alpha \in \Delta(\fu)} U_{\alpha} = (\prod_{\alpha \in \Delta(\fu \cap \fp_f)} U_{\alpha}) (\prod_{\alpha \in \Delta(\fu \cap \fu_f^{op})} U_{\alpha}) = (U \cap P_f)(U \cap U_f^{op}).\end{equation}
Combining (\ref{eq:decomp1}) and (\ref{eq:decomp2}), we arrive at the decomposition
\begin{equation}\label{eq:decomp3}P_g = (L \cap P_g)(U \cap P_f)(U \cap U_f^{op}).\end{equation}
Since $\{g \in P_g \mid \Ad(g)e \in \fu_f\}$ is stable under left-multiplication by $P_g \cap P_f$ and $(L \cap P_g)(U_ \cap P_f) \subset P_g \cap P_f$, it in fact suffices to show that
$$\{g \in U \cap U_f^{op} \mid \Ad(g)e \in \fu_f\} = \{1\}.$$
{\it Step 3.} Suppose, for contradiction, that $g$ is a nontrivial element in $U \cap U_f^{op}$ such that $\Ad(g)e \in \fu_f$. Since $U \cap U_f^{op}$ is a unipotent group, the exponential map $\exp\colon \fu \cap \fu_f^{op} \to U \cap U_f^{op}$ is a bijection. In particular, $g = \exp(Y)$ for some nonzero element $Y \in \fu \cap \fu_f^{op}$. 

Let $t$ denote the sum of the positive co-roots in $\fu$, so that  $\fl = \fz_{\fg}(t)$ and $\fu$ is the sum of the positive weight spaces for $t$. Consider the $\ZZ^3$-grading of $\fg$ defined by the (commuting) semisimple operators $\ad(\lambda_f)$, $\ad(h)$, and $\ad(t)$:
\begin{equation}\label{eq:trigrading}
\fg = \bigoplus_{k,l,m} \fg_{k,l,m}, \qquad  \fg_{k,l,m} = \{\xi \in \fg \mid [\lambda_f,\xi] = k\xi, \ [h,\xi] = l\xi, \ [t,\xi] = m\xi\}.
\end{equation}
Consider the decomposition of $Y$ with respect to this grading:
$$Y = \sum Y_{k,l,m}, \qquad Y_{k,l,m} \in \fg_{k,l,m}.$$
Since $Y \in \fu \cap \fu_f^{op}$, we have that $k \leq -1$ and $m \geq 1$ whenever $Y_{k,l,m} \neq 0$. Let 
$$k_0 = \max \{k \in \ZZ \mid \fg_{k,l,m} \neq \{0\} \text{ for some } l,m\}.$$
Since $k_0+1 \leq 0$ and $\Ad(g)e \in \fu_f$, we have that
\begin{equation}\label{eq:zerocomponent}(\Ad(g)e)_{k_0+1,l,m} = 0, \qquad \forall l,m \in \ZZ.\end{equation}
On the other hand
$$\Ad(g)e = \sum_{n \geq 0} \frac{1}{n!}\ad(Y)^ne = \sum_{k_1,...,k_n,l_1,...,l_n,m_1,...,m_n} \frac{1}{n!}\ad(Y_{k_1,l_1,m_1})...\ad(Y_{k_n,l_n,m_n})e$$
(this sum is finite since $\ad(Y)$ is nilpotent). Since $e \in \fg_{1,2,0}$, we have that
$$\ad(Y_{k_1,l_1,m_1})...\ad(Y_{k_n,l_n,m_n})e \in \fg_{1+\sum k_i, 2+\sum l_i, \sum m_i}.$$
So 
\begin{equation}\label{eq:component}(\Ad(g)e)_{k_0+1,l,m} = \sum_{l,m \in \ZZ} [Y_{k_0,l,m},e].\end{equation}
Combining (\ref{eq:zerocomponent}) and (\ref{eq:component}), we deduce 
$$[Y_{k_0,l,m},e] = 0, \qquad \forall l,m \in \ZZ.$$
Choose $l,m \in \ZZ$ so that $Y_{k_0,l,m} \neq 0$ (such $l$ and $m$ exist by the definition of $k_0$). Since $[Y_{k_0,l,m},e]=0$ and $[h,Y_{k_0,l,m}] = lY_{k_0,l,m}$, we have by the representation theory of $\mathfrak{sl}_2(\CC)$ that $l \geq 0$. On the other hand, $Y_{k_0,l,m}$ is a weight vector for $z_0$ of weight $k_0 - \frac{1}{2}l$:
$$[z_0,Y_{k_0,l,m}] = [\lambda_f - \frac{1}{2}h, Y_{k_0,l,m}] = k_0Y_{k_0,l,m} - \frac{1}{2}lY_{k_0,l,m} = (k_0 - \frac{1}{2}l)Y_{k_0,l,m}.$$
Thus, since $Y_{k_0,l,m} \in \fu$, our assumption on $z_0$ implies that $k_0-\frac{1}{2}l \geq 0$. Hence $k_0 \geq \frac{1}{2}l \geq 0$. This is a contradiction.
\end{proof}

Next, we will show that there are no contributions to $\{g \in \overline{P_fP_g} \mid \Ad(g)e \in \fu_f\}$ coming from the boundary $\partial P_fP_g = \overline{P_fP_g} \setminus P_fP_g$. Let $W=N_G(T)/T$ denote the Weyl group of $G$. Let $\ell\colon W \to \ZZ_{\geq 0}$ denote the length function and let $\leq$ denote the Bruhat order. The following lemma gives an explicit description of $\partial P_fP_g$ as a union of Bruhat cells.

\begin{lemma}\label{lem:Bruhat}
Let $v \in W$ denote the (unique) minimal-length element such that $v\lambda_f$ is dominant, and let
$$W_{\partial} := \{w \in W \mid vw < v\}.$$
Then
$$\partial P_fP_g = \bigcup_{w \in W_{\partial}} P_f \dot{w}P_g,$$
where, for each $w \in W$, $\dot{w}$ denotes a lift of $w$ to $N_G(T)$. 
\end{lemma}

\begin{proof}
Since $v\lambda_f$ is dominant, the parabolic subgroup $P_f':=\dot{w}P_f\dot{w}^{-1}$ contains the Borel subgroup $B$. So by the Bruhat decomposition, we have, for elements $w_1,w_2 \in W$
$$P_f'\dot{w}_1P_g \nsubseteq \overline{P_f'\dot{w}_2P_g} \iff w_1^{min} < w_2^{min},$$
where $w^{min}$ denotes the (unique) minimal-length element in the double-coset $W_f'wW_g$. Hence
\begin{align*}
P_f\dot{w}P_g \subseteq \partial P_fP_g &\iff  \dot{v}P_f\dot{w}P_g \subseteq \partial (\dot{v} P_fP_g)\\
&\iff P_f'\dot{vw}P_g \subseteq \partial(P_f'\dot{v}P_g)\\
&\iff vw < v^{min}=v.
\end{align*}
This completes the proof.
\end{proof}

For the next proposition, we will use the following notation. For any $x \in \fg$, there is a unique decomposition
$$x = t+\sum_{\alpha \in \Delta(\fg)} x_{\alpha}, \qquad t \in \ft, \ x_{\alpha} \in \fg_{\alpha}.$$
We write
$$\mathrm{Supp}(x) := \{\alpha \in \Delta(\fg) \mid x_{\alpha} \neq 0\}.$$

\begin{prop}\label{prop:boundary}
In the setting described in the beginning of Section \ref{sec:threeparabolics}, we have
$$\{g \in \partial P_fP_g \mid \Ad(g)e \in \fu_f\} = \emptyset.$$
\end{prop}

\begin{proof}
The proof has several steps.

{\it Step 1.} By Lemma \ref{lem:Bruhat}, it suffices to show that
$$w \in W_{\partial} \implies \Ad(P_f\dot{w}P_g)e \cap \fu_f = \emptyset.$$
Since $\fu_f$ is preserved by $\Ad(P_f)$, this is equivalent to the claim that 
\begin{equation}\label{eq:e1}w \in W_{\partial} \implies \Ad(\dot{w}P_g)e \cap \fu_f = \emptyset.\end{equation}
Suppose $\Ad(\dot{w}L_g)e \cap \fu_f = \emptyset$ for some $w \in W$ and let $p \in P_g$ be arbitrary. Write $p=ul$ for $u \in U_g$ and $l \in L_g$. Let $e' = \Ad(l)e$, and $u=\exp(Y)$ with $Y \in \fu_g$. Then
$$
\Ad(\dot{w}p)e = \Ad(\dot{w})\Ad(u)e'= \Ad(\dot{w})e' + \Ad(\dot{w}) \sum_{n \geq 1} \frac{1}{n!}\ad(Y)^n e'$$
Since $l$ centralizes $\lambda_g$ and $e$ has $\lambda_g$-weight $1$, $e'$ must also have $\lambda_g$-weight $1$. On the other hand, since $Y \in \fu_g$, each term $ \frac{1}{n!}\ad(Y)^n e'$ is a sum of weight vectors for $\lambda_g$ of weight strictly greater than $1$. Thus, $\mathrm{Supp}(e')$ is disjoint from $\mathrm{Supp}(\sum_{n \geq 1} \frac{1}{n!}\ad(Y)^n e')$. Since $\Ad(\dot{w})$ permutes the root spaces, this implies that $\mathrm{Supp}(\Ad(\dot{w}e')$ is disjoint from $\mathrm{Supp}(Ad(\dot{w}) \sum_{n \geq 1} \frac{1}{n!}\ad(Y)^n e')$. It follows that
$$\mathrm{Supp}(\Ad(\dot{w}p)e) =  \mathrm{Supp}(\Ad(\dot{w})e') \sqcup \mathrm{Supp}(Ad(\dot{w}) \sum_{n \geq 1} \frac{1}{n!}\ad(Y)^n e').$$
Since $\Ad(\dot{w}e') \cap \fu_f = \emptyset$, we know that $ \mathrm{Supp}(\Ad(\dot{w}e'))$ contains a root $\alpha$ such that $\alpha(\lambda_f) \leq 0$. It follows from the above that $\mathrm{Supp}(\Ad(\dot{w}p)e)$ must also contain this root and hence that $\Ad(\dot{w}p)e \notin \fu_f$. So in fact (\ref{eq:e1}) is equivalent to the implication
\begin{equation}\label{eq:e2}w \in W_{\partial} \implies \Ad(\dot{w}L_g)e \cap \fu_f = \emptyset.\end{equation}
This is what we will prove in Steps 2-4.

{\it Step 2.} Recall that if $\lambda \in \ft$ is integral and dominant for $\Delta^+(\fg)$ and $w_1 < w_2$, then $w_1\lambda-w_2\lambda$ is a sum of positive co-roots for $\fg$. Applying this fact to the Weyl group elements $(vw)^{-1} < v^{-1}$ and the integral dominant coweight $v\lambda_f \in \ft$, we see that $\gamma:=w^{-1}\lambda_f - \lambda_f$ is a sum of positive co-roots for $\fg$ (a nonzero sum, as $v$ is of minimal-length and $vw < v$). Fix a positive-definite $W$-invariant symmetric bilinear form $\langle -, -\rangle$ on $X_*(T) \otimes_{\ZZ} \RR$. Condition (iii) on $z_f$ implies that $z_f$ is dominant for $\fg$. So $\langle z_f, \gamma\rangle \geq 0$. This implies that $\langle h, \gamma\rangle < 0$ by the following calculation:
\begin{align*}
\langle h, \gamma\rangle &= 2\langle \lambda_f,\gamma\rangle - 2\langle z_f,\gamma\rangle\\
&= (||\lambda_f+\gamma||^2 - ||\lambda_f||^2 - ||\gamma||^2) - 2\langle z_f, \gamma\rangle\\
&= (||w^{-1}\lambda_f||^2 - ||\lambda_f||^2 - ||\gamma||^2) - 2\langle z_f, \gamma\rangle\\
&= - ||\gamma||^2 - 2\langle z_f, \gamma\rangle\\
&< 0.
\end{align*}
{\it Step 3.} Now let $V$ denote the $2$-eigenspace of $\ad(h)$ on $\fl$. Recall from Section \ref{subsec:PVS} that $(L_g,V)$ is a prehomogeneous vector space with open dense orbit $\Ad(L_g)e$. Moreover, by Proposition \ref{prop:codim1}, $V \setminus \Ad(L_g)e$ is the union of its codimension one irreducible components $S_1,...,S_n$. Let $f_1,...,f_n$ be irreducible polynomials defining these components. These are relative invariants by Proposition \ref{prop:generators}. Denote the corresponding characters of $L_g$ by $\chi_1,...,\chi_n$. By Proposition \ref{prop:positivityofh}, there are nonnegative real numbers $c_1,...,c_n$ such that
$$\langle h, \gamma \rangle = \sum_{i=1}^n c_i d\chi_i(\gamma).$$
Since $\langle h, \gamma\rangle <0$, we must therefore have $d\chi_i(\gamma)<0$ for some $\chi_i$. Substituting $w^{-1}\lambda_f - \lambda_f$ for $\gamma$ and $\frac{1}{2}h+z_f$ for $\lambda_f$, we get the following inequality:
\begin{equation}\label{eq:e3}d\chi_i(w^{-1}\lambda_f) < d\chi_i(\lambda_f) = \frac{1}{2} d\chi_i(h) + d\chi_i(z_f) = \frac{1}{2}d\chi_i(h).\end{equation}

{\it Step 4.} Let $\Delta(V) = \{\alpha \in \Delta(\fl) \mid \alpha(h)=2\}$, the set of roots in $V$, and choose a basis of nonzero root vectors $\{E_{\alpha}\}_{\alpha \in \Delta(V)}$. Recall that $f_i$ is a nonzero homogeneous polynomial on $V$. Let $d_i$ denote its homogeneous degree. Then $f_i(x)$ is of the following form:
$$f_i(x) = \sum_{|S|=d_i} c_S \prod_{\alpha \in S} x(\alpha), \qquad x = \sum_{\alpha \in \Delta(V)} x(\alpha) E_{\alpha}.$$
where the sum runs over all order-$d_i$ multisets of roots in $\Delta(V) = \{\alpha \in \Delta(\fl) \mid \alpha(h)=2\}$ and the coefficients $c_S$ are complex numbers (not all $0$). Moreover, since $f_i$ is a relative invariant with character $\chi_i$, we have
\begin{equation}\label{eq:e4}d\chi_i = \sum S,\end{equation}
for any order-$d_i$ multiset $S \subset \Delta(V)$ with $c_S \neq 0$. 

Suppose that $x \in \Ad(L_g)e$. Then in particular $f_i(x) \neq 0$. This means that there is an order-$d_i$ multiset $S \subset \Delta(V)$ such that $c_S \neq 0$ and $\prod_{\alpha \in S} x(\alpha) \neq 0$. In particular $S \subseteq \mathrm{Supp}(x)$ and so
\begin{equation}\label{eq:e5}
wS \subseteq w\mathrm{Supp}(x) = \mathrm{Supp}(\Ad(\dot{w})x).
\end{equation}
Evaluating (\ref{eq:e4}) at the element $h \in \ft$, we obtain the equation
$$d\chi_i(h) = (\sum S)(h) = 2d_i.$$
Combining with (\ref{eq:e3}), we get a strict inequality
$$\sum_{\alpha \in S} \alpha(w^{-1}\lambda_f) < d_i.$$
Since the sum on the left consists of $d_i$ terms, there must be some $\alpha \in S$ such that $\alpha(w^{-1}\lambda_f) < 1$. Since $\lambda_f$, and hence $w^{-1}\lambda_f$, has integral weights on $\fg$, this implies that $\alpha(w^{-1}\lambda_f) \leq 0$, i.e. that $w\alpha \notin \Delta(\fu_f)$. But by (\ref{eq:e5}), $w\alpha \in \mathrm{Supp}(\Ad(\dot{w})x)$. We conclude that $\Ad(\dot{w})x \notin \fu_f$. This completes the proof.
\end{proof} 

Now, Theorem \ref{thm:threeparabolics} follows immediately from Propositions \ref{prop:opencell} and \ref{prop:boundary}.

\section{Whittaker data and special positive systems}
\label{s:whittaker}

Let $G$ be a complex connected reductive algebraic group equipped with a real form $\sigma$ and a corresponding Cartan involution $\theta$. Let $G(\RR) = G^{\sigma}$ and $K=G^{\theta}$. 

\begin{prop}[{\cite[Lemma A.5]{AdamsAfgoustidis}}]\label{prop:SvsW}
There is a canonical bijection between the following two sets:
\begin{itemize}
    \item $G(\RR)$-conjugacy classes of pairs $(N,\chi)$ consisting of a $\sigma$-stable maximal unipotent subgroup $N$ and a non-degenerate unitary character $\chi$ of $N(\RR) = N^{\sigma}$.
    \item $K$-conjugacy classes of large $\theta$-stable Borel subgroups $B$ of $G$.
\end{itemize}
This bijection is defined as follows: $(N,\chi)$ corresponds to $B$ if and only if a fundamental series $(\fg,K)$-module corresponding to $B$ admits a Whittaker model of type $(N,\chi)$. 
\end{prop}

\begin{definition}\label{d:bijection}
Suppose $B$ is a large $\theta$-stable Borel subgroup of $G$.
Let $\mathcal N( B,\sigma)$ be the $G(\RR)$-conjugacy class of pairs determined by Proposition \ref{prop:SvsW}.
Conversely given  $(N,\chi)$  let $\mathcal B(N,\chi,\theta)$ be the $K$-conjugacy class of large Borel subgroups given by Proposition \ref{prop:SvsW}.
\end{definition}

Here is an explicit description of these maps.

\begin{prop}
  \label{p:explicit}
Suppose $(G^\Gamma,\mathcal{W})$ is an extended group. Choose
$(\delta,\mathcal P)\in \mathcal{W}$ and write $\mathcal P=(T,B,\{X_\alpha\})$.
Set $\sigma=\Ad(\delta)$. For $\alpha$ simple choose $X_{-\alpha}$ by the condition $[X_\alpha,X_{-\alpha}]=\alpha$. Define positive integer constants $c_\alpha$ by
$$
2\rho=\sum_{\alpha\in \Delta}c_\alpha\alpha.
$$
Set
$$
E=\sum_{\alpha}\sqrt{c_\alpha}X_\alpha, \quad
F=\sum_{\alpha}\sqrt{c_\alpha}X_{-\alpha}
$$
(sums over simple roots) 
and
$$
g=\exp(-\frac{\pi i}4(E+F)).
$$
Let $\gamma$ be the distinguished automorphism defining the inner class. Set $\theta(X_\alpha)=-X_{-\gamma\alpha}$ and $K=G^\theta$.  Then $\theta$ is compatible with $\sigma$,
and 
$\mathcal B(N,\chi,\sigma)$ is the $K$-conjugacy class of $\Ad(g)(B)$. 
\end{prop}

Recall that every finitely-generated $(\fg,K)$-module $X$ has an \emph{associated variety} $\mathrm{AV}(X) \subset (\fg/\fk)^*$ (see \cite{Vogan1991}).

\begin{prop}[{\cite[Lemma A.6]{AdamsAfgoustidis}}]\label{prop:WhittakerAV}
Let $X$ be a finite-length $(\fg,K)$-module. Let $(N,\chi)$ be a Whittaker datum for $G(\RR)$ and let $B \subset G$ be a large $\theta$-stable Borel corresponding to $(N,\chi)$ under the bijection of Proposition \ref{prop:SvsW}.  Then $X$ admits a Whittaker model of type $(N,\chi)$ if and only if
$$(\fg/(\fb+\fk))^* \subset \mathrm{AV}(X).$$
\end{prop}

Now let $Q \subset G$ be a parabolic subgroup with Levi decomposition $Q=MU$. We assume that $M$ is both $\sigma$- and $\theta$-stable. Also fix a finite-order element $z \in Z({}^{\vee}G)^{\theta_Z}$ and recall the order-$2$ element $z(\rho_{\fu}) \in Z({}^{\vee}M)^{\theta_Z}$ (\ref{eq:zrhou}). Parabolic induction is a right exact functor
$$\mathcal{R}^{\fg}_{\fq}: M^{zz(\rho_{\fu})}(\fm,M \cap K) \to M^z(\fg,K)$$
with left-derived functors
$$L_i\mathcal{R}^{\fg}_{\fq}: M^{zz(\rho_{\fu})}(\fm,M \cap K) \to M^z(\fg,K), \qquad i \geq 0.$$
Recall that there are two special cases of particular interest: if $Q$ is $\sigma$-stable, then $\mathcal{R}^{\fg}_{\fq}$ is exact and coincides with (normalized) real parabolic induction. If $Q$ is $\theta$-stable, the functors $L_i\mathcal{R}^{\fg}_{\fq}$ are called `cohomological induction'. In this case, if $X_0 \in M(\fl,L \cap K)$ is in the good range for $\fq$ (Definition \ref{def:weaklyfair}(i)), then $L_i\mathcal{R}^{\fg}_{\fq}X_0 = 0$ for $i \neq \dim(\fu \cap \fk)$; if $X_0$ is irreducible, then so is $L_{\dim(\fu \cap \fk)}\mathcal{R}^{\fg}_{\fq}X_0$.
We will need at various points a formula for the associated variety of a cohomologically induced module. For this, we consider the restriction map
$$p_{(\fq,\theta)}: (\fg/(\fu + \fk))^* \to (\fl/\fl \cap \fk)^*$$

\begin{prop}\label{prop:AVcohinduced}
Assume $Q$ is $\theta$-stable and let $X_0$ be a finite-length $(\fl,L \cap K)$-module. Then for every $i \geq 0$, we have an inclusion
$$\mathrm{AV}(L_i\mathcal{R}^{\fg}_{\fq}X_0) \subseteq K \cdot p_{(\fq,\theta)}^{-1}\mathrm{AV}(X_0).$$
If $X_0$ is in the good range for $Q$, then this inclusion is an equality for $i = \dim(\fu \cap \fk)$.
\end{prop}

\begin{proof}
This follows from the $\mathcal{D}$-modules description of cohomological induction (\cite{HechtMilicicSchmidWolf}), the relation between the associated variety of a $(\fg,K)$-module and the singular support of its localization over the flag variety (\cite{BorhoBrylisnki}), 
and the behavior of singular support under pushforward by a closed embedding (\cite[Lemma 2.3.5]{HTT}). We leave the details to the reader.
\end{proof}

\begin{prop}[\cite{AdamsBarbaschVogan}]\label{prop:Whittakerinduction}
Let $(N_0,\chi_0)$ and $(N,\chi)$ be Whittaker data for $M(\RR)$ and $G(\RR)$, respectively, and let $B_0 \subset M$ and $B \subset G$ be corresponding large $\theta$-stable Borels. Suppose that $X_0 \in M^{zz(\rho_{\fu})}(\fm,M \cap K)$ admits a Whittaker model of type $(N_0,\chi_0)$.
\begin{itemize}
    \item[(i)] Assume that $Q$ is $\sigma$-stable, $N_0 \subset N$, and $\chi_0 =\chi|_{N_0(\RR)}$. Then the parabolically induced module $\mathcal{R}^{\fg}_{\fq}X_0$ admits a Whittaker model of type $(N,\chi)$.
    \item[(ii)] Assume that $B = B_0U$ (this condition implies that $Q$ is $\theta$-stable) and that $X_0$ is in the good range for $Q$ (Definition \ref{def:weaklyfair}). Then the cohomologically induced module $L_{\dim(\fu \cap \fk)}\mathcal{R}^{\fg}_{\fq}X_0$ admits a Whittaker model of type $(N,\chi)$.
\end{itemize}
\end{prop}

\begin{proof}
(i) is \cite[Lemma 4.11]{AdamsBarbaschVogan}. (ii) follows from Propositions \ref{prop:WhittakerAV} and \ref{prop:AVcohinduced}.
\end{proof}

We will also need some facts about `special positive systems', which play a role in \cite{AdamsBarbaschVogan}. Let $H \subset G$ be a $\sigma$- and $\theta$-stable maximal torus.

\begin{definition}[{\cite[Definition 6.29]{AdamsVogan}}]
Let $B \subset G$ be a large $\theta$-stable Borel subgroup and let $\Sigma$ be a positive system for the set of imaginary roots $\Delta_{i\RR}(\fg,\fh)$. We say that $\Sigma$ is \emph{special} with respect to $B$ if there is a $\theta$-stable parabolic subgroup $P \subset G$ such that
\begin{itemize}
    \item[(i)] $L := Z_G(\fh^{\theta})$ is a Levi factor of $P$;
    \item[(ii)] $\Sigma$ is a subset of the roots of $H$ in the unipotent radical of $P$;
    \item[(iii)] $P$ contains a $K$-conjugate of $B$.
\end{itemize}
\end{definition}

We will need the following alternative characterization of special systems.

\begin{prop}[{\cite[Theorem A.10]{AdamsVogan}}]\label{prop:specialAV}
Let $B \subset G$ be a large $\theta$-stable Borel subgroup and let $\Sigma$ be a positive system for $\Delta_{i\RR}(\fg,\fh)$. Choose a $\sigma$-stable parabolic subgroup $P' \subset G$ with Levi factor $L' := Z_G(\fh^{-\theta})$ and let $X_0$ be a relative discrete series $(\fl', L' \cap K)$-module corresponding to $\Sigma$. Then $\Sigma$ is special with respect to $B$ if and only if 
$$(\fg/(\fb+\fk))^* \subset \mathrm{AV}(\mathcal{R}^{\fg}_{\fp'}X_0)$$
\end{prop}

Now, let us assume that $H \subset M$. 

\begin{prop}\label{prop:specialinduction}
\begin{itemize}
    \item[(i)] Let $B_0 \subset M$ and $B \subset G$ be large $\theta$-stable Borel subgroups such that $B=B_0U$ (this condition implies that $Q$ is $\theta$-stable). Let $\Sigma$ be a positive system for $\Delta_{i\RR}(\fm,\fh)$. Assume that $\Sigma$ is special with respect to $B_0$. Then $\Sigma \sqcup \Delta_{i\RR}(\fu)$ is special with respect to $B$.
    \item[(ii)] Let $(N_0,\chi_0)$ and $(N,\chi)$ be Whittaker data for $M(\RR)$ and $G(\RR)$ and let $B_0 \subset M$ and $B \subset G$ be corresponding large $\theta$-stable Borels. Assume that $N_0 \subset N$, $\chi_0 = \chi|_{N_0(\RR)}$, and that $Q$ is $\sigma$-stable. Let $\Sigma$ be a positive system for the set of imaginary roots $\Delta_{i\RR}(\fg,\fh)=\Delta_{i\RR}(\fm,\fh)$. Suppose that $\Sigma$ is special with respect to $B_0$. Then $\Sigma$ is special with respect to $B$.
\end{itemize}
\end{prop}

\begin{proof}
(i) Let $L_0 = Z_M(H^{\theta}) = Z_G(H^{\theta})$. Since $\Sigma$ is special with respect to $B_0$, there is a $\theta$-stable parabolic $P_0$ with Levi factor $L_0$ containing the imaginary roots $\Sigma$ in its unipotent radical such that $P_0$ contains a $K \cap M$-conjugate of $B_0$, say $k(B_0)k^{-1}$ for $k \in K \cap M$. Let $P = P_0U$. Then $P$ is a $\theta$-stable parabolic subgroup of $G$ with Levi factor $L_0 = Z_G(H^{\theta})$ containing the imaginary roots $\Sigma \sqcup \Delta_{i\RR}(\fu)$ in its unipotent radical. Moreover, $P = P_0U$ clearly contains $(kB_0k^{-1})U = k(B_0U)k^{-1} = kBk^{-1}$, a $K$-conjugate of $B$. Thus, $\Sigma \sqcup \Delta_{i\RR}(\fu)$ is special with respect to $B$.

(ii) Choose a regular element $\gamma \in X^*(H)^{-\theta}$. Let $P' \subset G$ and $P'_0 \subset M$ be the corresponding $\sigma$-stable parabolic subgroups. Note that $P'$ and $P'_0$ share a common Levi factor $L' = Z_G(\fh^{-\theta}) = Z_M(\fh^{-\theta})$ and $P' = P'_0U$. Let $X_0$ be a relative discrete series $(\fl', L' \cap K)$-module corresponding to $\Sigma$. Since $\Sigma$ is special with respect to $B_0$, Proposition \ref{prop:specialAV} implies that
$$(\fm/(\fb_0 + \fk \cap \fm))^* \subset \mathrm{AV}(\mathcal{R}^{\fm}_{\fp'_0} X_0)$$
Hence by Proposition \ref{prop:WhittakerAV}, $\mathcal{R}^{\fm}_{\fp'_0} X_0$ admits a Whittaker model of type $(N_0,\chi_0)$. So by Proposition \ref{prop:Whittakerinduction} and transitivity of induction $\mathcal{R}^{\fg}_{\fp'}X_0 = \mathcal{R}^{\fg}_{\fq}\mathcal{R}^{\fm}_{\fp'_0}X_0$ admits a Whittaker model of type $(N,\chi)$. Applying Proposition \ref{prop:WhittakerAV} once again, we deduce that
$$(\fg/(\fb+\fk))^* \subset \mathrm{AV}(\mathcal{R}^{\fg}_{\fp'}X_0).$$
Hence, $\Sigma$ is special with respect to $B$ by Proposition \ref{prop:specialAV}.
 \end{proof}

\begin{sloppypar} \printbibliography[title={References}] \end{sloppypar}

\end{document}